\documentclass[12pt]{article}

\usepackage{fullpage,amsfonts,amsmath,amsthm,amssymb,mathrsfs,graphicx}
\usepackage[colorlinks, bookmarks=true]{hyperref}
\usepackage[usenames,dvipsnames]{color}

\usepackage[margin=1in]{geometry}

\begin{document}

\def\dd{r}
\def\ld{{\widehat L}}
\def\Gcal{\mathcal G}
\def\eps{\epsilon}
\def\eqn{\eqref}
\def\wtau{\widehat \tau}
\def\ttau{\widetilde \tau}
\def\Q{\mathcal Q}
\def\hG{\widehat H}
\def\C{\mathcal C}
\def\ds{{\mathcal D}}
\def\br{br}
\def\M{{\cal M}}
\def\H{{\cal H}}
\def\bell{\ensuremath{\boldsymbol\ell}}
\def\Bin{{\bf Bin}}
\def\hex{{\widehat{\mathbb E}}}
\def\blank{{\Lambda}}
\def\LL{{\mathcal L}}
\def\F{{\mathcal F}}
\def\Po{{\bf Po}}

\newcommand{\be}{\begin{equation}}
\newcommand{\ee}{\end{equation}}
\newcommand{\bea}{\begin{eqnarray}}
\newcommand{\eea}{\end{eqnarray}}
\newcommand{\bean}{\begin{eqnarray*}}
\newcommand{\eean}{\end{eqnarray*}}
\newcommand{\non}{\nonumber}
\newcommand{\no}{\noindent}
\newcommand\floor[1]{{\lfloor #1 \rfloor}}
\newcommand\ceil[1]{{\lceil #1 \rceil}}

\newcommand{\fix}[1]{{#1}}
\newcommand{\remove}[1]{}
\newcommand{\lab}[1]{\label{#1}\ }

\def\a{\alpha}
\def\b{\beta}
\def\c{\chi}
\def\d{\delta}
\def\D{\Delta}
\def\e{\epsilon}
\def\f{\phi}
\def\F{\Phi}
\def\g{\gamma}
\def\G{\Gamma}
\def\k{\kappa}
\def\K{\Kappa}
\def\z{\zeta}
\def\th{\theta}
\def\Th{\Theta}
\def\l{\lambda}
\def\la{\lambda}
\def\La{\Lambda}
\def\m{\mu}
\def\u{\upsilon}
\def\p{\pi}
\def\P{\Pi}
\def\r{\rho}
\def\R{\Rho}
\def\s{\sigma}
\def\S{\Sigma}
\def\t{\tau}
\def\om{\omega}
\def\Om{\Omega}
\def\smallo{{\rm o}}
\def\bigo{{\rm O}}
\def\to{\rightarrow}
\def\E{{\mathbb E}}
\def\ex{{\mathbb E}}
\def\cd{{\cal D}}
\def\rme{{\rm e}}
\def\hf{{1\over2}}
\def\R{{\bf  R}}
\def\cala{{\cal A}}
\def\cale{{\cal E}}
\def\call{{\cal L}}
\def\calb{{\cal B}}
\def\cald{{\cal D}}
\def\calz{{\cal Z}}
\def\calf{{\cal F}}
\def\Fscr{{\cal F}}
\def\cc{{\cal C}}
\def\calc{{\cal C}}
\def\calh{{\cal H}}
\def\calk{{\cal K}}
\def\cals{{\cal S}}
\def\calr{{\cal R}}
\def\calt{{\cal T}}
\def\msq{{\mathscr Q}}
\def\bk{\backslash}

\def\out{{\rm Out}}
\def\temp{{\rm Temp}}
\def\overused{{\rm Overused}}
\def\big{{\rm Big}}
\def\moderate{{\rm Moderate}}
\def\swappable{{\rm Swappable}}
\def\candidate{{\rm Candidate}}
\def\bad{{\rm Bad}}
\def\crit{{\rm Crit}}
\def\col{{\rm Col}}
\def\dist{{\rm dist}}
\def\poly{{\rm poly}}

\def\tdT{{\widetilde\Theta}}

\newcommand{\var}{\mbox{\bf Var}}
\newcommand{\pr}{\mbox{\bf Pr}}

\newtheorem{lemma}{Lemma}
\newtheorem{theorem}[lemma]{Theorem}
\newtheorem{corollary}[lemma]{Corollary}
\newtheorem{claim}[lemma]{Claim}
\newtheorem{remark}[lemma]{Remark}
\newtheorem{proposition}[lemma]{Proposition}
\newtheorem{observation}[lemma]{Observation}
\theoremstyle{definition}
\newtheorem{definition}[lemma]{Definition}

\newcommand{\limninf}{\lim_{n \rightarrow \infty}}
\newcommand{\proofstart}{{\bf Proof\hspace{2em}}}
\newcommand{\tset}{\mbox{$\cal T$}}
\newcommand{\proofend}{\hspace*{\fill}\mbox{$\Box$}}
\newcommand{\bfm}[1]{\mbox{\boldmath $#1$}}
\newcommand{\reals}{\mbox{\bfm{R}}}
\newcommand{\expect}{\mbox{\bf Exp}}
\newcommand{\he}{\hat{\e}}
\newcommand{\card}[1]{\mbox{$|#1|$}}
\newcommand{\rup}[1]{\mbox{$\lceil{ #1}\rceil$}}
\newcommand{\rdn}[1]{\mbox{$\lfloor{ #1}\rfloor$}}
\newcommand{\ov}[1]{\mbox{$\overline{ #1}$}}
\newcommand{\inv}[1]{\frac{1}{#1} }
\newcommand{\imax}{I_{\rm max}}

\newcommand{\whp}{w.h.p.}
\newcommand{\aas}{a.a.s.\ }

\date{\empty}

\title{The stripping process can be slow: part I}

\author{Pu Gao\footnote{Research supported by an NSERC Postdoctoral Fellowship and an NSERC Discovery Grant. This author is currently affiliated with Monash University.}
and Michael Molloy\footnote{Research supported by an NSERC Discovery Grant and Accelerator Supplement.} \\
Department of Computer Science, University of Toronto\\
10 King's College Road, Toronto, ON\\
jane.gao@monash.edu \hspace{5ex} molloy@cs.toronto.edu}

\maketitle

\begin{abstract}

Given an integer $k$, we consider the parallel $k$-stripping process applied to a hypergraph $H$:  removing all vertices with degree less than $k$ in each iteration until reaching the $k$-core of $H$. {Take $H$ as $\H_r(n,m)$: a random $r$-uniform hypergraph on $n$ vertices and $m$ hyperedges with the uniform distribution.} Fixing $k,r\ge 2$ with $(k,r)\neq (2,2)$, it has previously been proved that there is a constant $c_{r,k}$ such that for all $m=cn$ with constant $c\neq c_{r,k}$, with high probability, the parallel $k$-stripping process takes $O(\log n)$ iterations. In this paper we investigate the critical case when $c=c_{r,k}+o(1)$. We show that the number of iterations that the process takes can go up to some power of $n$, as long as $c$ approaches $c_{r,k}$ sufficiently fast. A second result we show involves the depth of a non-$k$-core vertex $v$: the minimum number of steps required to delete $v$ from $\H_r(n,m)$ where in each step one vertex with degree less than $k$ is removed. We will prove lower and upper bounds on the maximum depth over all non-$k$-core vertices.

\end{abstract}

\section{Introduction}

Given a nonnegative integer $k$ and a (hyper)graph $H$, the $k$-core of $H$, denoted by $\C_k(H)$, is the maximum subgraph of $H$ with minimum degree at least $k$. The $k$-core was first studied by Bollob\'as\cite{bbcore} and has since become a major focus in random graph theory. Its many applications include erasure codes\cite{lmss,lmss2}, colouring\cite{mr2}, hashing\cite{hmwc}, and graph orientability\cite{pw,csw,fkp,fern}. We define $\calh_{\dd}(n,m)$ to be a random $r$-uniform hypergraph on vertex set $[n]:=\{1,2,\ldots,n\}$ {and exactly $m$ hyperedges, with the uniform distribution}.  We focus on sparse random hypergraphs with bounded average degree; thus the typical range of focus is $m=\Theta(n)$. 
The threshold for the appearance of a non-empty $k$-core in $\calh_{\dd}(n,cn)$ was first determined by Pittel, Spencer and Wormald~\cite{psw} for $r=2$ and $k\ge 3$, whereas for general $(r,k)\neq (2,2)$, the threshold is determined\cite{mmcore,jhk} to be:
\begin{equation}\lab{krthreshold}
c_{\dd,k}=\inf_{\mu > 0}
 \frac{\mu}{r\left[e^{-\mu}\sum_{i = {k-1}}^{\infty} \mu^i/i!\right]^{\dd-1}}
 \enspace .
 \end{equation}


The $k$-core of a (hyper)graph $H$ can be obtained by repeatedly removing all vertices with degree less than $k$. In the parallel $k$-stripping process, all vertices with degree less than $k$ are removed at once in each iteration until the $k$-core is reached. The number of rounds this process takes is called the {\em $k$-stripping number}, denoted by $s_k(H)$. This number is an important parameter associated with a \fix{(hyper)graph} and its $k$-core. Jiang, Mitzenmacher and Thaler~\cite{jmt} discussed several applications of the parallel stripping process in computer science and the importance of analysing the $k$-stripping number. An upper bound $O(\log n)$ is given in~\cite{amxor} for $s_k(\H_r(n,cn))$ when $c>c_{r,k}+\eps$ and this bound is proved to be tight in~\cite{jmt}. Independently, \cite{ikkm} proved an upper bound of $\poly(\log n)$. On the other hand, it was proved in~\cite{jmt,g2} that $s_k(\H_r(n,cn))=O(\log\log n)$ if $c<c_{r,k}-\eps$ {(in fact, results in~\cite{amxor,jmt,g2} are presented for \fix{$\H_r(n,p)$, $p=c/n^{r-1}$}, but they easily translate to $\H_r(n,m)$ as well)}. However, as we will prove in this paper, the stripping process can take a long time as $c\to c_{r,k}$: the number of rounds needed can go up to some power of $n$, depending on the rate at which $c$ approaches $c_{r,k}$.

We give some intuitive explanation of this phenomenon. Note that $c_{r,k}$ is the emergence threshold of the $k$-core. When $c\to c_{r,k}$, the parallel $k$-stripping process undergoes a bottleneck. At a certain point, the number of vertices removed in each {iteration} (i.e.\ of degree less than $k$) becomes sublinear in $n$. If there is {no}  $k$-core, then the process continues for a long enough time to eventually pass through the bottleneck: the number of vertices removed in each iteration becomes and stays linear in $n$ again and eventually all vertices of the \fix{hypergraph} will be removed, producing an empty $k$-core.  If there is a $k$-core, then the process terminates after only $o(n)$ total vertices are removed during the bottleneck; what remains is the giant $k$-core. In both cases, the number of remaining vertices with degree less than $k$ mimics a random walk whose expected change with each vertex deletion is very close to zero during the bottleneck. Hence  a large number of vertices must be removed for it to either decrease to 0 or increase to linear in $n$. Since few vertices are removed in each iteration, it takes many iterations to remove this large number of vertices.

The analysis for the stripping number when $c\to c_{r,k}$ becomes subtle.  For $c$ inside the critical window \fix{$|c-c_{r,k}|=O(1/\sqrt{n})$}, \fix{it is not certain that $\H_r(n,cn)$ has a non-empty $k$-core; i.e.\ the probability that $\H_r(n,cn)$ contains a non-empty $k$-core is bounded away from 0 and 1 as $n\to\infty$~\cite{JL}}. This brings in certain complications in our analysis of the stripping number, especially for the upper bound. We will prove a lower bound for the stripping number when $c\le c_{r,k}+n^{-1/2+\eps}$ \fix{(note this range contains the critical window)},  whereas the upper bound for $c$ in {this range} will be studied in a subsequent paper. For $c$ above {$c_{r,k}+n^{-1/2+\eps}$}, we will prove both upper and lower bounds that are tight in the asymptotic order. It is interesting to note that several graph parameters/properties related to the $k$-core have a radical change when $c$ approaches to this threshold, including the robustness of the $k$-core~\cite{s}.  

Rather than removing all vertices with degree less than $k$ at once in each iteration, we may remove one vertex at a time. This produces a stripping sequence: a sequence of vertices removed from the initial hypergraph $H$. Vertices in a stripping sequence can appear in an arbitrary order, as long as each vertex has degree less than $k$ at the moment of its removal. {Note that} a stripping sequence does not {necessarily} need to terminate with the $k$-core: \fix{for a stripping sequence $\Psi$, a subsequence of $\Psi$ composed of the first arbitrary number of vertices in $\Psi$ is also a stripping sequence}. Given a non-$k$-core vertex $v$, we are interested in the length of a shortest stripping sequence ending with $v$: roughly speaking, this is the minimum number of vertices required to be deleted before deleting $v$. This number is called the {\em depth} of $v$. A formal definition is given in Section~\ref{sec:results}. The maximum depth of $\H_r(n,m)$ (over all non-$k$-core vertices) is bounded in~\cite{amxor} by $O(\log n)$ when $c$ differs from $c_{r,k}$ by some absolute constant. In this paper, our goal is to bound (both from above and below) this parameter when $c\to c_{r,k}$.

The aforementioned bounds on the depth and stripping number from \cite{amxor,ikkm} were motivated by applications to solution clustering in random XORSAT.  We will provide analogous applications for our present results in a subsequent paper ({a preliminary version containing partial results on their applications to clustering in random XORSAT is available in~\cite{gmarxiv}}).

\section{Main results} \lab{sec:results}


We first give a formal definition of the parallel $k$-stripping process and the $k$-stripping number.

\begin{definition} The {\em parallel $k$-stripping process}, applied to a hypergraph $H$, consists of iteratively removing {\em all} vertices of degree less than $k$ at once along with any hyperedges containing any of those vertices, until no vertices of degree less than $k$ remain; i.e.\ until we are left with the $k$-core of $H$.
We use $S_i$ to denote the vertices that are removed during iteration $i$. We use $\widehat{H}_i$ to denote the hypergraph remaining after $i-1$ iterations, i.e. after removing $S_1,...,S_{i-1}$.
\end{definition}

We will analyze the number of rounds that this process takes:

\begin{definition}  The {\em $k$-stripping number of $H$}, denoted $s_k(H)$, is the number of iterations in the  parallel $k$-stripping process applied to $H$.  We often drop the ``$k$'' and speak of the stripping number and $s(H)$.
\end{definition}

We use the following standard notation. Given a sequence of probability spaces $\Omega_n$, we say a sequence of events $A_n$ occurs {\em asymptotically almost surely} (a.a.s.) if the probability that $A_n$ occurs in $\Omega_n$ tends to $1$ as $n\to\infty$. With two sequences of real numbers $a_n$ and $b_n$, we use $a_n=O(b_n)$ to denote that there is an absolute constant $C$ such that $|a_n|\le C |b_n|$. We write $a_n=o(b_n)$ if  $\lim_{n\to\infty} a_n/b_n=0$. We write $a_n=\Omega(b_n)$ if $b_n=O(a_n)$ {and $a_n, b_n \ge 0$ eventually;} 
$a_n=\Theta(b_n)$ if $a_n=O(b_n)$ and $b_n=O(a_n)$ and $a_n, b_n\ge 0$ eventually. Thus, if $a_n=\Omega(b_n)$ or $a_n=\Theta(b_n)$ then $a_n,b_n$ must both be positive (for large $n$). All asymptotics in this paper refer to $n\to\infty$.

We first present our main result on the stripping number of $\H_r(n,cn)$. 

\begin{theorem}\lab{mt}
Consider constants  $\dd,k\geq2, (\dd,k)\neq (2,2)$ and $0<\d<1/2$. 
\begin{enumerate}
\item[(a)]  If $c=c_{r,k}+n^{-\d}$  then \aas $s(\calh_{\dd}(n,cn))={\Theta}(n^{\d/2}\log n)$.
\item[(b)]   If $|c-c_{r,k}|\le n^{-\d}$  then \aas $s(\calh_{\dd}(n,cn))= \Omega(n^{\d/2})$.
\end{enumerate}
\end{theorem}

So if we roughly define the following three ranges of $c$, the behaviour of the stripping number is as follows:
\begin{itemize} 
\item $c=c_{r,k}+n^{-\d}$, $0<\d<1/2$.  Then the stripping number is ${\Theta}(n^{\d/2}\log n)$.
\item $c=c_{r,k}-n^{-\d}$, $0<\d<1/2$.  Then the stripping number is at least $\Omega(n^{\d/2})$.
\item $c$ is between $c_{r,k}-n^{-1/2+o(1)}$ and $c_{r,k}+n^{-1/2+o(1)}$. Then the stripping number is $\Omega(n^{1/4+o(1)})$.

\end{itemize}

In the first range, the stripping number is specified within a constant factor. In the other two ranges, we do not say anything about upper bounds on the stripping number; those upper bounds will be studied in a subsequent paper.

\noindent {\bf Remark}. With the same proof, $n^{-\d}$ in Theorem~\ref{mt} (and in all relevant lemmas in Sections~\ref{sec:coresize} and~\ref{smt1}) can be replaced by $\xi_n$ (and correspondingly $n^{\d/2}$ is replaced by $\sqrt{1/\xi_n}$) for any $\xi_n=o(1)$ such that $\xi_n\ge n^{-1/2+\e}$ for some constant $\e>0$ (e.g.\ $\xi_n=1/\log n$ or $\xi_n=n^{-1/3}\log n$). We use the less general statement for a cleaner presentation.

We are also interested in the number of vertices that must be deleted in order to remove a particular vertex $v$. I.e.\ what is the smallest number of vertices that must be deleted in order to demonstrate that $v$ is not in the $k$-core?
\begin{definition} A {\em $k$-stripping sequence} is a sequence of vertices that can be deleted from a hypergraph, one-at-a-time, {along with their incident hyperedges} such that at the time of deletion each vertex has degree less than $k$.
For any vertex $v$ not in the $k$-core, the {\em depth} of $v$ is the length of a shortest $k$-stripping sequence ending with $v$.
\end{definition}

While every non-$k$-core vertex has depth $O(\log n)$ for any constant $c>c_{k,r}$, as proved in~\cite{amxor},
one of our main contributions in this paper is to prove that {when $c$ approaches} the $k$-core emergence threshold, the maximum depth can rise to $n^{\Theta(1)}$.

\begin{theorem}\lab{mt2}
{Let $\dd,k\geq2, (\dd,k)\neq (2,2)$ be fixed. There are constants $Z=Z(k,r)$ and $\kappa=\kappa(k,r)$ such that: for any fixed $0<\d<1/2$, if $c=c_{\dd,k}+n^{-\d}$ then \aas
the maximum depth of all non-$k$-core vertices in $\calh_{\dd}(n,cn)$ is between $Zn^{\d/2}$ and $n^{\kappa \d}$.}
\end{theorem}

\no {\bf Remark}.

\no (a) Again {here and in Theorem~\ref{mt3} in Section~\ref{smd} below}, with the same proof, $n^{-\d}$ above (and in all relevant lemmas in Sections~\ref{smd} and~\ref{slsi}) can be replaced by $\xi_n$ for any $\xi_n={o(1/\log^{7}n)}$ and $\xi_n\ge n^{-1/2+\e}$ for some constant $\e>0$. We did not 
try to optimize the power of $\log n$.

\no (b) For the upper bound in this theorem, we will actually prove a stronger statement, because the stronger statement will be used in another paper on solution clustering in random XOR-SAT. See Theorem~\ref{mt3} in Section~\ref{smd}.\smallskip

\no (c) Both Theorems~\ref{mt} and~\ref{mt2} translate to $\H_r(n,p)$ with $p=r!c/n^{r-1}$. {Translations of a.a.s.\ properties from $\H_r(n,m)$ to $\H_r(n,p)$ are usually standard  by conditioning the number of hyperedges $X$ in $\H_r(n,p)$ on its typical values, as long the properties under consideration hold for all $m$ such that $|m-\ex X|=O(\sqrt{\var X})$. Our situation is a little subtle as our bounds on the stripping number depend on how close $c$ is to $c_{r,k}$, and some of our results cover the case that $|(c-c_{r,k})n|$ is smaller than the standard deviation of $X$, e.g.\ if $|c-c_{r,k}|=o(n^{-1/2})$. The reason that the translation holds in this paper is that for $c$ in the range in Theorem~\ref{mt}(b), the lower bound is uniform for all $c$ and the size of the range is of order $n^{1-\delta}$, which is much greater than the deviation of $X$, for any $\d<1/2$. }

The $k$-stripping number clearly bounds the maximum depth from below, and so the lower bound of $\Omega(n^{\d/2})$ from Theorem \ref{mt} implies the lower bound in Theorem~\ref{mt2}. The difficult part of the proof of Theorem~\ref{mt2} is to show the upper bound.

Our requirement that {$c\ge c_{r,k}+n^{-1/2+\eps}$ in Theorem~\ref{mt2}} comes from the similar restriction on the actual appearance of the $k$-core, as shown by Kim\cite{jhk} (see also~\cite{dmcore}; {a more precise statement of Theorem~\ref{tkim}(b) is given in Lemma~\ref{lcoresize2} in Section~\ref{sec:coresize}}):

\begin{theorem}\lab{tkim}\cite{jhk} For  $\dd,k\geq2, (\dd,k)\neq (2,2)$ and for any constant $\e>0$.
\begin{enumerate}
\item[(a)] For $c\le c_{\dd,k}-n^{-1/2+\e}$, \aas the $k$-core  of $\calh_{\dd}(n,cn)$ is empty.
\item[(b)] For $c\ge c_{\dd,k}+n^{-1/2+\e}$, \aas $\calh_{\dd}(n,cn)$ has a $k$-core with $\Theta(n)$ vertices.
\end{enumerate}
\end{theorem}

The most challenging difference between the setting of this paper and that of~\cite{amxor} is:
when $c=c_{r,k}+\Theta(1)$, the number of vertices with degree less than $k$ remaining after each round of the parallel stripping process drops geometrically (see the comment following Lemma~\ref{llt1}).
That easily implies that \aas the stripping number is $O(\log n)$.  Furthermore, the property which implies that this number drops geometrically is also key in the analysis of~\cite{amxor}  proving that the depth is $O(\log n)$.  For $c=c_{r,k}+n^{-\d}$, the number of such vertices drops much more slowly, leading to an increase in the stripping number and hence in the maximum depth, which is easily seen to be bounded from below by the stripping number.  This requires us to use a very different, and much more intricate, analysis.

One novelty in our analysis of the depth is as follows: The most straightforward approach to analyze the stripping process is to repeatedly expose the vertices and edges removed in each iteration.  Instead, we only expose the vertex sets $S_1, S_2,...$.  The advantage is that, when considering the depth of a vertex $v$ in a particular level $S_i$, we can treat the edges removed in previous iterations as random. Thus even though we have exposed the fact that $v$ is in $S_i$, there is still some randomness that we can make use of in the specific sequence of vertices whose deletions led to $v$ being deleted in iteration $i$.  Of course, exposing the vertex sets that are deleted in each step exposes something about the edges that are deleted, so we have to be careful about the conditional distribution of those edges. The details of how we do this are in Section~\ref{sec:exposure}.

We excluded the case $(r,k)=(2,2)$ from our main results as this case has already been extensively studied.
 The 2-core of the random graph $G(n,cn)$ is well-studied\cite{tlcomp,tlcomp2,rw,DKLP,dklp2,frmix}. For $0<\d<\inv{3}$ and $p=1+n^{-\d}$ a.a.s.\ the largest component consists of a large 2-core $C_2$, with a Poisson Galton-Watson tree of branching parameter $1-n^{-\d}$ rooted at each vertex.  In addition, there are $\Theta(n)$ smaller components distributed essentially like Poisson Galton-Watson trees of branching parameter $1-n^{-\d}$, except that some of them contain a single cycle. We can define the  stripping number to be the stripping number of the largest component, or to be the maximum of the stripping number of all components. Either way, it is of the order of the height of the tallest of those trees which is $n^{\d}\poly(\log n)$.  Similarly, the maximum depth is of the order of the size of the largest tree, which is  $n^{2\d}\poly(\log n)$.


We will prove Theorem~\ref{mt}(b), and the upper bound for Theorem~\ref{mt}(a) in Section~\ref{smt1}, along with a weaker version of the  lower bound for Theorem~\ref{mt}(a). The lower bound for Theorem~\ref{mt}(a) is proved in Section~\ref{slsi}.  We prove Theorem~\ref{mt2} in Section~\ref{smd}. In Section~\ref{sec:coresize} we discuss the size of the $k$-core as a preparation for the analysis in Sections~\ref{smt1}~--~\ref{slsi}. A key lemma used in the proof of Theorems~\ref{mt}(a) and~\ref{mt2} will be presented in Section~\ref{slsi}.

\section{Size of the $k$-core}
\lab{sec:coresize}

We give a more precise version of Theorem~\ref{tkim}(b), where the size of the $k$-core is specified, when $c\ge c_{r,k}+n^{-1/2+\e}$. We start by defining:
\begin{eqnarray}
f_t(\mu)&=&e^{-\mu}\sum_{i\geq t}\frac{\mu^i}{i!} \label{e.fk}\\
h(\mu)=h_{\dd,k}(\mu)&=&\frac{\mu}{f_{k{-1}}(\mu)^{\dd-1}}. \label{e.hrk}
\end{eqnarray}

Note that $f_t(\mu)$ is the probability that a Poisson variable with mean $\mu$, \fix{denoted by $\Po(\mu)$,} is at least $t$.
Thus {by~\eqn{krthreshold},} 
\[c_{\dd,k}=\inf_{\mu>0} \frac{h(\mu)}{{r}}.\] 
Now for any $\dd,k\geq2, (\dd,k)\neq (2,2)$, we define
$\mu_{\dd,k}$ to be the value of $\mu$ that minimizes $h(\mu)$; i.e.
\begin{equation}
\mu_{\dd,k} \mbox{ is the unique solution to }
h(\mu_{\dd,k}) = {r}c_{\dd,k}  \lab{murk}
\end{equation}
Define
\begin{eqnarray}
\a=\a_{\dd,k}=f_k(\mu_{\dd,k}),\quad \b=\b_{\dd,k}=\frac{1}{\dd}\mu_{\dd,k} f_{k-1}(\mu_{\dd,k}). \lab{alpha-beta}
\end{eqnarray}

 For ease of notation, we drop most of the $\dd,k$ subscripts.
For any $c\ge c_{r,k}$, we define $\mu(c)$ to be the larger solution to
\[c=\frac{h(\mu)}{{r}}.\]
{Then, $\mu_{\dd,k}=\mu(c_{\dd,k})$.}
Define
\begin{eqnarray}
\a(c)&=&f_k(\mu(c)),\quad \b(c)=\frac{1}{r}\mu(c) f_{k-1}(\mu(c)).\lab{alpha2}
\end{eqnarray}

 Theorem 1.7 of \cite{jhk} yields the size of the $k$-core as follows {(Theorem 1.7 of \cite{jhk} is for $\H_r(n,p)$ but it easily translates to $\H_r(n,m)$ by a standard coupling argument for $\H_r(n,p)$ and $\H_r(n,m)$, and by the fact that having a non-empty $k$-core is an increasing graph property)}:
\begin{lemma}\lab{lcoresize2}
Fix $\dd,k\geq2, (\dd,k)\neq (2,2)$ and $\e>0$. If $c\geq c_{\dd,k}+ n^{-1/2+\e}$ then \aas the $k$-core of $\calh_{\dd}(n,cn)$ has
\begin{itemize}
\item $\a(c) n + O(n^{3/4})$ vertices and
\item $\b(c) n +O(n^{3/4})$ edges.
\end{itemize}
\end{lemma}

\fix{While the proof of~\cite[Theorem 1.7]{jhk} (i.e.\ Lemma~\ref{lcoresize2}) is very technical, we give a brief heuristic explanation of why $c_{r,k}$ is the emergence threshold for a non-empty $k$-core and why the $k$-core of $\H_r(n,cn)$ is expected to contain around $\a(c)n$ vertices, when $c>c_{r,k}$. For vertex $v\in \H_r(n,cn)$, let $\la_t$ denote the probability that $v$ survives after $t$ iterations of the parallel stripping process. Intuitively, $\la:=\lim_{t\to\infty} \la_t$ is the probability that $v$ is in the $k$-core of $\H_r(n,cn)$ and thus the $k$-core is expected to contain $\la n$ vertices. It is easy to prove that the neighbourhood of $v$ locally converges in distribution to a hyper-tree  generated by the Poisson branching process with parameter $cr$ starting at $v$. Then, $v$ survives after $t$ iterations of the parallel stripping process if and only if $v$ is incident with at least $k$ hyperedges surviving from $t-1$ iterations of the stripping process, i.e.\ for each of these surviving hyperedges $x$, all the other $r-1$ vertices in $x$ are incident with at least $k-1$ other hyperedges that survive after $t-2$ iterations of the stripping process, and so on. To compute $\la_t$, let $\rho_t$ denote the probability that $v$ is incident with at least $k-1$ hyperedges surviving after $t$ iterations of the stripping process. As mentioned before, the number of hyperedges incident with $v$ is $\Po(cr)$. Each hyperedge survives after $t-1$ iterations of the stripping process with probability $\rho_{t-1}^{r-1} $. It is then easy to prove that the number of hyperedges incident with $v$ that survive after $t-1$ iterations of the stripping process is $\Po(\rho_{t-1}^{r-1} cr)$. Hence, we derive the recursion for $\rho_t$ and $\la_t$: 
\bean
\rho_0&=&1;\\
\rho_t&=&\pr(\Po(\rho_{t-1}^{r-1} cr)\ge k-1);\\
\la_t&=&\pr(\Po(\rho_{t-1}^{r-1} cr)\ge k).
\eean 
Put $\la=\pr(\Po(\rho^{r-1} cr)\ge k)$, where $\rho$ satisfies $\rho=\pr(\Po(\rho^{r-1} cr)\ge k-1)$. It is then easy to show that $\la=\rho=0$ if $c<c_{r,k}$ and thus, the heuristics above implies that the $k$-core is likely to be empty. For $c>c_{r,k}$, it is easy to show that $\rho_t$ converges to the larger root of $\rho=\pr(\Po(\rho^{r-1} cr)\ge k-1)$ and then $\la$ gives $\a(c)$ in~\eqn{alpha2}.} \fix{The expected sum of the degrees is $n\times\sum_{i\geq k}i\pr(\Po(\rho^{r-1} cr)=i)$ which is $2\b$.} 
\smallskip

By a close analysis of $\a(c)$ and $\b(c)$ we have:
\begin{lemma}\lab{l:diff} Fix $\dd,k\geq2, (\dd,k)\neq (2,2)$ and $0<\d<1/2$. There exist positive constants $K_1$, $K_2$ and $K_3$ such that if $c= c_{\dd,k}+n^{-\d}$ then
\begin{eqnarray*}
\mu(c)-\mu_{\dd,k}&=&K_1 n^{-\d/2}+O(n^{-\d}),\\
{\a(c)}-\alpha&=&K_2n^{-\d/2}+O(n^{-\d}), \\
{\b(c)}-\beta&=&K_3n^{-\d/2}+O(n^{-\d}).
\end{eqnarray*}
\end{lemma}

\fix{We present its proof in the Appendix.}

Lemma~\ref{lcoresize2}, along with Lemma~\ref{l:diff} above, {together with the fact that $\mu f_{k-1}(\mu)/f_k(\mu)$ is an increasing function on $\mu>0$ (see Lemma~\ref{l:gk} below),} yields the following bounds on the size of the $k$-core:

\begin{corollary}\lab{ccoresize}
Fix $\dd,k\geq2, (\dd,k)\neq (2,2)$ and $0<\d<1/2$. There exist {positive constants} $K_1=K_1(r,k),K_2=K_2(r,k)$ and \fix{$K_3=K_3(r,k)$}: if $c=c_{\dd,k}+ n^{-\d}$ then \aas the $k$-core of $\calh_{\dd}(n,cn)$ has
\begin{enumerate}
\item[(a)] $\a n +K_1n^{1-\d/2}+O(n^{1-\d}+n^{3/4})$ vertices and
\item[(b)] $\b n +K_2n^{1-\d/2}+O(n^{1-\d}+n^{{3/4}})$ \fix{hyperedges} and
\item[(c)] \fix{average degree $r\beta/\alpha+K_3n^{-\d/2}+O(n^{-\d}+n^{-1/4})$.}
\end{enumerate}
\end{corollary}

We close this section by introducing two terms that are standard when analyzing the $k$-core stripping process:
\begin{definition} The {\em light} vertices of a \fix{hypergraph} are the vertices of degree less than $k$.  The {\em heavy} vertices are those of degree at least $k$.
\end{definition}

\section{Bounding the stripping number}\lab{smt1}

We restate the parallel stripping process in a slowed-down version, which will be more convenient to analyze.   Rather than removing all vertices of $S_i$ at once, we remove them one at a time.  When removing a vertex, we slow down further by removing one edge at a time.

To facilitate this, we maintain a queue $\msq$ containing all deletable vertices; i.e. all vertices of degree less than $k$:


\begin{tabbing}
{\bf SLOW-STRIP}\\
{\bf Input:}  A hypergraph $G$.\\
Ini\=tialize: \=$t:=0,G_0:=G$, $\fix{\msq_1=\msq}$ is a list of all vertices of degree less than $k$ in $G$,\\ 
\>\>ordered uniformly at random.\\
While $\msq\neq\emptyset$:\\
\>Let $v$ \fix{be the first vertex in $\msq$}.\\
\>Remove a hyperedge $e$  selected uniformly at random from all those containing $v$.\\
\>If any other vertex of $e$  has its degree drop to below $k$ then add that vertex to the end of $\msq$.\\
\>\fix{Repeatedly remove the vertex in the front of $\msq$ if its degree is zero.}\\
\>$G_{t+1}$ is the resulting hypergraph; \fix{$\msq_{t+1}:=\msq$}; $t:=t+1$.\\
\end{tabbing}


\fix{To be clear: Note that the removal of $v$ might cause the degree of a vertex not at the front of $\msq$ to drop to zero.  That vertex remains in $\msq$ until it reaches the front, at which point it will be removed.  So it it possible that multiple vertices are removed from the front of $\msq$ during one step of SLOW-STRIP.}


At any point, \fix{$\msq$} may contain some vertices from $S_i$ and some from $S_{i+1}$.  However,
the vertices of $S_i$ are removed before the vertices of $S_{i+1}$.  Note also that when processing a vertex $v\in\fix{\msq}$, all edges from $v$ are removed (and hence $v$ is removed) before moving to the next vertex of $\fix{\msq}$.  So this procedure removes vertices in the same order as the parallel stripping process. In particular, if $t$ is the \fix{first step in SLOW-STRIP during which the  vertex at the front of $\msq_t$ is in $S_i$}, then the set of vertices in $\msq_t$ is exactly $S_i$ \fix{minus possibly some}  vertices with degree zero in ${\widehat H}_i$. Therefore,  the total degree of the light vertices in $G_t$ (i.e.\ vertices in $\msq_t$) equals exactly the total degree of $S_i$ in ${\widehat H}_i$. 


\begin{definition}  We use $t(i)$ to denote the \fix{ first iteration of SLOW-STRIP in which the vertex at the front of $\msq_t$ is in $S_i$.}

\fix{If every vertex of $S_i$ has degree 0 in $G_i$ then there is no iteration in which the vertex at the front of $\msq_t$ is in $S_i$ and so we define $t(i)$ to be the iteration during which all vertices of $S_i$ are removed.  If this is the case then $i$ is the final iteration of the parallel stripping process and $t(i)$ is the final iteration of SLOW-STRIP.}
\end{definition}

\fix{Roughly speaking, we can view $t(i)$ as the iteration of SLOW-STRIP in which iteration $i$ of the parallel stripping process begins.  This is not quite accurate in that perhaps iteration $i$ began during iteration $t(i)-1$ of SLOW-STRIP if the first vertices of $S_i$ to reach the front of $\msq$ had degree zero.}

We also define:
\[\tau \mbox{ is the iteration in which SLOW-STRIP halts.}\]

We will focus much of our analysis on the following parameters of $G_t$:
\begin{itemize}
\item $L_t$ is the total degree of the light vertices in $G_t$; i.e.\ of the vertices in $\msq_t$.
\item $N_t$ is the number of heavy vertices in $G_t$; i.e.\ of the vertices outside of $\msq_t$.
\item $D_t$ is the  total degree of the heavy vertices in $G_t$.
\end{itemize}
We denote the triple of these values as:
\[\fix{\calt_t}=(L_t,N_t,D_t),\]
\fix{and
\[
\calf_t=\{\calt_s\}_{s\le t}:=\{\calt_0,\calt_1,\ldots,\calt_t\}.
\]
}

\fix{
We first give an overview of the proof of Theorem~\ref{mt}.  \fix{As mentioned earlier, $L_{t(i)}$ equals the total degree of $S_i$ in ${\widehat H}_i$. This will allow us to relate $|S_i|$ to $L_{t(i)}$.} The key arguments are to rather precisely describe the evolution of $(L_t)_{t\ge 0}$, especially in a critical range of $t$.  We will show that after a sufficiently large but bounded number $B$ of iterations of the parallel stripping process, the number of vertices remaining in the hypergraph becomes very close to $\a n$ (see Lemma~\ref{l:B}). Then, we will prove that in SLOW-STRIP, $L_t$ decreases with at least a certain rate for all $t\ge t(B)$ (see Lemma~\ref{llt1}). Using that we can bound from below the rate at which $L_{t(i)}$ (or $|S_i|$, roughly speaking) decreases for each $i\ge B$. This allows us to obtain the upper bound for $s_k(\H_r(n,cn))$ as in Theorem~\ref{mt}.

To obtain the lower bound of $s_k(\H_r(n,cn))$ in the supercritical case as in Theorem~\ref{mt}(a), we will tightly bound the rate (from both below and above) at which $L_t$ decreases for $t\ge t(B)$ (See Section~\ref{slsi}).  This enables us to establish a rather precise description of $(|S_i|)_{i\ge B}$ (see Lemma~\ref{lsi}), and hence deduce the desired lower bound on the stripping number.    

However, we will first present a slightly weaker lower bound, i.e.\ without the logarithmic factor, with a much simpler proof in Section~\ref{sec:mt}. \fix{We do so because this weaker bound is part of the proof of Theorem~\ref{mt}(b).} The key idea is to focus on the steps of SLOW-STRIP during which the last $Kn^{1-\d/2}$ vertices are removed before reaching the $k$-core, for some constant $K>0$. We will show that with high probability $L_t=O(n^{1-\d})$ in all these steps (see Lemma~\ref{llt2}). We will then consider the iterations of the parallel stripping process during which the last $Kn^{1-\d/2}$ vertices are removed. In each iteration, the total number of vertices being removed must be $O(n^{1-\d})$ since their total degree is $O(n^{1-\d})$. In order to remove $Kn^{1-\d/2}$ vertices, at least $\Omega(n^{\d/2})$ iterations of parallel stripping are required. This yields the slightly weaker lower bound. This same proof, combined with a coupling argument, will then yield the lower bound claimed in Theorem~\ref{mt}(b) (see Sections~\ref{slt} and~\ref{smtb}). 
} 
\subsection{The allocation-partition model}\lab{scm}

{We use the following {\em allocation-partition model} ({\em AP-model}), denoted {by} $AP_r(n,m)$. It was used in~\cite{pw} and is a slight modification of what is called the pairing-allocation model in~\cite{CW}.  We are given a set of $rm$ vertex-copies. We represent each vertex $v$ as a bin. We choose two random objects:  (i) a uniformly random partition of the vertex-copies into parts of size exactly $\dd$; (ii) a uniform  allocation of each vertex-copy into a bin. We call the output a {\em configuration.}  

Having chosen a configuration from $AP_r(n,m)$, we can transform it into a \fix{hypergraph} as follows: We  contract the bins into vertices, and each part of the partition becomes a hyperedge.  It is easy to show, with simple counting arguments, that all simple hypergraphs with $n$ vertices and $m$ edges are generated with equal probability. ({\em Simple} means that no two hyperedges  are identical,
and no vertex appears twice in the same hyperedge).

The AP-model differs from the configuration model of Bollob\'{a}s~\cite{bb} in that the degree sequence is not fixed in advance, and the vertex-copies are not initially assigned to actual vertices.  Note that the random partition and the random allocation are orthogonal and they can be chosen independently of each other.  This will be very helpful below when we condition on events (specifically values of $L_t,N_t,D_t$) which specify partial information about both the partition and the allocation.

$\H_r(n,m)$ only selects {\em simple} hypergraphs.  Of course, the hypergraph formed by the AP-model might not be simple, but when $m=O(n)$, the probability of obtaining a simple hypergraph is at least $\eps$ for some absolute constant $\eps>0$.
This immediately yields:

\begin{corollary}\lab{ccon0}
If $m=O(n)$ and if property $Q$ holds a.a.s.\ for $AP_r(n,m)$, then $Q$ holds a.a.s.\ for $\H_r(n,m)$.
\end{corollary}

\subsection{SLOW-STRIP on the AP-model}

We will analyse the running of SLOW-STRIP on a random configuration generated by the AP-model, and then use Corollary~\ref{ccon0} to translate the results to $\H_r(n,m)$.  We continue to use hypergraph terminology to refer to the configuration, and so strictly speaking, a ``vertex'' is a bin, a ``hyperedge'' is an $r$-tuple of the partition along with an allocation of the vertex-copies in that $r$-tuple, $G_t$ is the configuration remaining after $t$ iterations, etc. We say a vertex-copy is {\em light/heavy} if the bin it is allocated to is {\em light/heavy}. Where we say ``graph" in this context, we mean the (not neccessarily simple) hypergraph obtained from the configuration by contracting the bins.

SLOW-STRIP runs on a configuration {in $AP_r(n,m)$}  as follows: Initially, the queue $\msq_0$ contains all the light bins; i.e. the bins of size less than $k$. In each step, SLOW-STRIP removes a vertex-copy $x$ of the bin $u$ in the front of the queue, together with the $r-1$ vertex-copies in the same part as $x$. Each time we delete one of the $D_t$ vertex-copies not in $\msq_t$, we query whether the bin containing that copy now has size $k-1$; if so then we move that bin to the queue. When all vertex-copies of $u$ are removed then $u$ is deleted from the queue.  \fix{We define the $k$-core of a configuration to be what remains when SLOW-STRIP terminates.}

\fix{
We use $G_t$ to denote the configuration remaining after $t$ steps.  In terms of a configuration, our three key parameters become:
\begin{itemize}
\item $L_t$ is the number of light vertex-copies in $G_t$; i.e.\ the total number of copies in  $\msq_t$.
\item $N_t$ is the number of heavy bins in $G_t$.
\item $D_t$ is the number of heavy  vertex-copies in $G_t$. 
\end{itemize}
}

We will typically condition on the values of $\calf_t=\fix{\{(L_s,N_s,D_s)\}_{s\le t}}$.  \fix{The reader has likely noticed a lack of symmetry: we do not condition on the number of light bins; i.e.\ there is no light analogue to $N_t$.  This is because the number and sizes of the light bins have no significant effect on the running of SLOW-STRIP - all that matters is the number of copies in $\msq_t$.  Also note that $D_t+L_t=r(cn-t)$, as this is the number of remaining vertex-copies, and so it is not necessary for $\calf_t$ to record both $D_t$ and $L_t$; but we find it convenient to do so.}

The following observation (first shown in ~\cite{CW}) enables our analysis.  \fix{ Recall that a configuration consists of (i) a partition of the vertex-copies into parts of size $r$, and (ii) an allocation of the vertex-copies to bins.  Recall further that in our random model, this partition and allocation are chosen uniformly and independently of each other.} 

\begin{observation}\label{oft}
Upon conditioning on $\calf_t$,
\fix{
\begin{enumerate}
\item[(a)] every partition of the remaining vertex-copies is equally likely; and
\item[(b)] every allocation of the $D_t$ heavy vertex-copies to the $N_t$ heavy bins such that each bin has size at least $k$ is equally likely.
\end{enumerate}
}
\end{observation}

\fix{In other words, $\{(L_t,N_t,D_t)\}_{t\ge 0}$ is Markovian.  This Markovian property will allow easy analysis of $\ex(L_{t+1}-L_t\mid \calf_t)$ for instance.}\smallskip

 \fix{

\proofstart {\em Part (a):}  $L_t,N_t,D_t$ say nothing about the partition; they are parameters of the allocation.  So every partition of the remaining vertex-copies is equally likely. (The deleted vertex copies have already been assigned to parts.)

{\em Part (b):}   We begin with some intuition:

 $N_t,D_t$ only change when one of the $r-1$ copies that we choose to be in the same part as $x$, and hence delete, is heavy.  Suppose that the chosen copy is in a bin of size $k$.  Then that entire bin is moved to $\msq_t$, and (at least intuitively) we have not exposed anything new about the remaining bins. So any allocation in which each remaining heavy bin has size at least $k$ is equally likely.  The more subtle case is when the chosen copy is in a bin of size greater than $k$.  It is important to note that we do not expose the size of that bin, only the fact that the size is at least $k+1$.  So when we delete the vertex-copy, our exposure only says that the bin now has size at least $k$. Again, any allocation in which each  heavy bin has size at least $k$ is equally likely. 
 
 }
 
 \fix{
 And now a proof:  Consider any configuration $G$ with $n$ vertices and $rm$ vertex-copies. Let $G=G_0,G_1,...,G_t$  be the sequence of configurations obtained by running SLOW-STRIP on $G$ for $t$ iterations; define $\calf_0,...,\calf_t$ similarly.  
 
 Expose $\Theta$, the $D_t$ heavy vertex-copies in $G_t$, and $\calb$, the $N_t$ heavy bins. Let $H_t$ be the allocation of $\Theta$ to $\calb$ in $G$.  Let $H_t'$ be any other allocation of $\Theta$ to $\calb$ such that each bin in $\calb$ receives at least $k$ members of $\Theta$.   We will prove that $H'_t$ occurs with the same probability as $H_t$ when conditioning on $\Theta,\calb$ being the heavy vertex-copies and bins.  This is what is asserted by part (b).
 
Let $G'$ be the configuration obtained from $G$ by replacing $H_t$ with $H_t'$; i.e. by allocating every vertex-copy in $\Theta$ according to $H_t'$ and allocating every other vertex-copy according to $G$.  Simliarly, define $G'_i$ to be the configuration obtained from $G_i$ by replacing $H_t$ with $H_t'$ for each $0\leq i\leq t$.  We make two key observations:

 {\em Observation 1:} If SLOW-STRIP on $G'$ chooses the same initial ordering of the light bins as SLOW-STRIP on $G$, and at each step chooses the same vertex-copy from the bin at the front of the queue to be deleted, then it will produce the sequence $G_0',...,G'_t$. 
 
{\em Observation 2:} The probability of the sequence $G_0,....,G_t$ is the same as the probability of the sequence $G_0',...,G_t'$.    To see this, note first that $G,G'$ are chosen as  our random configuration from $AP_r(n,m)$ with the same probability as they both have the same $n$ bins and $rm$ vertex-copies. Then note that $G,G'$ have the same number of light bins and for $0\leq i\leq t$  each bin has the same size in $G_i$ as in $G_i'$.  So Observation 1 implies that the two sequences occur with the same probability.

The probability that after $t$ steps, the set of heavy vertex-copies is $\Theta$, the set of heavy bins is $\calb$ and the allocation of $\Theta$ to $\calb$ is $H_t$, is the  sum over all sequences $G_0,...,G_t$ that end in a $G_t$ with $\Theta,\calb,H_t$ of the probability of that sequence.  The analogous statement is true of $H'_t$.   We have established 
a bijection from the sequences yielding $\Theta,\calb,H_t$ to the sequences yielding $\Theta,\calb,H'_t$ which preserves probabilities.  So the probability of obtaining $H_t$ is equal to the probability of obtaining $H_t'$.  This is part (b).
}

\proofend

\smallskip

Since $\calh_r(n,cn)$ is relatively hard to analyse directly; many of the existing proofs for the size of the $k$-core of $\calh_r(n,cn)$ use simpler models \fix{and then the results are translated to $\calh_r(n,cn)$; specifically~\cite{jhk} proved Lemma~\ref{lcoresize2} for the Poisson cloning model.  We will show that monotone results can be translated from the Poisson cloning model to the AP-model and so we have:}

\begin{lemma}\lab{lem:AP}
Both Lemma~\ref{lcoresize2} and Corollary~\ref{ccoresize} hold for the AP-model.
\end{lemma}

\fix{The proof is in the appendix.

The reader may notice that our result does not actually require} Lemma~\ref{lem:AP}. \fix{By Lemma~\ref{lcoresize2} and Corollary~\ref{ccoresize}, we may run SLOW-STRIP on a random hypergraph $H\in \calh_r(n,cn)$ until the remaining subgraph $G_t$ has the desired order as required in Section~\ref{sec:mt}. It is easy to see that by conditioning on ${\calf_t}$, $G_t$ is uniformly distributed over all hypergraphs with graph parameters agreeing with $\calt_t$.  Therefore, we could have applied the AP-model to $G_t$ and started  our analysis from there. However, it makes the proof a little easier to understand if we start with $H\in AP_r(n,cn)$ and apply Lemma~\ref{lem:AP}, which guarantees that with a high probability, we will obtain some configuration $G_t$ of the desired order. This is how Section~\ref{sec:mt} will be presented.}

\subsection{The supercritical case: $c=c_{r,k}+n^{-\d}$} \lab{sec:mt}

We first consider the supercritical case in which we assume $c=c_{r,k}+n^{-\d}$ where $0<\d<1/2$. These conditions on $\d$ are assumed in the rest of Section~\ref{smt1} as desired by the hypotheses in Theorem~\ref{mt}. 
Our goal is to prove that $s(\calh_r(n,cn))$ is bounded above by $O(n^{\d/2}\log n)$ and
below by $\Omega(n^{\d/2})$. \fix{Note that this lower bound is slightly weaker than that of Theorem~\ref{mt}(a), but the proof is much simpler. The tight bound of Theorem~\ref{mt}(a) will be proved in Section~\ref{slsi}.}

Recall that we are running SLOW-STRIP on a configuration generated from the AP-model.

The following lemma allows us to assume that the remaining subgraph is sufficiently close to the $k$-core. Roughly speaking, after a constant number of rounds of the parallel stripping process, we can get within any linear distance of the $k$-core.  Recall that  $\hG_i$ is the subgraph remaining after $i-1$ rounds of the parallel stripping process and $\C_k(H)$ is the $k$-core of $H$.

\begin{lemma}\lab{l:B}  Let $\eps,\eps_0>0$ be fixed. Assume $c\ge c_{\dd,k}+n^{-1/2+\e}$. Then, for $H\in AP_{\dd}(n,cn)$,
 there exists a constant $B=B(r,k,\e_0)>0$, such that a.a.s.\ $|\hG_B\setminus \C_k(H)|\le \eps_0 n$.
\end{lemma}

\begin{proof} 
Let $\eps'>0$ be a small constant to be determined later and let $c'=c+\eps'$. Choose $H'$ {according to the distribution of} $AP_{\dd}(n,c'n)$  and generate $H$ by removing $\eps'n$ edges chosen uniformly at random in $H'$. (I.e.\ we remove $\eps' n$ uniform parts and the vertex-copies contained in those parts.) Then $H\subseteq H'$ and $H$ is distributed as {$AP_r(n,cn)$}.  Run the parallel $k$-stripping process on $H'$, using $\hG'_i$ to denote the subgraph remaining after $i-1$ iterations. Several other papers (see eg. Proposition 31 of~\cite{amxor}), show that for any $\s>0$ there exists a constant $B>0$ such that $|\hG'_B\setminus \C_k(H')|\le \sigma n$, if $H'\in \calh_{r}(n,c'n)$; the same conclusion translates to the AP-model with a similar argument.

Note that every vertex removed during the first $B-1$ iterations of the parallel stripping process applied to $H'$ would also have been removed  during the first $B-1$ iterations of the parallel stripping process applied to $H$.  Thus $\hG_B\subseteq \hG'_B$.  Also, $\calc_k(H)\subseteq\calc_k(H')$ and, by Lemmas~\ref{lcoresize2} {and~\ref{lem:AP}}, 
 $|\calc_k(H')\setminus\calc_k(H)|=(\a(c')-\a(c))n +o(n)$. Therefore:
\bean
|\hG_B\setminus \C_k(H)|&\leq& |\hG'_B\setminus \C_k(H')|+|\calc_k(H')\setminus\calc_k(H)|\le \sigma n+(\a(c')-\a(c))n +o(n)\\
&<&2(\s +\a(c')-\a(c))n<\eps_0 n,
\eean
where the last inequality holds as long as $\e',\s$ are sufficiently small (both depending on $\e_0$).
\end{proof}

Lemma~\ref{l:B} says that for any $\e_0>0$ there is a $B=B(\e_0)$ such that after iteration $B-1$  of the parallel process, or equivalently, at the beginning of step $t(B)$  of SLOW-STRIP, the size of the remaining \fix{hypergraph} $G_{t(B)}$ is at most $\e_0 n$ greater than the size of the $k$-core. This implies that various parameters are very close to those of the $k$-core, {and this is where our analysis starts: we}
focus on $t\geq t(B)$.

A key part of our analysis is to control the change in $L_t$.  We will express the expected change of $L_t$ at step $t$ as a function of $L_t$, $N_t$ and $D_t$. {Recall}:
\[\calf_t=\fix{\{\calt_s\}_{s\le t}=\{(L_s,N_s,D_s)\}_{s\le t}}\]
and we will estimate $\ex(L_{t+1} \mid \calf_t)$. The following lemma bounds this expectation.

\begin{lemma}\lab{llt1}  There are constants $B,K$ such that: If $c=c_{\dd,k}+n^{-\d}$, then 
a.a.s.\ for every $t(B)\le t<\tau$,
\[\ex(L_{t+1}\mid \calf_t )\leq L_t-Kn^{-\d/2}.\]
\end{lemma}

\noindent{\bf Remark}. 
  The key Lemma 32 of~\cite{amxor} implies that, when $c=c_{r,k}+\e$, we have \aas $\ex(L_{t+1}{\mid \calf_t})\leq (1-\z) L_t$ for a constant $\z>0$.  That difference is what causes the stripping number and depth to rise from $O(\log n)$ for $c>c_{r,k}+\eps$ to $n^{\Theta(1)}$ for $c=c_{r,k}+n^{\d}$.  It is also the cause of most of the difficulties in this paper.  However, the weaker fact that $\ex(L_{t+1}-L_t{\mid \calf_t})$  remains bounded below $-Kn^{-\d/2}$,
no matter how small $L_{t}$ gets,  is still very useful to our analysis.

{Let $\a=\a(r,k)$ be the constant given in~\eqn{alpha-beta} and $K_1=K_1(r,k)$ be the constant specified in Corollary~\ref{ccoresize}. Then, by Corollary~\ref{ccoresize}, for} $c=c_{r,k}+n^{-\d}$, \aas the size of the $k$-core is $\a n+K_1 n^{1-\d/2}+o(n^{1-\d/2})$.

Lemma~\ref{llt1} will be used to bound the stripping number from above.   To obtain a lower bound, we will focus on the so-called Phase 2 of SLOW-STRIP, defined as follows:
\[\g=3\fix{K_2}n^{1-\d/2}.\]
Let $i^*$ be the first iteration of the parallel stripping process at the beginning of which the number of \fix{hyperedges} in the remaining \fix{hypergraph} is at most  $\lfloor\fix{\beta} n+\g\rfloor$.

\begin{definition} \lab{def:t0}
 $t_0$ is the first step of SLOW-STRIP at the beginning of which the number of \fix{hyperedges} in the remaining \fix{hypergraph} is exactly $\lfloor\fix{\beta} n+\g\rfloor$.     We refer to steps $t=0,....,t_0-1$ as {\em Phase 1} and the remaining steps as {\em Phase 2}.
\end{definition}

Note that we can obtain a lower bound on the number of \fix{iterations} in Phase 2:

\begin{lemma}\lab{lp0}  If $c=c_{\dd,k} + n^{-\d}$  then \aas Phase 2 \fix{lasts at least $\g/3=K_2n^{1-\d/2}$ iterations}.
\end{lemma}

\proofstart  By Corollary~\ref{ccoresize}(b), \aas the $k$-core has fewer than $\fix{\beta} n + 2\fix{K_2}n^{1-\d/2}$ hyperedges and so {SLOW-STRIP has to remove at least $\fix{K_2}n^{1-\d/2}$ hyperedges to reach it}.
\proofend

We will show that, throughout Phase 2, $L_t$ is small.

\begin{lemma}\lab{llt2}  If $c=c_{\dd,k}+n^{-\d}$, then \aas for every $t_0<t\leq\tau$, $L_{t}= O(n^{1-\d}).$
\end{lemma}

We defer the proofs of Lemmas~\ref{llt1} and~\ref{llt2} to Section~\ref{slt}. 

We now show how these lemmas nearly yield a proof of Theorem~\ref{mt}(a). \fix{Note that} Lemma~\ref{lp0}  bounds {from below} the total number of hyperedges that are removed during Phase 2. Lemma~\ref{llt2}  bounds the total degree of each $S_i$, which bounds {from above} the number of hyperedges that are deleted in each iteration. This implies that Phase 2 requires many iterations; specifically, it yields  a lower bound of $\Omega(n^{-\d/2})$ on the stripping number, which is within a factor of $\log n$ of the bound in Theorem~\ref{mt}(a) \fix{(see subsection~\ref{s.mtlb} for details)}.
For the upper bound: Lemma~\ref{llt1} bounds the rate at which $L_t$ decreases and so implies that it drops to zero quickly; of course, when $L_t=0$ then the stripping process has ended \fix{(see subsection~\ref{s.mtub} for details)}.

\subsubsection{Theorem~\ref{mt}(a): proof of a weaker lower bound}
\lab{s.mtlb}

We begin with a lower bound  of $\Omega(n^{-\d/2})$ on the stripping number when $c=c_{r,k}+n^{-\d}$ for some $0<\d<\hf$. This lower bound is slightly weaker than that of Theorem~\ref{mt}(a), but the proof is much simpler. Furthermore, this weaker lower bound will be used in Section~\ref{smtb} to prove Theorem~\ref{mt}(b). The proof for the tight lower bound of Theorem~\ref{mt}(a) will be presented in Section~\ref{slsi}.

We run SLOW-STRIP and recall that this can be viewed as also running the parallel stripping process slowly. 
\fix{Exactly one hyperedge is deleted in each iteration of SLOW-STRIP. By Lemma~\ref{lp0}}, at least $\fix{K_2} n^{1-\d/2}$ hyperedges are removed from $G_t$ during Phase 2. By Lemma~\ref{llt2}, at most $|L_{t(j)}|=O(n^{1-\d})$ of them belong to $S_j$ for each $j$. Therefore, we require at least $\fix{K_2}n^{1-\d/2}/O(n^{1-\d})=\Omega(n^{\d/2})$ iterations of the parallel stripping process to remove them all.

This proves that \aas\ the stripping number of $AP_r(n,cn)$ is at least $\Omega(n^{\d/2})$; Corollary~\ref{ccon0} implies that the same is true of $\H_r(n,cn)$.
\proofend

\subsubsection{Theorem~\ref{mt}(a): proof of the upper bound}
\lab{s.mtub}

Now we turn to the upper bound on the stripping number when $c=c_{r,k}+n^{-\d}$ for some $0<\d<\hf$.

 We will focus on the change in $L_{t(i)}$, the sum, over all $v\in S_i$ of the degree of $v$ at the beginning of iteration $i$ of the parallel stripping process.

From iterations $t(i)$ to $t(i+1)-1$ of SLOW-STRIP, all hyperedges from all vertices in $S_i$ must be deleted. One hyperedge is removed in each iteration, and it touches \fix{at least one and} at most $\dd$ members of $S_i$.  Thus 
\be
\inv{\dd} L_{t(i)} \le t(i+1)-t(i)\le L_{t(i)}. \lab{eq:ti}
\ee

\fix{Let $B$ and $K$ be constants specified in Lemma~\ref{llt1}.} Let $\tau^*$ be the first step \fix{$t$ such that $t\ge t(B)$ and} $\ex(L_{t+1}{\mid \calf_t})> L_t-Kn^{-\d/2}$; if there is no such step then we set $\tau^*=\tau$.  \fix{Then,  for all $t(B)\le t<\tau^*$,  
 \[
\ex(L_{t+1}\mid \calf_t) \le L_t-Kn^{-\d/2}. 
 \]

 Taking conditional expectation on $\calf_{t(i)}$ of both sides, for $i\ge B$, yields
 \[
\ex(L_{t+1}\mid \calf_{t(i)}) \le \ex(L_t\mid \calf_{t(i)})-Kn^{-\d/2},\quad \mbox{for all}\  t(i) \leq t <\tau^*.
 \]
\fix{If $t(i+1)<\tau^*$ then} inductively applying the above for all $t(i)\le t\le t(i+1)$ we get}
\[
\ex(L_{t(i+1)}|\calf_{t(i)})\leq L_{t(i)}-\Big(t(i+1)-t(i)\Big)Kn^{-\d/2}\leq L_{t(i)}-\inv{\dd} L_{t(i)}Kn^{-\d/2}
=L_{t(i)}\left(1-\frac{K}{\dd}n^{-\d/2}\right).
\]

Define the random process $(\LL_i)_{i\ge 0}$ as follows. For all $i$ such that $t(i)\le \tau^*$, let $\LL_i=L_{t(i)}$ and for all $i$ such that $t(i)>\tau^*$,  define $\LL_{i}=\LL_{i-1}(1-(K/\dd) n^{-\d/2})$. Let $T$ denote the minimum integer such that $\LL_{T}\le 0$; thus \aas $\t=t(T)$. {By Lemma~\ref{llt1}} a.a.s.\ $\tau^*=\tau$, and so a.a.s.\ $L_{t(i)}=\LL_i$ for all $0\le i\le T$ and $L_{t(T)}=\LL_T=0$. Hence, we only need to obtain an upper bound for $T$. Now we have that for every $i\ge B$,
\[
\ex (\LL_{i+1}\mid \LL_{i})\leq  \LL_{i}\left(1-\frac{K}{\dd}n^{-\d/2}\right).
\]
Taking expectation on both sides we obtain that
\[
\ex \LL_{i+1}\leq \ex \LL_{i}\left(1-\frac{K}{\dd}n^{-\d/2}\right).
\]
Hence, as $\LL_B\le n$, for each $i\ge B$,
\[\ex \LL_{i}\leq n\left(1-\frac{K}{\dd}n^{-\d/2}\right)^{i-B}.\]
Thus, for $i>B+\frac{2\dd}{K} n^{\d/2}\log n$, $\ex(\LL_{i})=o(1)$ and so \aas $T\le B+(2\dd/K)n^{\d/2}\log n$. This implies that  a.a.s.\ the process of $(\LL_i)_{i\ge 0}$ reaches $\LL_i\le 0$ within $B+(2\dd/K)n^{\d/2}\log n$ iterations and so a.a.s.\ the parallel stripping process halts within $B+(2\dd/K)n^{\d/2}\log n$ iterations.} Therefore \aas the stripping number of $AP_r(n,cn)$ is at most
$B+\frac{2\dd}{K} n^{\d/2}\log n=O(n^{\d/2}\log n)$; Corollary~\ref{ccon0} implies that the same is true of $\H_r(n,cn)$.
\proofend

\subsection{Bounds on $L_i$: proof of Lemmas~\ref{llt1} and~\ref{llt2}}\lab{slt}

In this section, we prove Lemmas~\ref{llt1} and~\ref{llt2}. So throughout, we have  $c=c_{r,k}+n^{-\d}$ for some $0<\d<\hf$.

Let $H\in AP_{\dd}(n,cn)$ and run the SLOW-STRIP algorithm on $H$.  Recall from Definition~\ref{def:t0} that
\[\g=3\fix{K_2}n^{1-\d/2}.\]
and $t_0$ is the first iteration of SLOW-STRIP in which the number of \fix{hyperedges} in the remaining \fix{hypergraph} is exactly $\lfloor\fix{\b} n+\g\rfloor$. The Second Phase of SLOW-STRIP consists of iterations \fix{$t\ge t_0$}. Lemma~\ref{lp0} enables us to \fix{focus on the case when the algorithm does enter this phase}.

\begin{definition} We say that $G_{t_0}$ is {\em nice} if the number of \fix{heavy} vertices in $G_{t_0}$ is \fix{between $\a n$ and $\a n+2K_1n^{1-\d/2}+(r-1)\g$, and the total degree of light vertices in $G_{t_0}$ is at most $r\gamma$.}
\end{definition}

\fix{Each iteration of SLOW-STRIP deletes at most $r-1$ heavy vertices, and reduces $L_t$ by at most $r$.  Moreover, a.a.s.\ Phase 2 of SLOW-STRIP lasts less than $\gamma$ steps by Lemma~\ref{lp0}. So Corollary~\ref{ccoresize} immediately implies} that \aas $G_{t_0}$ is nice.

Recall that for every $t\ge 0$, $G_t$ is the hypergraph remaining at the beginning of iteration $t$ of SLOW-STRIP. The light vertices in $G_t$ are defined to be the vertices with degree less than $k$ and $L_t$ denotes the total degree of the light vertices in $G_t$. Recall that
$\tau$ is the iteration in which SLOW-STRIP halts.

The following proposition follows immediately from Corollary~\ref{ccoresize} and Lemma~\ref{lp0}.

\begin{proposition}\lab{p:tau}
A.a.s.\ $t_0+\gamma/3\le\tau\le \fix{ t_0+\gamma}$.
\end{proposition}

 It will be convenient to define
\[\wtau=\min\{\tau,t_0+\gamma\},\]
and so Proposition~\ref{p:tau} implies that 
\be\lab{widetau}
\aas\ \t=\wtau.
\ee
The following proposition follows from the definition of $t_0$ and $\wtau$.

\begin{proposition}\lab{p:Gt} Assume $\tau>t_0$ {and $G_{t_0}$ is nice}.
Then for all $t_0\le t\le \wtau$, 
\begin{enumerate}
\item[(a)] the number of \fix{heavy} vertices in $G_t$ is $\alpha n+O(\gamma)$ and the number of hyperedges in $G_t$ is $\beta n+O(\gamma)$;
\item[(b)] $L_t=O(\g)$. 
\end{enumerate}
\end{proposition}

\begin{proof} By definition, the number of \fix{hyperedges} in $G_{t_0}$ is $\fix{\b} n+\g$. Since $G_{t_0}$ is nice, the number of \fix{heavy vertices} in $G_{t_0}$ is $\fix{\a n}+O(\g)$. Moreover, by the definition of $\wtau$, $\wtau-t_0=O(\g)$. We remove one hyperedge in each step, so for all $t_0\le t\le \wtau$, the number of hyperedges \fix{and heavy vertices} in $G_t$ change by $O(\g)$ from those in $G_{t_0}$. This immediately confirms part (a). \fix{Moreover, since $G_{t_0}$ is nice, $L_{t_0}=O(\g)$.  As $L_t$ changes by at most $k$ in each step, we have $L_t=O(\g)$ for every $t_0\le t\le \wtau$ and this proves part (b).}
\end{proof}

\subsubsection{Creating light vertices}\lab{sec:deg}
The key to analyzing the evolution of $L_t$ is determining the rate at which new  vertices are added to $\msq_t$; i.e.\ become light.  We begin by examining the distribution from Observation~\ref{oft}.

Given positive integers $N$, $D$ and $k\ge 0$ such that $D\ge kN$, define $Multi(N,D,k)$, the {\em truncated multinomial distribution}, to be the probability space consisting of integer vectors ${\bf X}=(X_1,\ldots,X_N)$ with domain ${\mathcal I}_k:=\{{\bf d}=(d_1,\ldots,d_N):\ \sum_{i=1}^N d_i=D,\ d_i\ge k,\ \forall i\in[N]\}$, such that for any ${\bf d}\in {\mathcal I}_k$,
$$
\pr({\bf X}={\bf d})=\frac{D!}{N^D \Psi}\prod_{i\in [N]}\frac{1}{d_i! }=\frac{\prod_{i\in[N]}1/d_i!}{\sum_{{\bf d}\in{\mathcal I}_k}\prod_{i\in[N]}1/d_i!},
$$
where
$$
\Psi=\sum_{{\bf d}\in {\mathcal I}_k}\frac{D!}{N^D}\prod_{i\in [N]}\frac{1}{ d_i!}.
$$

The degree distribution of the heavy vertices of $G_t$, conditional on $\calf_t$,
is exactly $Multi(N_t,D_t,k)$, by Observation~\ref{oft}.
It was proved in~\cite[Lemma 1]{CW} \fix{(also appeared in~\cite[eq.\ (7)]{AFP})} that the truncated multinomial variables can be well approximated by truncated Poisson random variables \fix{with expectation $D_t/N_t$}. 
The result is stated as follows. 
Recall the definition of $f_k(\la)$ from~(\ref{e.fk}). 
 We define:
\be\lab{e.gkx}
g_k(\la)=\la f_{k-1}(\la)/f_k(\la).
\ee
\fix{
Note that $g_k(\la)$ is the expectation of a Poisson random variable with parameter $\la$ truncated at $k$.
}
\begin{proposition}\lab{p:Poisson}
\fix{Let $k\ge 0$ be fixed, and $N$ and $D$ satisfy $D-kN=\Omega(N)$}. Assume ${\bf X}\sim Multi(N,D,k)$. For any $j\ge k$, let $\rho_j$ denote the proportion of $X$ that equals $j$. Then, {with probability $1-o(1/\fix{N})$,}
\begin{equation}
\rho_j=e^{-\la}\frac{\la^j}{f_k(\la)j!}+O(\fix{N^{-1/2}{\log N}}), \lab{PoissonApprox}
\end{equation}
where $\la$ satisfies {$g_k(\la)=D/N$}.
\end{proposition}

 By Lemma~\ref{l:gk} (below), $g_k(x)$ is an increasing function on $x>0$. It is easy to show that $\lim_{x\to 0} g_k(x)=k$. Hence, for any $D> kN$, there is a unique $\la$ that satisfies $\la f_{k-1}(\la)/f_k(\la)=D/N$ in the above proposition. 

It is easy to check from the definition of $\mu_{\dd,k}$ above~\eqn{murk} that
\begin{equation}\lab{egmrk}
g_k(\mu_{\dd,k})=r\beta/\alpha.
\end{equation}
\fix{Define:
\be
\zeta=\zeta_{\dd,k}=\dd\beta/\alpha .\lab{zeta}
\ee
\fix{The following lemma justifies that $\zeta>k$.  The proof consists of some tedious calculus, so we defer it to the Appendix.}
\begin{lemma} \lab{l:rho}
Suppose $\dd,k\ge 2$ and $(\dd,k)\neq (2,2)$. Then, $k<\zeta<\dd(k-1)$.
\end{lemma}
}

 Thus,
Proposition~\ref{p:Poisson} and Corollary~\ref{ccoresize} imply that \aas\ the proportion of degree $k$ vertices in the $k$-core is approximately
\begin{equation}
\bar \rho_{r,k}=e^{-\mu_{\dd,k}}\frac{\mu_{\dd,k}^k}{f_k(\mu_{\dd,k})k!}.\lab{pdk}
\end{equation}

The following technical lemma has appeared in several other papers (e.g.\ in~\cite{amxor}), but as the proof is short we include it \fix{in the Appendix.}
\begin{lemma}\lab{l:degreeK}
For every $\dd,k\ge 2$ $(\dd,k)\neq (2,2)$,
$$
\frac{k \bar\rho_{\dd,k}\cdot\alpha}{\dd\beta}=\frac{1}{(\dd-1)(k-1)}.
$$
\end{lemma}

 And now we can prove the key lemma governing the evolution of $L_t$.  Recall from above~\eqn{widetau} that $\wtau=\fix{\min}\{\t,t_0+\g\}$.

\begin{lemma}\lab{c:degreeK}  Assume that $\t>t_0$ and $G_{t_0}$ is nice.  For any $t_0\le t\le \wtau$, $\calf_t$ must be such that: for each of the $r-1$ random vertex-copies chosen during iteration $t$, the probability, conditional on $\calf_t$, that the vertex-copy belongs to a heavy bin which becomes light in this iteration, is $1/(\dd-1)(k-1)+O(\gamma/n)$.
\end{lemma}

\begin{proof}   By Observation~\ref{oft}, we can treat the allocation of the $D_t$ heavy vertex-copies to the $N_t$ heavy bins as a uniform allocation subject to each bin \fix{receiving} at least $k$ copies.  We first choose the $r-1$ vertex-copies, then we choose the allocation.  Let $x$ be any one vertex-copy and let $b$ be the bin to which $x$ is allocated; we wish to bound the probability  that $|b|=k$. 

Consider a different experiment: first conduct the allocation, and then choose $x$ uniformly from amongst all vertex-copies, and let $b$ be the bin to which $x$ was allocated; clearly these two experiments are equivalent ways to choose $b$. \fix{Let $\la^*$  be the unique root of $g_k(\la)=D_t/N_t$.} By Proposition~\ref{p:Poisson} the probability that we choose $x$ from a bin of size $k$ is:
\[q=\frac{k}{D_t}\times e^{-\la^*}\frac{{\la^*}^k}{f_k(\la^*)k!}N_t + O(n^{-1/2}\log n).\]
 (Note that the $o(1/n)$ failure probability in Proposition~\ref{p:Poisson} is absorbed into the $O(n^{-1/2}\log n)$ term.)

Since $G_{t_0}$ is nice, Proposition~\ref{p:Gt} yields:
\begin{equation}
N_t=\alpha n+O(\gamma),\ \ D_t=\dd\beta n+O(\gamma). \lab{vertex-edge}
\end{equation}
So
\[g_k(\fix{\la^*})=\fix{\frac{r\b n+O(\g)}{\a n+O(\g)}=\frac{r\b}{\a}+O(\g/n)}=g_k(\mu_{r,k})+O(\g/n).\]
\fix{Let $\d=g'_k(\mu_{r,k})$; recall that $g'_k$ and $\mu_{r,k}$ are independent of $\g,n$ and so $\d$ is a constant independent of $\g,n$.  Lemma~\ref{l:gk} (below) implies that $\d>0$.  Thus $g'_k(\la)>\hf\d$ for all $\la=\mu_{r,k}+o(n)$; similarly,  Lemma~\ref{l:gk} implies that for any $\z>0$, $g_k'(\la)>0$ for all $\la>\z n$.  These two bounds, along with~(\ref{vertex-edge}) easily imply that 
\[\la^*=\mu_{r,k}+O(\g/n).\]}  
Setting $h(x)=e^{-x}\frac{x^k}{f_k(x)k!}$,  recalling~(\ref{pdk}), and noting that $h'(\mu_{\dd,k})=O(1)$,  and applying Lemma~\ref{l:degreeK}, we have
\bean
q&=&\frac{kN_t}{D_t}h(\fix{\la^*})+O(n^{-1/2}{\log n})=\frac{k\a}{r\b}h(\mu_{r,k})+O(\g/n)
=\frac{k\a \bar\r_{r,k}}{r\b}+O(\g/n)\\
&=&\frac{1}{(\dd-1)(k-1)}\fix{+O(\g/n)}.
\eean
 \end{proof}

We close this subsection with a technical lemma proving the monotonicity of $g_k(x)$. \fix{The proof will be given in the Appendix.}

\begin{lemma}\lab{l:gk}
For any $x>0$,
$g'_k(x)>0$.
\end{lemma}

\subsubsection{Proof of Lemma~\ref{llt2}}\lab{sec:llt2}

Recall that we are running SLOW-STRIP on the AP-model.
Recall also from Section~\ref{smt1} that $L_t$ denotes the total degree of the vertices with degree less than $k$ in $G_t$, and recall from Definition~\ref{def:t0} that $t_0$ is the beginning of Phase 2 of SLOW-STRIP. {Corollary~\ref{ccoresize} {ensures that a.a.s.\ $G_{t_0}$ is nice} and so Proposition~\ref{p:Gt} immediately yields a weaker version of Lemma~\ref{llt2}:} \aas for all $t_0\le t\le \wtau$,
$L_t=O(n^{1-\d/2})=O(\gamma)$.
This allows us to prove that a weaker form of Lemma~\ref{llt1} holds for $t\geq t_0$:

\begin{lemma}\lab{llt3} {Assume that $\wtau>  t_0$ and $G_{t_0}$ is nice}. For every {$t_0\le t<\wtau$},

\[\ex(L_{t+1}{\mid \calf_t})= L_t\pm O(n^{-\d/2}).\]

\end{lemma}

\begin{proof} Let $v$ be the vertex taken from $\msq_t$ at iteration $t$.  Let $u_1,\ldots,u_{r-1}$ denote the other vertex-copies (except the copy in $v$) that are selected to be deleted in this step. The removal of the vertex-copy in $v$ contributes $-1$ to $\Delta L_t:=L_{t+1}-L_t$ always. We consider the contribution to $\Delta L_t$ from the removal of $u_i$. Let $v(u_i)$ denote the bin that contains $u_i$. There are four cases.

{\em Case 1}: The removed edge contains two copies of the same vertex.  The probability of this case is $O(1/n)$ and so the contribution of this case to $\ex(L_{t+1}-L_t\mid \calf_t)$ is $O(1/n)$.

For the remaining three cases, we can assume that no other copy in $v(u_i)$ has already been removed during iteration $t$.

{\em Case 2}: $v(u_i)$ was in $\msq_t$ in $G_t$. In this case, the contribution of the removal of $u_i$ to $\Delta L_t$ is $-1$. Since $u_i$ is chosen by the algorithm u.a.r.\ from all remaining vertex-copies, the probability of this event is $O(L_t/(D_t+L_t))=O(\gamma/n)$ by~\eqn{vertex-edge} and {Proposition~\ref{p:Gt}(b)}.

{\em Case 3}: $v(u_i)$ enters $\msq_t$ at iteration $t$. In this case, the size of $v(u_i)$ is $k$ in $G_t$ and the contribution to $\Delta L_t$ from the removal of $u_i$ is $k-1$. The probability of this event is $1/(\fix{\dd}-1)(k-1)+O(\gamma/n)$ by Lemma~\ref{c:degreeK}.

{\em Case 4}: $v(u_i)$ was not in \fix{$\msq_t$} and does not enter \fix{$\msq_{t+1}$} at iteration \fix{$t+1$}. In this case, the size of $v(u_i)$ is more than $k$ in $G_t$ and the contribution to $\Delta L_t$ from the removal of $u_i$ is $0$. The probability of this event is $1-1/\fix{(\dd-1)}(k-1)+O(\gamma/n)$.

By the linearity of expectation, summing the contributions of $u_1,\ldots,u_{r-1}$, we have
\bean
\ex(L_{t+1}-L_t\mid \calf_t)&=&-1+(\dd-1)\left((-1)\cdot O(\gamma/n)+(k-1)\left(\frac{1}{(\dd-1)(k-1)}+O(\gamma/n)\right)\right)\\
&&+O(1/n)=O(\gamma/n).
\eean
The lemma follows by noting that $\gamma=\Theta(n^{1-\d/2})$ by Definition~\ref{def:t0}.
\end{proof}

The following lemma is a simple application of the Hoeffding-Azuma Inequality\cite{azuma,hoeffding}.

\begin{lemma}\lab{l:azuma} Let $a_n$ and $c_n\ge 0$ be real numbers and $(X_{n,i})_{i\ge 0}$ be random variables with respect to a random process $(G_{n,i})_{i\ge 0}$   such that for all $t\ge 0$
$$
\ex(X_{n,i+1}\mid \fix{\{G_{n,s}\}_{s\le i}})\le X_{n,i}+a_n,
$$
and $|X_{n,i+1}-X_{n,i}|\le c_n$,
for every $i\ge 0$ and all (sufficiently large) $n$. Then, for any real number $j\ge 0$,
$$
\pr(X_{n,t}-X_{n,0}\ge ta_n+j )\le \exp\left(-\frac{j^2}{2t(c_n+|a_n|)^2}\right).
$$
\end{lemma}
\begin{proof} Define $Y_{n,i}=X_{n,i}-ia_n$ for all $i\ge 0$. Then, 
$$
\ex(Y_{n,i+1}\mid \fix{\{G_{n,s}\}_{s\le i}})=\ex(X_{n,i+1}\mid \fix{\{G_{n,s}\}_{s\le i}})-(i+1)a_n\le  X_{n,i}-ia_n = Y_{n,i}.
$$
Thus, $(Y_{n,i})_{ i\ge 0}$ is a supermartingale. Moreover, $|Y_{n,i+1}-Y_{n,i}|\le c_n+|a_n|$. By the Hoeffding-Azuma Inequality,
$$
\pr(Y_{n,t}-Y_{n,0}\ge j)\le \exp\left(-\frac{j^2}{2t(c_n+|a_n|)^2}\right).
$$

This completes the proof of the lemma.
\end{proof}

We now recall the statement of Lemma~\ref{llt2}:

\newtheorem*{llt2b}{Lemma~\ref{llt2}}
\begin{llt2b} If $c=c_{\dd,k}+ n^{-\d}$, then \aas for every {$ t_0\le t\le \tau$}:
\[L_{t}= O(n^{1-\d}).\]
\end{llt2b}

\begin{proof} 
{Without loss of generality we may assume $t_0<\tau$. By Corollary~\ref{ccoresize} we may also assume that $G_{t_0}$ is nice. Recall that $\wtau=\min\{\tau,t_0+\g\}$. }
By Lemma~\ref{llt3}, there exists a nonnegative sequence $(a_n)_{n\ge 1}$ such that $a_n=O(n^{-\d/2})=O(\gamma/n)$ and {for every $t_0\le t<\wtau$,}
$$
\ex(L_{t+1}\mid \calf_t)\le L_t+a_n,\ \ \ \ex(-L_{t+1}\mid \calf_t)\le -L_t+a_n.
$$
{Define ${\LL_t}$ such that $\LL_t=L_t$ for every $t\le \wtau$ and $\LL_{t+1}=\LL_t$ for all $t\ge \wtau$. Now for every $t\ge t_0$,
$$
\ex(\LL_{t+1}\mid \fix{\{\LL_s\}_{s\le t}})\le \LL_t+a_n,\ \ \ \ex(-\LL_{t+1}\mid \fix{\{\LL_s\}_{s\le t}})\le -\LL_t+a_n.
$$
It is clear that $|\LL_{t+1}-\LL_t|=O(1)$ always as $L_{t+1}=L_t+O(1)$ always.
}
By Lemma~\ref{l:azuma} with $j=\gamma^{1/2}\log n$ {and $c_n=O(1)$}, 
{for every $t_0\le t\le t_0+\g$,}
$$
\pr(\LL_t\ge \LL_{t_0}+a_n(t-t_0)+\gamma^{1/2}\log n)=o(n^{-1}),\quad \pr(-\LL_t\ge -\LL_{t_0}+a_n(t-t_0)+\gamma^{1/2}\log n)=o(n^{-1}).
$$

We apply the union bound over all $t_0\le t<{t_0+\g}$, along with the asymptotics $a_n=O(\gamma/n)$, 
 and $\gamma^{1/2}\log n =o(\gamma^2/n)$
(since $\g=3\fix{K_2}n^{1-\d/2}$ and  $\d<1/2$), to obtain that \aas for all $t_0\le t\le {t_0+\g}$:
\begin{eqnarray}
\LL_{t}&\ge& \LL_{t_0}-O(\g^2 / n)\lab{L0a}\\
\LL_{t}&\le& \LL_{t_0}+O(\g^2 / n) \lab{Lt}
\end{eqnarray}
By{~\eqn{widetau}}, a.a.s.\ {$\wtau=\tau$ and so a.a.s.\ $\tau=\wtau\le t_0+\g$ and $L_{\wtau}=L_{\tau}=\LL_{\tau}=0$}.  Therefore, (\ref{L0a}) with $t=\tau$ yields
\begin{equation}\lab{L0b}
L_{t_0}={\LL_{t_0}=} O(\gamma^2/n).
\end{equation}
Substituting that into (\ref{Lt}) yields that for all $t_0\le t\le\tau$,
$L_t={\LL_t}=O(\gamma^2/n)=O(n^{1-\d})$, thus establishing the lemma.
\end{proof}

\subsubsection{$\z_t$ and $\bar p_t$}\lab{s.vdk}

In this section, we  study two parameters $\zeta_t$ and $\bar p_t$ to be defined below,
which allow for a more careful analysis of the rate at which new light vertices are formed.

Recall that $L_t$ is the {total degree} of the light vertices in $G_t$ (i.e. those of degree less than $k$); {eventually, our goal is to prove Lemma~\ref{llt1}, i.e.\ to bound $\ex (L_{t+1}-L_t\mid \calf_t)$ from above. This quantity has been bounded from below in the previous section for $t\ge t_0$ (Lemma~\ref{llt3}).  The main challenge for Lemma~\ref{llt1} is that we need to show a uniform bound for all $t\ge t(B)$, as long as $B$ is a sufficiently large constant (recall that $t_0$ grows with $n$). This also makes a significant difference from Lemma~\ref{llt3} which only requires $t\ge t_0$.}

Now let $H\in AP_{\dd}(n,cn)$, where $c=c_{\dd,k}+n^{-\delta}$ and let $B$ be a sufficiently large constant whose value is to be determined later. Recall that $G_t$ is the hypergraph remaining after $t$ iterations of SLOW-STRIP and so  $G_{t(B)}=\hG_B$.

Recalling that $N_t$ is the number of heavy vertices, and $D_t$ is the total degree of the heavy vertices, 
we define:
\[\zeta_t=\z_t(\calf_t)=D_t/N_t \mbox{ is the average degree of the heavy vertices in } G_t.\]

Recalling that, by Corollary~\ref{ccoresize}, the $k$-core \aas has $\a n +o(n)$ vertices and $\b n+o(n)$ edges and $
\zeta=\dd\beta/\alpha$.
Therefore the $k$-core \aas has average degree $\z+o(1)$ and so $\z_t$ approaches $\z$.
Note that $\z=g_k(\mu_{r,k})$ by~(\ref{egmrk}).

We are interested in the probability that a particular heavy vertex-copy is allocated to a bin of size $k$
since this tells the proportion of the heavy vertices in deleted hyperedges that become light. So, recalling Observation~\ref{oft}, we define:

\begin{definition}
$\bar p_t=\bar p_t(\fix{\calt_t})$ \hspace{1ex}  is the probability that a given vertex-copy is assigned to a bin of size $k$ in a uniformly random allocation of $D_t$ points to $N_t$ bins  subject to each bin receiving at least $k$ points.
\end{definition}

{We will deduce $\bar p_t$ as a certain function of $\zeta_t$, approximately.}
For all $x>k$, we define:

\begin{eqnarray}
\la(x) \mbox{ is the root of }  g_k(\la)&=&x; \qquad \mbox{ i.e. } \la f_{k-1}(\la)=x f_k(\la)
\label{e.lax}\\
\psi(x)&=&\frac{e^{-\la{(x)}}\la{(x)}^{k-1}}{f_{k-1}(\la{(x)})(k-1)!} \lab{h}
\end{eqnarray}

With some basic calculations we can show that $\psi(x)$ is \fix{strictly decreasing for $x>k$}. 

\begin{lemma}\lab{l2:monotone}
For all $x>k$, $\psi'_k(x)<0$.
\end{lemma}

We defer the proof to \fix{the Appendix}.

In the $k$-core, {a.a.s.}\ the proportion of the total degree that comes from vertices of degree $k$ is {approximately} $\psi(\z)$ {by Corollary~\ref{ccoresize} and Proposition~\ref{p:Poisson}}.

Since $\bar p_t$ approaches $\psi(\z)$, we should have $\bar p_t \approx \psi(\z_t)$.
Our next lemma formalizes this approximation:

\begin{lemma}\lab{l:monotone}
\fix{Assume $N_t=\Omega(n)$} for every $t\le \tau$. Then  $\bar p_t= (1+O(n^{-1/2}\log n))\psi(\zeta_t)$ \fix{for all $t\le \tau$}.
\end{lemma}
\begin{proof} Let $\la$ be chosen such that $\la f_{k-1}(\la)=\zeta_t f_k(\la)$. Then by Proposition~\ref{p:Poisson}, the following holds for all $t\le \tau$:
\bean
{\bar p_t}&=&\frac{ke^{-\la}\la^k N_t}{f_k(\la)(k-1)!D_t}(1+O(n^{-1/2}\log n))\fix{+o(n^{-1})}=\frac{e^{-\la}\la^k}{f_k(\la)(k-1)!\zeta_t}(1+O(n^{-1/2}\log n))\\
&=&\frac{e^{-\la}\la^{k-1}}{f_{k-1}(\la)(k-1)!}(1+O(n^{-1/2}\log n))\qquad\mbox{by (\ref{e.lax}).}\qedhere
\eean
\end{proof}

We will analyze $\z_t,\bar p_t$ in order to prove Lemma~\ref{llt1}.

Recall that $g_k(x)=x f_{k-1}(x)/f_k(x)$ as defined in~\eqn{e.gkx} and that $\zeta=\dd \beta/\alpha$. Then, $\mu_{\dd,k}$ is the root of $g_k(x)=\zeta$ by~\eqn{egmrk}. Recall also that
\[
\psi(x)=\frac{e^{-\la(x)}\la(x)^{k-1}}{f_{k-1}(\la(x))(k-1)!},
\]
where $\la(x)$ is the root of $g_k(\la)=x$. Now define
$p^*=\psi(\zeta)$. Then, we have
\[
p^*=\frac{e^{-\mu_{\dd,k}}\mu_{\dd,{k}}^{k-1}}{f_{k-1}(\mu_{\dd,k})(k-1)!}.
\]
By the definition of $\alpha$, $\beta$ below~\eqn{murk} and $\bar\rho_{r,k}$ in~\eqn{pdk}, we have $p^*=k\bar\rho_{r,k}/\zeta$. Then, by Lemma~\ref{l:degreeK}, we have $p^*=1/(\dd-1)(k-1)$. As a summary of the above discussion, we have the following equalities.
\begin{eqnarray}
g_{k}(\mu_{\dd,k})&=&\zeta=\dd\beta/\alpha;\lab{relation1}\\
p^*&=&\psi(\zeta)=\frac{1}{(\dd-1)(k-1)}.\lab{relation2}
\end{eqnarray}

\fix{The expected change in $L_t$, is closely tied to $\z_t$.  So in order to prove Lemma~\ref{llt1},  we} begin by bounding  $\z_t$ over the next two lemmas. {Recall that $\z_t=D_t/N_t$ denotes the average degree of heavy vertices in $G_t$. We will apply Lemma~\ref{l:monotone} to relate $\bar p_t$ with $\fix{\psi}(\z_t)$ for all $t$ in the range we analyse and therefore we must restrict to a sequence of ``typical'' \fix{hypergraphs} $(G_t)$ \fix{such that $N_t=\Omega(n)$ holds in the whole process}. To formalise the idea, we define: 

\begin{center}
$G_t$ is {\em normal} if $|\bar p_t-\psi(\zeta_t)|\le n^{-1/2}\log^2 n$. 
\end{center}

Note that whether $G_t$ is normal is determined by $\calf_t$ so, equivalently, we could say {\em $\calf_t$ is normal}.
We define the stopping times:

\bea
&&\mbox{$\tau_1$ is the minimum integer $t$ such that $G_t$ is not normal;}\lab{tau1}\\
&&\mbox{$\tau_1=\tau$ if $G_t$ is normal for all $0\le t\le \tau$.}\non
\eea
We will focus on steps $t$ such that $\zeta_t$ is close to $\zeta$. Hence, we will restrict to sequences $(G_t)$ such that $\zeta_t$ gets close to $\zeta$ eventually. Given a constant $\eps>0$, define:
\bea
&&\mbox{$t_{\e}$ is the minimum integer that $||V(G_t)|-\a n|\le\eps n$ and $|\zeta_t-\z|\le \eps$;}\lab{te}\\
&& \mbox{$t_{\e}=\tau$ if such an integer does not exist.}\non
  \eea

\begin{observation}\label{lztz} For all $t_{\e}\leq t\leq t_{\e}+k\eps n$, we have $\z_t=\z +O(\e)$.
\end{observation}

\begin{proof}  Each iteration of SLOW-STRIP removes exactly one hyperedge.  So it reduces each of $D_t$ and $N_t$ by at most $r$.  Since $D_{t_{\e}}\geq r\b n$ and $N_{t_{\e}}\geq \a n$, each step can only change $\z_t=D_t/N_t$ by at most $O(\e)$, and this will be true for at least ${k}\e n$ steps. 
\end{proof}

\begin{lemma}\lab{l:zetaMartingale}
\begin{enumerate}
\item[(a)] There are a sufficiently small constant $\eps=\e(r,k)>0$  and two constants $\rho_1=\rho_1(r,k)>0$ and $\rho_2=\rho_2(r,k)>0$ such that such that  for all $t_{\e}\le t< \min\{t_{\e}+k\eps n,\tau_1\}$, $$
-\frac{\rho_1}{n}\le \ex(\zeta_{t+1}-\zeta_t\mid \calf_t)\le -\frac{\rho_2}{n}.
$$
\item[(b)] Given a constant $\eps>0$, a.a.s.\  there is a large constant $B=B(r,k,\eps)$ for which: if $c=c_{r,k}+n^{-\d}$ then $t_{\e}\le t(B)$. Moreover, a.a.s.\ $\tau_1=\tau$ and $\tau<t_{\e}+k\eps n$.
\end{enumerate}
\end{lemma}

\begin{proof} Let $\eps>0$ be be a small constant whose value is to be determined later. We first prove part (a). We may assume that $t_{\e}<\tau_1$ since otherwise there is nothing to prove.  

 Consider any $t\ge t_{\e}$, and let $v$ be the vertex taken from $\msq_t$ during iteration $t$. {Recall that when SLOW-STRIP runs on the AP-model,} it removes one vertex-copy from $v$ and another $\dd-1$ vertex-copies $u_1,\ldots,u_{\dd-1}$ chosen uniformly at random  from the remaining ones.  We will split the single step into $\dd-1$ substeps $(T_{t,i})_{1\le i\le \dd-1}$, such that $u_i$ is removed in step $T_{t,i}$ for all $1\le i\le \dd-1$ and let $T_{t,0}=t$. We consider the contribution of \fix{deleting} $u_i$, \fix{denoted by $C(u_i)$,} to $\ex(\zeta_{{t+1}}-\zeta_{t} \fix{\mid \calf_t})$.  If $u_{i+1}$ is light  then $\fix{C(u_i)=0}$. If it is heavy (which occurs with probability at least $1/2$ if $\eps$ is sufficiently small) then 
{(extending the definition of $D_t,N_t, \bar p_t$ to $D_{T_{t,i}}, N_{T_{t,i}}, \bar p_{T_{t,i}}$ in the obvious manner):}
\be
\fix{C(u_i)=}\frac{D_{T_{t,i}}-1}{N_{T_{t,i}}}(1-\bar p_{T_{t,i}}) + \frac{D_{T_{t,i}}-k}{N_{T_{t,i}}-1} \bar p_{T_{t,i}}-\frac{D_{T_{t,i}}}{N_{T_{t,i}}}.\lab{expect1}
\ee
\fix{In this case, dropping} the subscript and substituting $1/(N-1)=(1/N)(1+1/N+O(1/n^2))$ (as $N=\Omega(n)$), \fix{we have that uniformly for all $1\le i\le \dd-1$, }
\bea
\fix{C(u_i)}&=&\frac{1}{N}\Big((D-1)(1-\bar p)+\bar p(D-k)(1+1/N+O(n^{-2}))-D\Big)\non\\
&=&\frac{1}{N}\Big(-1 - \bar p(D-1)+\bar p(D+\frac{D}{N}-k+O(n^{-1}))\Big) \non\\
&=&\frac{1}{N}\left(-1+\bar p\left(\frac{D}{N}-(k-1)+O(n^{-1})\right)\right).\lab{expect2}
\eea
By the definition of $t_{\e}$ and applying Observation~\ref{lztz}, we have
$\zeta_t=\dd\beta /\alpha +O(\eps){=\z +O(\eps)}$ for every $t_{\e}\le t\le t_{\e}+k\eps n$. Recall from~\eqn{relation2} that
\[
 p^*=\psi(\zeta)=\frac{1}{(\dd-1)(k-1).}
 \]

 By Lemma~\ref{l:rho}  $\zeta>k$. Since $\z=\z_{r,k}$ is a constant,  Lemma~\ref{l2:monotone} implies that $\psi'(\zeta)<0$ and $\psi'(x)$ is uniformly bounded away from zero in a small neighbourhood of $\zeta$.
 By the definition of $\tau_1$ in~\eqn{tau1}, for all ${t_{\e}}\le t<\min\{t_{\e}+k\eps n,\tau_1\}$ and $0\le i\le \dd-1$, 
\[\bar p_{{T_{t,i}}}=\psi(\zeta_t)+o(1)=\psi(\zeta)+O(\eps)=p^*+O(\eps).
\]
If we were to substitute $\bar p=p^*$, $D=r\beta n$, $N=\alpha n$ into the RHS of (\ref{expect2}), and simplify it, {we would} obtain:

 \bea
\fix{C(u_i)} &=&\inv{\a {n}}\left(-1+\frac{1}{(\dd-1)(k-1)}\left(\frac{\dd\beta}{\alpha}-(k-1)+O(n^{-1})\right)\right)\non\\
 &=& \inv{\a {n}}\left(-1+\frac{1}{(\dd-1)(k-1)}\left(\zeta-(k-1)+O(n^{-1})\right)\right).\lab{expect3}
 \eea	

By Lemma~\ref{l:rho}, there is a $\sigma>0$ such that $\zeta<\dd(k-1)-\sigma$. Hence,
\bea
-1+\frac{1}{(\dd-1)(k-1)}\left(\zeta-(k-1)+O(n^{-1})\right)&<&-1+\frac{1}{(\dd-1)(k-1)}((\dd-1)(k-1)-\sigma/2)\nonumber\\
&<&-\sigma/2(\dd-1)(k-1).\nonumber
\eea
Since in every step, the quantity of variables in~\eqn{expect2} (e.g.\ $\bar p$ and $D/N$) differs from that in~\eqn{expect3} by $O(\eps)$, by choosing $\eps>0$ sufficiently small, there exists {constants $\rho_1',\rho_2'>0$ such that a.a.s.\ }
$$
 {-\frac{\rho_1'}{n}\le }\fix{C(u_i)}\le -\frac{\rho_2'}{n}\quad \mbox{for all}\ 1\le i\le \dd-1,
$$
for all {${t_{\e}}\le t<\min\{t_{\e}+k\eps n,\tau_1\}$}.
Since there are $r-1$ $u_i$'s, part (a) follows with $\r_1=(r-1)\r_1',\r_2=(r-1)\r_2'$.

{Next, we prove part (b).
By Lemma~\ref{l:B}, there is a constant $B=B(r,k,\eps')$, such that  a.a.s.\  $|V(G_{t(B)})\setminus \C_k|\le \eps' n$. By Corollary~\ref{ccoresize}, a.a.s.\ the number of vertices and  hyperedges in $G_{t(B)}$ is $\alpha n+O(\eps' n)$ and $\dd\beta n+O(\eps' n)$ respectively. Therefore, $|\z_{t(B)}-\z|<A\eps'$ for some constant $A>0$. We choose $\eps'$ sufficiently small (correspondingly $B$ sufficiently large), so that $\eps'<\eps$ and $\eps'<\eps/A$. Thus we have $t(B)\ge t_{\e}$. Hence, a.a.s.\ there exists large constant $B=B(\e)$ so that $t_{\e}\le t(B)$. The fact that a.a.s.\ $\tau_1=\tau$ follows from the definition of $\tau_1$ in~\eqn{tau1} and  Lemmas~\ref{l:monotone} \fix{and~\ref{lcoresize2}.}

The fact that a.a.s.\ $\tau\le t_{\e}+k\eps n$ follows by the definition of $t_{\e}$ in~\eqn{te} and Corollary~\ref{ccoresize}: a.a.s.\ there are at most $\eps n +o(n)$ vertices to be deleted from $G_{t_{\e}}$ until SLOW-STRIP terminates, and it takes at most $k-1$ steps to remove each vertex in $\msq_t$.
}

\end{proof}

In the next lemma, we obtain a coarse bound on $\zeta_t$, \fix{which will suffice to prove Lemma~\ref{llt1}.}  We will refine this bound in the later part of this paper \fix{when we require a stronger form of Lemma~\ref{llt1}} (see Lemma~\ref{l:zetaseq}).

\begin{lemma} \lab{l:zetaT} For any sufficiently small constant $\eps>0$, there exists a sufficiently large constant $B$ such that  for every $\eps'>0$,
a.a.s.\ for all $t\ge t(B)$: $\zeta_t\ge \zeta_{\tau}-n^{-1/2+\eps'}$ and {$|\zeta_t-\z|\leq \eps$}.
\end{lemma}

\begin{proof} 
Take a sufficiently small constant $\eps_1$ to satisfy 
Lemma~\ref{l:zetaMartingale}(a); then a sufficiently large constant $B=B(r,k,\eps_1)$ as in Lemma~\ref{l:zetaMartingale}(b). Then, a.a.s.\ 
\be
t_{\e_1}\le t(B)\ \  \mbox{and} \ \ \tau_1=\tau\le t_{\e_1}+k\eps_1 n.\lab{tB}
\ee
 Moreover,  by Lemma~\ref{l:zetaMartingale}(a), for all $t_{\eps_1}\le t\le \tau':=\min\{t_{\eps_1}+k\eps_1 n,\tau_1\}$: $\zeta_t$ is a supermartingale and $\zeta_{t+1}-\zeta_t$ is at most $O(1/n)$. {We couple $\z_t$ with another process $Z_t$ as in the proof of Lemma~\ref{llt2}: for all $t_{\eps_1}\le t\le \tau'$, let $Z_t=\z_t$;
 for all $t\ge \tau'$, let $Z_{t+1}=Z_t$. Now, $(Z_t)$ is a supermartingale for $t\ge t_{\eps_1}$ and $Z_{t+1}=Z_t+O(1/n)$. 
 } By Azuma's inequality, we have that for all {$t_{\eps_1}\le t_2<t_1\le t_{\eps_1}+k\eps_1 n$} and for any $j>0$,
$$
\pr(Z_{t_1}-Z_{t_2}\ge j)\le \exp\left(-\Omega\left(\frac{j^2}{(t_1-t_2)n^{-2}}\right)\right).
$$
{Taking $j=n^{-1/2+\eps'}$ and taking the union bound over all pairs $t_{\eps_1}\le t_1<t_2\le t_{\eps_1}+k\eps_1 n$, the probability that there is a pair $t_2<t_1$ in the above range such that 
$Z_{t_1}-Z_{t_2}>n^{-1/2+\eps'}$ is at most
$n^2\exp\left(-\Omega\left(n^{2\eps'}\right)\right)=o(1)$. Hence, a.a.s.\ $Z_{t_2}\ge Z_{t_1}-n^{-1/2+\eps'}$ for any pair $t_{\eps_1}\le t_1<t_2\le t_{\eps_1}+k\eps_1 n$.

By~\eqn{tB}, a.a.s.\ $t(B)\ge t_{\e_1}$ and $\tau'=\tau_1=\tau$ and so a.a.s.\ $Z_{t}=\z_{t}$ for all $t_{\eps_1}\le t\le \tau$. So, a.a.s.\
$\zeta_t\ge \zeta_{\tau}-n^{-1/2+\eps'}$ for all $t\ge t(B)$.}

To prove that a.a.s.\ $|\z_t-\z|\leq\eps$, we simply choose $\e_1$ to be sufficiently small in terms of $\e$ (and the implicit constant in  Observation~\ref{lztz}).  (\ref{tB}) yields that \aas for all $t(B)\leq t\leq \t$ we have $t_{\e_1}\leq t\leq t_{\e_1}+k\e_1 n$, so Observation~\ref{lztz} yields
\[ |\z_t-\z| \leq O(\e_1)\leq \e. \qedhere\]
\end{proof}

This immediately yields the following corollary.
\begin{corollary}\lab{cor:zetaT} For any sufficiently small constant $\eps>0$, there exist constants $B,K>0$ such that
a.a.s.\  for all $t\ge t(B)$: 
\begin{enumerate}
\item[(a)] $\zeta_t\ge \zeta+Kn^{-\d/2}$;
\item[(b)] $|\zeta_t-\z|\leq \eps$.
\end{enumerate}
\end{corollary}
\begin{proof}

Choose $B$ and $\eps$ to satisfy Lemma~\ref{l:zetaT}. Immediately, we have a.a.s.\ $|\z_t-\z|\leq\eps$ for all $t\ge t(B)$.

Recall the definition of $g_k(x)$ in~\eqn{e.gkx}. By Lemma~\ref{lcoresize2}  and since $\mu_{r,k}>0$, 
a.a.s.\ 
\bean
\zeta_{\tau}&=&\frac{r\b(c)+O(n^{3/4})}{\a(c)+O(n^{3/4})}=g_k(\mu(c))+O(n^{-1/4})\\
&=& g_k(\mu_{r,k})+g'_k(\mu_{r,k})(\mu(c)-\mu_{r,k})+o(\mu(c)-\mu_{r,k})+O(n^{-1/4}),
\eean
By Lemma~\ref{l:gk}, $g'_k(\mu_{r,k})>0$. By Lemma~\ref{l:diff}, $\mu(c)-\mu_{r,k}=K_1n^{-\d/2}+o(n^{\fix{-}\d/2})$ for some constant $K_1>0$. Recall also that  $\zeta=\dd\beta/\alpha=g_k(\mu_{r,k})$.
 So a.a.s.
\be
\zeta_{\tau}= \zeta+\Theta(n^{-\d/2}),\lab{ztau}
 \ee
 as $n^{-1/4}$ is absorbed by $o(n^{-\d/2})$ since $\d<1/2$.
   Now applying Lemma~\ref{l:zetaT} with some $\e'<\inv{4}$ so that $n^{-\hf+\e'}=o(n^{-\d/2})$, we have that a.a.s.\ 
\[
\mbox{for all $t\ge t(B)$, } \zeta_{t}\ge\z_{\t}-n^{-\hf+\e'}\ge \zeta + Kn^{-\d/2},
\] 
for some appropriate constant $K>0$.
\end{proof}

\subsubsection{Proof of Lemma~\ref{llt1}}\lab{sec:lp2}

Recall that we analyze the running of SLOW-STRIP on $AP_r(n,cn)$. In each step $t(B)\le t\le\tau$, the algorithm removes a vertex-copy in a light vertex and another $\dd-1$ vertex-copies $u_1,\ldots,u_{\dd-1}$ chosen uniformly from all remaining ones. For each $1\le i\le \dd-1$, let $h_{t,i}$ denote the probability that $u_i$ is light. Recall that $\bar p_t$ is the probability that a heavy vertex-copy is allocated to a bin of size $k$.  Then
\begin{equation}
\ex(L_{t+1}-L_t\mid \calf_t)=-1+\sum_{i=1}^{\dd-1} \Big(-h_{t,i}+(1-h_{t,i})(k-1) \bar p_t+O(n^{-1})\Big),\lab{barpt}
\end{equation}
where $O(n^{-1})$ accounts for the change of $\bar p_t$ caused by the removal of the first $i-1$ points and the possibility that we select two copies of the same vertex. \eqn{barpt} is maximized when $h_{t,i}=0$ for all $1\le i\le \dd-1$. Thus,
\begin{equation}
\ex(L_{t+1}-L_t\mid \calf_t)\le -1 + (\dd-1)(k-1)\bar p_t+O(n^{-1}).\lab{barpt2}
\end{equation}

As in the proof of Lemma~\ref{l:zetaMartingale}, we need a relation between $\bar p_t$ and $\zeta_t$ and thus it is convenient to restrict to steps $t< \tau_1$ where $\tau_1$ is defined in~\eqn{tau1}. Moreover, we want to restrict to sequences $(G_t)$ such that Corollary~\ref{cor:zetaT}(a,b) hold. To formalise the idea, we first choose constants $B$, $\eps$ and $K$ to satisfy Corollary~\ref{cor:zetaT} (Note that $\eps$ can be chosen arbitrarily small which results in larger $B$).  Then, we define:
\bea
&&\mbox{$\tau_2$ is the minimum $t\ge t(B)$ such that $\zeta_{t}<\zeta+Kn^{-\d/2}$ or $|\z_{t}-\z|>\eps$;}\lab{tau2}\\
&&\mbox{$\tau_2=\tau$ if no such integer exists.}\non
\eea
Define
\be
\tau^*=\min\{\tau_1,\tau_2\}. \lab{taustar}
\ee
By Lemma~\ref{l:zetaMartingale}(b) and Corollary~\ref{cor:zetaT}, a.a.s.\ $\tau^*=\tau$.

By the definition of $\tau_1$ in~\eqn{tau1} and noting that $\tau^*\le \tau_1$, we have for all $t(B)\le t< \tau^*$: 
\be
\bar p_t=\psi(\z_t)+n^{-1/2}\log^2 n=\psi(\z)+\psi'(\z)(\z_t-\z)+O((\z_t-\z)^2)+n^{-1/2}\log^2 n, \lab{et*}
\ee
{by expanding $\psi(\z_t)$ at $\z$.}

By Lemmas~\ref{l2:monotone} and~\ref{l:rho}, and since $\z=\z(r,k)$ is a constant, we have $\psi'(\z)<-C$ for some constant $C>0$. By the definition of $\tau_2$ and $\tau^*$,
for all $t(B)\le t<\tau^*$, $\z_t-\z$ can be assumed sufficiently small (by choosing sufficiently small $\eps$) so that
\[
\psi'(\z)(\z_t-\z)+O((\z_t-\z)^2)\le -\frac{C}{2}(\z_t-\z).
\]

By~\eqn{relation2}, $\psi(\z)=\frac{1}{(\dd-1)(k-1)}$.
By Corollary~\ref{cor:zetaT}, there is a constant $K_1>0$ such that a.a.s.\ for all $t(B)\le t< \tau^*$, $\z_t\ge \z+K_1n^{-\d/2}$. Putting all this together yields:

\bea
\bar p_t&\le& \frac{1}{(\dd-1)(k-1)}-\frac{C}{2}K_1n^{-\d/2}+n^{-1/2}\log^2 n\non\\
&\le &\frac{1}{(\dd-1)(k-1)}-K_2n^{-\d/2},\lab{br1}
\eea
for any constant $K_2<K_1C/2$, as $\d/2 < \inv{2}$.

It follows then from (\ref{barpt2}) that there is a $K>0$ such that for all $t(B)\le t<\tau^*$,
\be
\ex(L_{t+1}-L_t\mid \calf_t)\le -Kn^{-\d/2}.\lab{brr}
\ee
By~(\ref{et*}), this holds \aas\ for all $t(B)\leq t<\tau^*$, and by~(\ref{taustar}) \aas\ $\tau=\tau^*$; this proves the lemma.\qed

\subsection{Proof of Theorem~\ref{mt}(b)} \label{smtb}

In this section, we prove Theorem~\ref{mt}(b). So $|c-c_{r,k}|\leq n^{-\d}$ for some $0<\d<\hf$. Our goal is to show that \aas\ the stripping number of $H\in \calh_{\dd}(n,cn)$ is $\Omega(n^{\d/2})$.
{Similarly to the supercritical case, we will analyse the process on $H\in AP_r(n,cn)$ instead.}

We will define $\d'\approx\d$ so that $c$ is not very far from $c'=c_{r,k}+n^{-\d'}$ and then argue
that the stripping number of $H$ is not much smaller than the stripping number of $H'={AP_{\dd}(n,c'n)}$.  We have already proven that the latter stripping number is $\Omega(n^{\d'/2})$ in Section~\ref{s.mtlb}. 

We specify a small $\e>0$ and define $\d'$ such that:
\[
n^{-\d'}=\left\{
\begin{array}{ll}
n^{-\d} & \mbox{if}\ c= c_{r,k}-n^{-\d}\ \mbox{where}\ \d\le 1/2-\e\\
2n^{-1/2 +\e} & \mbox{if} \ |c-c_{r,k}|< n^{-1/2 + \e}.
\end{array}
\right.
\]

\noindent {\bf Remark} So if $|c-c_{r,k}|<n^{-1/2 + \e}$, then $\d'$ is a function of $n$, rather than a constant.  We can still apply the results of Section~\ref{s.mtlb} to $H'$ as the proofs work even for non-constant $\d$, see the remark following the statement of Theorem~\ref{mt}.

It follows then that in both cases,
\be
\frac{1}{2}n^{-\d'}\le n^{-\d'}-(c-c_{r,k})\le 2n^{-\d'}. \lab{cdiff}
\ee
To prove Theorem~\ref{mt} for these ranges of $c$, it suffices to show that a.a.s.\
$s(\calh_r(n,cn))=\Omega(n^{\d'/2})$ with $\d'$ defined above, since $n^{\d'/2}={\Omega}(n^{\d/2})$.

Let $H'={AP}_{\dd}(n,c'n)$ where $c'=c_{r,k}+n^{-\d'}$, and we generate $H$ by removing $\fix{(c'-c)n}$ edges ({i.e.\ $r$-tuples}), {chosen uniformly at random} in $H'$. This way we couple $H\subseteq H'$ and $H$ is distributed as ${AP}_{r}(n,cn)$. Consider the following stripping procedure to find the $k$-core of $H$:

\begin{enumerate}
\item Run Phase 1 of SLOW-STRIP on $H'$, thus obtaining $G'_{t_0}\subseteq H'$.
\item For each vertex $v$ removed from $H'$ in Step 1, we also remove $v$ from $H$.
We call the remaining hypergraph $G_{t_0}$.
\item Run SLOW-STRIP on $G_{t_0}$.
\end{enumerate}

To be clear:  $\g$ and $t_0$ are defined for {the process of running SLOW-STRIP on} $H'$, not for $H$.  So:
\[\g=3\fix{K_2}n^{1-\d'/2},\]
and $t_0$ is the first step at which exactly $\fix{\b} n+\g$ \fix{hyperedges} remain in $H'$.

Note that since $H\subseteq H'$, the $k$-core of $H$ is contained in the $k$-core of $H'$, and so the $k$-core of $H$ is contained in $G'_{t_0}$.  $G_{t_0}$ contains every vertex of $G'_{t_0}$ and contains every edge of $G'_{t_0}$ that is an edge of $H$. So  the $k$-core of $H$ is contained in $G_{t_0}$. Thus, this is a valid way to obtain the $k$-core of $H$.

Note also that this is not equivalent to running SLOW-STRIP on $H$, since doing so could remove 
the vertices in a different order.  Nevertheless,
we still have $G_{t_0}\subseteq H$ and so the stripping number of $H$ is at least  the stripping number of $G_{t_0}$.

The number of hyperedges in $H'$ but not in $H$ is \fix{$(c'-c)n$, which is $\Theta(n^{1-\d'})$ by~\eqn{cdiff}}.  
We use $G_{i}$ to denote the subgraph of $H$ remaining after $i-t_0$ iterations of SLOW-STRIP on $G_{t_0}$.  We define $L_i, D_i, N_i$ as we did in Sections~\ref{sec:deg} and~\ref{sec:llt2}. Similarly we define $L'_{t_0}, D'_{t_0}, N'_{t_0}$ to be the values of the same parameters for $G'_{t_0}$.

$G'_{t_0}$ is the result of carrying out Phase 1 of SLOW-STRIP on $H'=AP_{\dd}(n,c'n)$. So \aas $L'_{t_0}, D'_{t_0}, N'_{t_0}$ satisfy the bounds in (\ref{vertex-edge}) and (\ref{L0b}). 
 Clearly, $|L_{t_0}-L'_{t_0}|, |D_{t_0}-D'_{t_0}|, |N_{t_0}-N'_{t_0}|$ are bounded by $r$ times the number of edges in $E(G'_{t_0})\bk E(G_{t_0})$. {Note that $|E(G'_{t_0})\bk E(G_{t_0})|$} is at most the number of edges in $E(H')\bk E(H)$, {which is $\fix{\Theta(n^{1-\d'})}=o(\g)$}. Combining this with the bounds (\ref{vertex-edge}) and (\ref{L0b}) on $L'_{t_0}, D'_{t_0}, N'_{t_0}$, we have:
\begin{eqnarray*}
N_{t_0}&=&N'_{t_0}+o(\g)=\a n+O(\g)\\
D_{t_0}&=&D'_{t_0}+o(\g)=r\b n+O(\g)\\
L_{t_0}&=&L'_{t_0}+O(\g^2/n)=O(\g^2/n)
\end{eqnarray*}

Next we prove that a.a.s.\ SLOW-STRIP applied to $G_{t_0}$ lasts for at least $\g/4$ steps. Applying Proposition~\ref{p:tau} to $H'$ tells us that applying SLOW-STRIP to $G'_{t_0}$
would take  at least $\g/3$ iterations.  Each iteration removes one hyperedge, and the removed hyperedge is not in the $k$-core of $G'_{t_0}$ and hence not in the $k$-core of $G_{t_0}$.  We have shown that at most $\fix{\Theta(n^{1-\d'})}=o(\g)$ of those  hyperedges are not in $G_{t_0}$. Therefore,
SLOW-STRIP takes at least $\g/3-o(\g)>\g/4$ iterations on $G_{t_0}$ {(c.f.\ Lemma~\ref{lp0})}.  
{It is then convenient to define $\tau'=\min\{ t_0+\g/4,\tau\}$ and we have just shown that a.a.s.\ $\tau\ge \tau'$.} 

Corollary~\ref{ccoresize} implies that \aas\ $G'_{t_0}$ is nice.  Since $G_{t_0}$ contains all but  $o(\g)$ hyperedges of $G'_{t_0}$, \aas\ $G_{t_0}$ is also nice.
Thus, the analysis of Section~\ref{sec:llt2} applies to the running of SLOW-STRIP on $G_{t_0}$. The analysis is similar except that $\d$ is replaced by $\d'$, and yields the conclusion of Lemma~\ref{llt3}; i.e.\ for every $ t_0\leq t\leq \tau'$,
\[\ex(L_{t+1}{\mid \calf_t})= L_t \pm O(n^{-\d'/2}).\]
  
The  same analysis as in the proof of Lemma~\ref{llt2} yields that (\ref{Lt}) holds; i.e. \aas for every $ t_0\leq t\leq {\tau'}$,
\begin{equation}\lab{ecoup1} L_{t}\leq  L_{t_0}+O(\g^2/ n)=O(n^{1-\d'}).
\end{equation}

The rest of the proof follows as in Section~\ref{s.mtlb}. We argued above that
SLOW-STRIP takes at least $\g/3-o(\g)>\g/4$ iterations on $G_{t_0}$ ({thus, a.a.s.\ $\tau'=t_0+\g/4$}).  It takes at most $k-1$ iterations to remove a vertex, and so the parallel stripping process, applied to $G_{t_0}$ removes at least $\g/(4(k-1))$ vertices.  By~(\ref{ecoup1}) each iteration $i$ removes  $|L_{t(i)}|=O(n^{1-\d'})$ vertices. So there must be at least $\g/O(n^{1-\d'})=\Omega(n^{\d'/2})$ ({recalling that $\gamma=\Theta(n^{1-\d'/2})$}) iterations of the parallel stripping process. I.e. the stripping number of $G_{t_0}$ is at least $\Omega(n^{\d'/2})$ and hence the stripping number of $H$ is also $\Omega(n^{\d'/2})=\Omega(n^{\d/2})$.

This proves that the lemma holds for $AP_r(n,cn)$.  Corollary~\ref{ccon0} implies that it also holds for $\calh_r(n,cn)$.
\proofend

\section{Bounding the maximum depth: proof of Theorem~\ref{mt2}}\lab{smd}

We first note that the stripping number provides a lower bound on the maximum depth over all non-$k$-core vertices.
 \begin{lemma} \lab{ls2} For any vertex $v\in S_i$, the depth of $v$ is at least $i$.
\end{lemma}

\proofstart  We prove by induction that every vertex $v\in\cup_{j\geq i} S_j$ has depth at least $i$.  This is trivial for $i=1$. Suppose it is true for $i$, and consider $v\in\cup_{j\geq i+1} S_j$.
Since $v\notin S_i$, $v$ has at least $k$ neighbours which are either in the $k$-core or in $\cup_{j\geq i} S_j$.  At least one of those neighbours must be removed before $v$ can be removed, and by our induction hypothesis, each such neighbour has depth at least $i$.  So any stripping sequence which removes $v$ must first include a sequence of length at least $i$ which removes a neighbour of $v$. Thus it must have length at least $i+1$; i.e. the depth of $v$ is at least $i+1$.
\proofend

Therefore, Theorem~\ref{mt}(a) provides the lower bound of Theorem~\ref{mt2}.  We will focus on the upper bound.

Recall that $S_i$ is the set of vertices removed during iteration $i$ of the parallel stripping process, and $\hG_i$ is the subhypergraph remaining after $S_1,...,S_{i-1}$ are removed. 

\fix{Define:
\[\imax \mbox{ is the number of iterations carried out by the parallel stripping process.}\]
}
We define the following hypergraph formed by the vertices of $S_i$:
\begin{definition}\label{dsi} The vertices of $\cals_i$ are $S_i$. For any hyperedge $f$ in $\hG_i$ that includes at least one vertex of $S_i$, $f'=f\cap S_i$ is a hyperedge of $\cals_i$. If $|f'|=a$ then $f'$ is said to be an {\em $a$-edge}.
\end{definition}
Note that $\cals_i$ may contain hyperedges of size one  (in fact, for large $i$, most of the edges will have size one).

We wish to bound the depth of a non-$k$-core  vertex $v$. We begin by defining a set $R(v)$ that contains such a stripping sequence.

\begin{definition}\lab{drv}  For each $1\leq i\leq \imax$ and any $v\in S_i$, we set $R'_i=R'_i(v):=\{v\}$ and for each $j = i$ to 1:
\begin{enumerate}
\item[(a)] we set $R_j=R_j(v)$ to be the union of the vertex sets of all components of $\cals_j$ that contain vertices of $R'_j$.
\item[(b)]  we set $R'_{j-1}$ to be the set of all vertices $v\in S_{j-1}$ that are adjacent to $\cup_{\ell=i}^{j} R_{\ell}$.
\end{enumerate}
We define $R(v)=\cup_{\ell=i}^{1} R_{\ell}$.
\end{definition}

\no {\bf Remark.}
For the purposes of this paper, we could have omitted step (a) and replaced $R'$ by $R$ elsewhere.  The set $R(v)$ would still have contained a stripping sequence leading to the removal of $v$.  We define, and bound, this larger set for our application to clusters in random XOR-SAT in another paper\cite{gmxor}. (A preliminary version is in~\cite{gmarxiv}).

\begin{observation}\label{ossr}
$R(v)$ contains a stripping sequence ending with $v$.
\end{observation}

\proofstart We prove this for each $v\in S_i$ using a simple induction on $i$.  For $i=1$, $v$ is a stripping sequence of length 1.  For $i\geq 2$, note that $R(v)$ contains every neighbour of $v$ lying in levels $1,...,i-1$; call these neighbours $u_1,...,u_q$.  The recursive construction
ensures that $R(v)$ also contains $R(u_1),...,R(u_q)$ and so, by induction, contains stripping sequences ending with $u_1,...,u_q$.  After the deletion of $u_1,...,u_q$, the degree of $v$ drops below $r$ (since $v\in S_i$), and so adding $v$ to the concatenation of those stripping sequences produces a stripping sequence ending in $v$.
\proofend

\begin{theorem}\lab{mt3}
{Let $r,k\ge 2$, $(r,k)\neq (2,2)$ be fixed. There is a constant $\kappa=\kappa(r,k)$ such that: for any $0<\d<1/2$, if $c=c_{r,k}+n^{-\d}$, then a.a.s.\ for every $v\in\H_r(n,c n)$:} $|R(v)|\le n^{\kappa\d}$.
\end{theorem}

We will prove the upper bound in Theorem~\ref{mt3} by showing that a.a.s.\ $|R(v)|=n^{O(\d)}$ for every non-$k$-core vertex $v$. {Here and in the following sections, the implicit constant involved in $O(.)$ is always independent of $\d$.} Theorem~\ref{mt2} follows immediately from Theorem~\ref{mt3} and Observation~\ref{ossr}. 

\fix{We close this section with an  overview of the proof for Theorem~\ref{mt3}. One natural approach to bound $R(v)$ is to: (a) carry out the stripping process until $v$ is removed in some iteration $i$, then (b) explore $R(v)$ starting with $R_i(v)$ and working through $R_{i-1}(v), R_{i-2}(v),...R_1(v)$ according to Definition~\ref{drv}.  However, after exposing all the edges deleted in iterations $1,...,i$ in part (a), we  have no randomness left to facilitate the analysis of part (b).   

To overcome this difficulty, we only expose a minimal amount of information in part (a).  We expose the vertices of each $S_i$; some information about the deleted edges, such as the number of vertices an edge has in $S_i$ and in $S_{i+1}$; and some degree information. But crucially we do not expose the actual vertices in each deleted edge.  When we carry out part (b), we expose the vertices of the edges relevant to $R(v)$.  This is enough randomness for us to bound $|R(v)|$.

To carry out part (a), we complete the entire stripping process using a procedure called  {\em EXPOSURE}.  The parameters we expose in this phase are listed at the beginning of the next subsection.  Lemma~\ref{lsi} bounds these parameters.  
We then expose the vertices in all deleted edges using a procedure called {\em EDGE-SELECTION}.  We use this procedure to expose the vertices in $R_i(v),R_{i-1}(v),...R_1(v)$ and obtain a recursive bound on $|R_j(v)|$ in terms of $R_{j+1}(v),...,R_i(v)$; see Section~\ref{s.arb} and its proof in Section~\ref{sec:SS1}.  

}


\subsection{Vertex-exposure and edge-selection}
\lab{sec:exposure}

As in Section~\ref{smt1}, we will analyze the running of the parallel stripping process on the AP-model. We will work with the AP-model for the remainder of Section~\ref{smd}. Again, when we use a graph theoretic term, we mean the obvious analogue for a configuration.
\fix{Recall}:
\[\imax \mbox{ is the number of iterations carried out by the parallel stripping process.}\]

 We parameterize the hyperedges removed during iteration $i$ as follows:
\begin{definition} A hyperedge removed during iteration $i$ is called an $(a,b)$-edge, where $a$ is the number of vertex-copies it contains from $S_i$ and $b$ is the number of vertex-copies it contains from $S_{i+1}$. Thus $a\geq 1$ and $0\leq b\leq r-a$ and if $i=\imax$ then we must have $b=0$.
\end{definition}

For every $i,1\leq a\leq r,0\leq b\leq r-a$ we define:
\[
 M_i^{a,b} \mbox{ is the number of $(a,b)$-edges in } \hG_i.\]

\begin{definition}
\fix{
\begin{itemize} For each $v\in S_i$:
\item  $d^+(v)$ is the degree of $v$ in $\hG_i$;  i.e. the number of vertex-copies of $v$ that are removed during iteration $i$.
\item (if $i<\imax$), $d^-(v)$ is the total degree of $v$ amongst the  hyperedges in $\hG_{i-1}$ that contain at least one vertex of $S_{i-1}$; i.e. the number of vertex-copies of $v$ in hyperedges  that are removed during iteration $i-1$.
\end{itemize}
}
\end{definition}

\fix{As described above,} we expose the hyperedges removed during the parallel stripping process in two phases.  First, we expose the vertices that are removed in each iteration, along with some degree and edge-count information:

\smallskip

{\bf EXPOSURE:}
\begin{enumerate}
\item Expose $\imax$.
\item Expose the vertices in $S_1,...,S_{\imax}$.
\item For each $0\leq i\leq \imax$ and each vertex $v\in S_i,u\in S_{i+1}$, expose $\fix{d^+}(v),d^-(u)$.
\item For each $0\leq i\leq \imax, 1\leq a\leq r,0\leq b\leq r-a$, expose $M_i^{a,b}$. 
\end{enumerate}
Of course, this also exposes the vertices and the number of edges in the $k$-core $\calc_k=\calc_k(H)$.

To clarify what is exposed in terms of the AP-model: We have exposed the bins (vertices) in  $S_1,...,S_{\imax}$.  For each bin $v\in S_i$, we have exposed the number of copies of $v$ that are removed in iteration $i-1$ and in iteration $i$ {respectively}. In fact, \fix{for convenience we} expose those actual copies. Other vertex-copies may still be allocated to bin $v$, but any such vertex-copy must be deleted during iterations $1,...,i-2$. For each relevant $i,a,b$, we have exposed the number of $r$-tuples removed in iteration $i$ that contain $a$ vertex-copies from $S_i$ and $b$ vertex-copies from $S_{i+1}$; but for each such $r$-tuple, we do not expose the actual bins (vertices) those vertex-copies lie in, nor do we expose the remaining $r-a-b$ vertex-copies in the $r$-tuple.

Next we expose the actual hyperedges (i.e.\ $r$-tuples of vertex-copies which are allocated to bins) that are removed during each iteration. We define:
\[\cale_i \mbox{ is the set of hyperedges removed during iteration } i.\]
$\cale_i$  must satisfy the following conditions for each $1\leq i\leq \imax$:

\smallskip

\noindent (P1) Each $v\in S_i$ has exactly $\fix{d^+}(v)$ vertex-copies in $\cale_i$.

\noindent (P2) Each $u\in S_{i+1}$ has exactly $d^-(u)$ vertex-copies in $\cale_i$.

\noindent (P3) $\cale_i$ contains exactly $M_i^{a,b}$ $(a,b)$-edges for each $a,b$.

\noindent (P4) In each of those $M_i^{a,b}$ $(a,b)$-edges,  the $r-a-b$ vertex-copies not in $S_i\cup S_{i+1}$ must be allocated to a bin in $\calc_k\bigcup_{j=i+2}^{\imax}S_j$.

\smallskip

\begin{lemma}\lab{lem:edgeSelect} Conditional on the outcome of EXPOSURE, any set of hyperedges ($r$-tuples) satisfying properties (P1--P3)  is equally likely to be $\cale_i$.
\end{lemma}

\proofstart  Fix all parameters that are  exposed in EXPOSURE (i.e.\ $\imax$, $S_1,\ldots, S_{\imax}$, $\fix{d^+}(v)$, etc). Consider two sets of hyperedges (i.e.\ $r$-tuples of vertex-copies, and an allocation of those vertex-copies into bins) $E_1,E_2$ satisfying properties (P1--P4).  We need to show that $E_1,E_2$ are equally likely to be $\cale_i$.

Recall that a configuration is a partition of the $rm$ vertex-copies into $r$-tuples, and an allocation of the $rm$ vertex-copies into the $n$ bins.
 Let $H_1$ be a configuration that is consistent with the  outcome of EXPOSURE (i.e.\ applying the parallel stripping process to $H_1$ will result in the various parameters being equal to what was exposed in EXPOSURE), and such that applying the parallel stripping process to $H_1$ will result in $\cale_i=E_1$.  Let $H_2$ be the hypergraph obtained from $H_1$ by replacing $E_1$ with $E_2$.  

We claim that: $H_2$ is also consistent with the outcome of EXPOSURE, and  that applying the parallel stripping process to $H_2$ will result in $\cale_i=E_2$. To verify this claim, we only need to show that, when we apply the stripping process to $H_2$, we take $\imax$ iterations, and for each $1\leq j\leq \imax$, the set of vertices removed in iteration $j$ is $S_j$ (the fact that $E_2$ satisfies (P1--P4) confirms the remainder of the claim.)  This follows easily from the fact that $E_1,E_2$ do not include any hyperedges that contain vertices from $S_1,...,S_{i-1}$ and that $E_1,E_2$ each consist of all the hyperedges that contain vertices from $S_i$ in what remains after removing $S_1,...,S_{i-1}$ from $H_1,H_2$, resp.

So we have  a bijection between the configurations which yield $\cale_i=E_1$ and $\cale_i=E_2$. Furthermore, every configuration is equally likely to be chosen as $AP_{\dd}(n,cn)$. Thus $E_1$ and $E_2$ are equally likely to be $\cale_i$, which establishes the lemma.
\proofend\smallskip

So our goal is to choose a uniform set of $r$-tuples satisfying (P1--P4), for each $1\leq i\leq\imax$.

{\bf EDGE-SELECTION:} For each $1\leq i\leq \imax$, we expose $\cale_i$, the hyperedges deleted in iteration $i$: We have already exposed $M_i^{a,b}$, the number of $(a,b)$-edges.  Each such edge will be assigned $a$ vertex-copies from $S_i$, $b$ vertex-copies from $S_{i+1}$, and $r-a-b$ vertex-copies from  $\calc_k\bigcup_{j=i+2}^{\imax}S_j$.  We assign those copies as follows:
\begin{enumerate}
\item Each vertex $v\in S_{i}$ has $\fix{d^+}(v)$ copies.  Noting that $\sum_{v\in S_{i}} \fix{d^+}(v)=\sum_{a,b} aM_i^{a,b}$, we take a uniformly random partition of all vertex-copies in $S_{i}$ so that each $(a,b)$-edge receives a part of size $a$.
\item Each vertex $u\in S_{i+1}$ has $d^-(u)$ copies.  Noting that $\sum_{u\in S_{i+1}} d^-(v)=\sum_{a,b} bM_i^{a,b}$, we take a uniformly random partition of all vertex-copies in $S_{i+1}$ so that each $(a,b)$-edge receives a part of size $b$.
\item For each $(a,b)$-edge, we choose the remaining $r-a-b$ vertex-copies uniformly from those that have not yet been allocated to bins.  Then we allocate each of these vertex-copies to a bin selected uniformly from $\calc_k\bigcup_{j=i+2}^{\imax}S_j$.
\end{enumerate}

Any two sets of hyperedges  (i.e.\ $r$-tuples of vertex-copies which are allocated to bins) satisfying (P1--P4) are equally likely to be chosen by EDGE-SELECTION.  Therefore, if we carry out EDGE-SELECTION, then by Lemma~\ref{lem:edgeSelect}, we have chosen the hyperedges from the correct distribution.

\subsection{A recursive bound}\label{s.arb}

\fix{Consider some vertex $v\in S_i$ and recall Definition~\ref{drv} where we define $R(v)=R_i(v)\cup R_{i-1}(v)\cup ...\cup R_1(v)$.  That definition naturally lends itself to a recursive bound of $|R_j(v)|$ in terms of $R_{j-1}(v),...,R_i(v)$.  We present that definition in this subsection, prove it in Section~\ref{sec:SS1}, and then use it to bound $|R(v)|$ in Section~\ref{sec:rec2}.  It will be convenient to restrict our analysis to deletion rounds which are late enough that certain bounds on various parameters hold.  So we} let $B=B(r,k)$ be a sufficiently large constant to be named later (it will come from Lemma~\ref{lsi} below) and set
\be
R^{(B)}(v)=\cup_{j=i}^B R_j= R(v)\setminus (\cup_{1\le i<B} S_i)\lab{RB}
\ee
\fix{Lemma~\ref{lem:neighbours}  bounds $|R(v)| =O(|R^{(B)}(v)|+\log n)$ and so bounding $R^{(B)}(v)$ will suffice to 
prove Theorem~\ref{mt3}.} 

The expected size of $R_{j-1}$ depends not just on the size of $R_j$, but also on the $d^-$ values of the vertices in $R_j$. So we will recursively bound the sum of those values, rather than bound the (nearly equal)  $|R_j|$:

\begin{definition}\label{d.d-}  For
$X\subseteq S_i$, we define
\[D^-(X)=\sum_{u\in X} d^-(u).\]
\end{definition}

Note that if $i>1$ then $ d^-(u)\geq 1$ for all $u\in S_i$ as otherwise, $u$ would have been deleted in iteration $i-1$, and so \fix{$D^-(R_j(u))\geq|R_j(u)|$ for all $j\le i$}. Therefore, bounding $D^-(R_j\fix{(u)})$ will suffice to bound $|R_j\fix{(u)}|$.

\fix{Given our} non-$k$-core vertex $v$, we define:
\[\mbox{$I_v$ is the integer $i$ such that $v\in S_i$; i.e.\ the iteration during which $v$ is removed.}\]

\fix{As described above, we are restricting our analysis to vertices deleted after iteration $B$, so we will assume that $I_v\geq B$. We will} bound $D^{-}(R^{\fix{(B)}}(v))$. The recursion starts with a base case bound on $D^-(R_{I_v}(v))$, and then bounds $D^-(R_{I_v-1}(v))$, $D^-(R_{I_v-2}(v))$,...,$\fix{D^-(R_{B}(v))}$.  Thus, we express $D^-(R_j(v))$ in terms of $D^-(R_{\ell}(v))$ for $\ell>j$, rather than the usual $\ell<j$.

\begin{lemma}\label{l.rj} There are constants $B=B(r,k),Z=Z(r,k)>0$ such that with probability at least
$1-n^{-3}$: for all {$B\le j\le I_v$}  with $|S_j|\ge n^{\d}\log^2n$, we have
\be
D^-(R_j(v))\leq D^-(R_{j+1}(v))+Z\frac{|S_j|}{n}\sum_{\ell=I_v}^{j+1}D^-(R_{\ell}(v))
+\log^{{14}} n.\lab{recMain}
\ee
\end{lemma}

We will analyze the recursive equation in Lemma~\ref{l.rj} to obtain a bound on $D^-(R_j)$ for each $j$, in Section~\ref{sec:rec2}.

We prove~(\ref{recMain}) by analyzing the exposure process from Section~\ref{sec:exposure}.  We begin by analyzing EXPOSURE: we bound $\imax$ and prove that the parameters $\fix{d^+}(v),d^-(v),M^{a,b}_i$ satisfy certain properties (Sections~\ref{sec:Si} and~\ref{slsi}).  Next, we analyze
EDGE-SELECTION one iteration at a time, with $i$ decreasing from $\imax$. When analyzing $R_j$ to prove~(\ref{recMain}), iterations $i=\imax,\ldots,j+1$ have already been completed. So for $v\in S_{I_v}$, we have already exposed $R_{I_v}, R_{I_v-1},...,R_{j+1}$ .

\subsection{Properties of $S_i$}
\lab{sec:Si}

We first run EXPOSURE. We prove that \aas the sets $S_i$ satisfy certain properties.

\begin{lemma}\lab{lsi} There exist constants $B,Y_1,Y_2,Z_1$, {dependent only on $r,k$,} such that \aas for every $ B\leq i< \imax$ with $|S_i|\ge n^{\d}\log^2 n$: 
\begin{enumerate}
\item[(a)] if $|S_i|< n^{1-\d}$ then $(1-Y_1n^{-\d/2})|S_i|\leq |S_{i+1}|\leq  (1-Y_2n^{-\d/2})|S_i|$;
\item[(b)] if $|S_i|\geq n^{1-\d}$ then $(1-Y_1\sqrt{\frac{|S_i|}{n}})|S_i|\leq |S_{i+1}|\leq  (1-Y_2\sqrt{\frac{|S_i|}{n}})|S_i|$;
\item[(c)] $\sum_{j\ge i}|S_j|\le {Z_1}|S_i|n^{\d/2}$.
\item[(d)] $|S_i|\leq\sum_{u\in S_{i}} d^-(u) < |S_{i}|+Z_1\frac{|S_{i}|^2}{n}+\log^2n$;
\item[(e)] $|S_{i+1}|\leq \sum_{a,b}abM_i^{a,b}\leq |S_{i+1}|+Z_1\frac{|S_{i}|^2}{n}+\log^2n$; 
\item[(f)]$\sum_{a\geq 2,b\leq r-a}abM_i^{a,b}\leq Z_1\frac{|S_{i}|^2}{n}+\log^2n$;   
\item[(g)] $\sum_{u\in S_i}\fix{d^+}(u)d^-(u)\leq\sum_{u\in S_i}\fix{d^+}(u)+Z_1\frac{|S_{i}|^2}{n}+\log^2n$;
\item[(h)] $\sum_{u\in S_{i}} (d^-(u))^2\le Z_1|S_i|$;
\item[(i)] $d^-(u)<\log n$ for all $u\in\cup_{i=2}^{\imax}S_i$.
\end{enumerate}
\end{lemma}

The proof is deferred until Section~\ref{slsi}.

\subsection{Proving the recursive bound}
\lab{sec:SS1}

Here, we will prove the recursive bound from Section~\ref{s.arb}.
Recall that we specified a vertex $v$, and $I_v\le \imax$ is the iteration during which $v$ was deleted.  

Our goal is to prove Lemma~\ref{l.rj}, and so we can assume  $j\geq B$ and $|S_j|>n^{\d}\log^2n$.  For each $j$, we will prove the probability that  (\ref{recMain}) fails is less than $n^{-4}$, and so a union-bound implies that the probability it holds for every $j$ is at least $1-I_vn^{-4}>1-n^{-3}$.

The setting is that
we have carried out EXPOSURE, and we assume that all properties in Lemma~\ref{lsi} hold. We have also carried out EDGE-SELECTION for iterations $\imax,...,j+1$, and in particular, we have exposed $R_{I_v},...,R_{j+1}$.

The random experiment in this section is  iteration $j$ of EDGE-SELECTION.  At this point, we know that there are $M_{j}^{a,b}$ $(a,b)$-edges for each $a\geq 1, 0\leq  b\leq r-a-b$, but we don't know which vertices are in these edges.  Each $(a,b)$-edge consists of $r$ {\em blanks} which will be filled with vertex-copies: the $a$ {\em level-$j$ blanks} will receive vertex-copies from $S_{j}$; the $b$ {\em level-$(j+1)$ blanks} will receive vertex-copies from $S_{j+1}$; the remaining $r-a-b$ blanks will receive vertex-copies which will then be allocated to bins in $\calc_k\bigcup_{i=j+2}^{\imax}S_{i}$.

We fill in these blanks with three independent steps, in reverse order of how they are listed above.

{\bf Step 1:   For each $(a,b)$-edge, we choose the $r-a-b$  vertex-copies and then allocate them to $\calc_k\bigcup_{i=j+2}^{\imax}S_{i}$.} 

In order to bound $D^-(R_j)$, we will first bound 
\bea\lab{Ydef}
\blank: && \mbox{the total number of level-$j$ blanks, amongst all hyperedges}\non\\
&& \mbox{deleted in iteration $j$ that contain  a vertex-copy from $\cup_{i=I_v}^{j+1} R_i$.} 
\eea
 Note that $\blank$ is an upper bound on $|R'_j|$. 

We let $\blank_1$ denote the number of these level-$j$ blanks whose hyperedge contains a  vertex from $\cup_{i=I_v}^{j+2} R_i$.  Note that $\blank_1$ is determined by Step 1: After allocating these vertex-copies to $\calc_k\bigcup_{i=j+2}^{\imax}S_{i}$, we know exactly which $(a,b)$-edges contain vertex-copies that are allocated to $\cup_{i=I_v}^{j+2} R_i$. 

Each time we allocate one of these copies, the probability that we allocate it to a bin from
$\cup_{i=I_v}^{j+2} R_i$ is at most
\[\frac{|\cup_{i=I_v}^{j+2} R_i|}{|\calc_k\bigcup_{i=j+2}^{\imax}S_i|}
\le \inv{\a(c) n+o(n)}\sum_{i=I_v}^{j+2} |R_i|,\qquad\mbox{where $\a(c)$ is from Lemma~\ref{lcoresize2}}.\]
There are at most $(k-1)|S_j|$ such hyperedges, and for each one we choose at most $r-1$ vertices.  So the expected number of hyperedges which have at least one vertex-copy allocated to $\cup_{i=I_v}^{j+2} R_i$ is at most  $\frac{(r-1)(k-1)}{\a(c) n+o(n)}|S_j|\sum_{i=I_v}^{j+2} |R_i|$.  Standard concentration bounds on binomial variables imply that the probability this number  exceeds $\frac{rk}{\a(c) n}|S_j|\sum_{i=I_v}^{j+2} |R_i|+\log^2 n$ is less than $n^{-5}$.   Each such hyperedge contains fewer than $r$ vertex-copies from $S_j$. So with probability at least $1-n^{-5}$, we have
\begin{equation}\lab{ey1}
\blank_1\leq\frac{r^2k}{\a(c)}\frac{|S_j|}{ n}\sum_{i=I_v}^{j+2} |R_i|  +{r}\log^2 n.
\end{equation}

{\bf Step 2: For each $(a,b)$-hyperedge, we choose the $b$ vertex-copies from $S_{j+1}$.}  

We let $\blank_2=\blank-\blank_1$.  So $\blank_2$ is the number  of level-$j$ blanks  whose hyperedge contains a vertex-copy from $R_{j+1}$ and was not counted in $\blank_1$. Note that $\blank_2$ is determined by Steps 1 and 2.

It will be easier to focus on $\blank_2^*\geq \blank_2$ which is defined to be the number of pairs  of vertex-copies $v\in S_j, u\in R_{j+1}$ that both lie in an $(a,b)$-edge for some $a,b$.  We only require an upper bound on $\blank_2$, so it will suffice to bound $\blank_2^*$.

Each $(a,b)$-edge contributes $a$ to $\blank_2^*$ for each of its $b$ copies from $S_{j+1}$ that lie in $R_{j+1}$.  The total number of copies in $R_{j+1}$ is $D^-(R_{j+1})$ so 
\begin{eqnarray}
\ex(\blank_2^*)&=&\sum_{a,b} abM_j^{a,b}\frac{D^-(R_{j+1})}{\sum_{u\in S_{j+1}}d^-(u)}\nonumber\\
&\leq&\frac{D^-(R_{j+1})}{|S_{j+1}|}\sum_{a,b} abM_j^{a,b}\qquad\mbox{since $d^-(u)\geq 1$ for all $u\in S_{j+1}$}\nonumber\\
&\leq&D^-(R_{j+1})\left(1+2Z_1\frac{|S_{j}|}{n}+\frac{\log^2n}{|S_{j+1}|}\right)
\qquad\mbox{by Lemma~\ref{lsi}(e,a,b)}\nonumber\\
&\leq& D^-(R_{j+1}) + 2Z_1\frac{|S_{j}|}{n}D^-(R_{j+1}) +Z_1\log^2 n, \label{eey2}
\end{eqnarray}
since, by Lemma~\ref{lsi}(g),  $D^-({R_{j+1}})\leq D^-(S_{j+1})\le \sum_{v\in S_{j+1}} (d^-(v))^2\leq Z_1|S_{j+1}|$.

To bound $\blank^*_2$, we will prove that it is concentrated.  Note that $\blank^*_2\geq D^-(R_{j+1})$ since every vertex-copy of $R_{j+1}$ lies in an $(a,b)$-edge for some $a\geq 1$, and hence contributes at least one of the pairs counted by $\blank^*_2$. This allows us to focus instead on proving the concentration of
\[X=\blank^*_2-D^-(R_{j+1}).\] 
$X$ is typically much smaller that $\blank^*_2$.  This will be an advantage when we use $\ex(X)$ rather than $\ex(\blank_2^*)$ in our calculations in (\ref{emd2}) below.  Note that $D^-(R_{j+1})$ is fixed.

To prove concentration of $X$, we apply McDiarmid's variation on Talagrand's Inequality\cite{mt}.
We use the version stated in ~\cite{mrbook}:

\noindent{\bf McDiarmid's Inequality}\cite{cm}
{\em Let $X$ be a non-negative random variable determined by
independent trials $T_1,...,T_{m}$ and independent permutations $\Pi_1,...,\Pi_{m'}$.
We call the outcome of one trial $T_{\imath}$, or the mapping of a single element in a permutation $\Pi_{\imath}$, a {\em choice}.
Suppose that for every
set of possible outcomes of the trials and permutations, we have:
\begin{enumerate}
\item[(i)] changing the outcome of any one trial can affect $X$
by at most $\varrho$;
\item[(ii)] interchanging two elements in any one permutation can affect $X$
by at most $\varrho$; and
\item[(iii)] for each $s>0$, if $X\geq s$ then there is a set
of at most $qs$ choices whose outcomes certify that $X\geq s$.
\end{enumerate}
Then for any $t\geq0$, we have
\[\pr(|X-\ex(X)|>t+25\varrho\sqrt{q\ex(X)}+128\varrho^2q)\leq 4\exp\left(-\frac{t^2}{32\varrho^2q(\ex(X)+t)}\right).\]
}

Our random choice in Step 2 is an assignment of the vertex-copies of $S_{j+1}$ to the $(a,b)$-edges.   This can be done by taking a uniform permutation of those vertex-copies:  There is an implicit listing of all $(a,b)$-edges over all $a,b$: For each edge, we know the values of $a,b$ corresponding to that edge, but we don't yet know exactly which $a$ copies from $S_j$ and $b$ copies from $S_{j+1}$ are assigned to each edge; the permutation of Step 2 determines the assignment of the copies of $S_{j+1}$.  Thus the choice corresponding to a particular vertex-copy specifies the hyperedge to which it is assigned. We will apply McDiarmid's Inequality using this permutation.

If we exchange the position of two vertex-copies in the permutation, we are swapping the hyperedges to which they were assigned. This can change $\blank_2^*$, and hence $X$, by at most 
$r-1$ (the extreme case is when the swap involves an $(r-1,1)$-edge).  So we take $\varrho=r-1$.

To certify $X\geq s$, we can always present $\ell^*\leq s$ vertex-copies $u_1,...,u_{\ell^*}$ such 
that each $u_{\ell}$ is assigned to an $(a_{\ell},b_{\ell})$-edge with \fix{$a_{\ell}\ge 2$ and $\sum_{\ell=1}^{\ell^*} (a_{\ell}-1)b_{\ell}\geq s$.}  So this can be certified by the outcomes of $\ell^*\leq s$ choices, and we can take $q=1$.  

\fix{Setting $t=\max\{\ex(X),\log^2 n\}$, it is easy to see that $25\varrho\sqrt{\ex(X)}+128\varrho^2< t$ and $\ex(X)+t\le 2t$.
So McDiarmid's Inequality yields:
\be\label{emd2}
\pr(X>\ex( X ) + 2t)\leq 4\exp\left(-\frac{t^2}{32\varrho^2(\ex(X)+t)}\right)<4\exp\left(-\frac{t}{64\varrho^2}\right)=o(n^{-5}).
\ee
}

Therefore, applying (\ref{eey2}), with probability at least $1-n^{-5}$, we have
\fix{
\begin{equation}\lab{ey2}
\blank_2\leq \blank_2^*\leq D^-(R_{j+1}) + 6Z_1\frac{|S_{j}|}{n}D^-(R_{j+1}) +3Z_1\log^2 n.
\end{equation}
}

{\bf Step 3: For each $(a,b)$-hyperedge, we choose the $a$ vertex-copies from $S_{j}$.}

In most of the analysis for this step, the probabilities are in terms only of the random choices in Step 3; 
 it is assumed that the choices from Steps 1 and 2 are fixed.  We will use the notation $\pr_3$ and $\ex_3$ to indicate this.

Recall from Definition~\ref{dsi} that $\cals_j$ is the hypergraph formed 
by treating each $(a,b)$-edge as a hyperedge on $a$ vertices of $S_j$; we call this an $a$-edge.
At this point, each $a$-edge consists of $a$ blanks, and we have a pool containing $d_j(u)$ vertex-copies of each $u\in S_j$. The random choice in Step 3 is a uniformly random bijection from these vertex-copies to these blanks.

We define 
\be\lab{defPhi}
\Phi:\ \mbox{the  hyperedges of $\cals_j$ that belong to $(a,b)$-edges  containing vertex-copies from $\cup_{i=I_v}^{j+1} R_i$};
\ee
 thus $R_j'$ is the set of vertices of $\Phi$.  Note that $\Phi$ was determined by Steps 1,2, and that $\blank=\blank_1+\blank_2$ is the total number of blanks contained in the hyperedges of $\Phi$. 

For each hyperedge $f\in\cals_j$,  $C_f$ denotes the component of $\cals_j$ that contains $f$.  Thus $R_j$ is the union  over all $f\in \Phi$ of the vertices in $C_f$, and $D^-(R_j)\leq \sum_{f\in\Phi}|D^-(C_f)|$ (there is an inequality here as some pairs $f,f'\in\Phi$ may lie in the same component).  
In order to bound $\ex( D^-(R_j))$, we will bound $\ex_3(D^-(C_f))$ for each  $f$.

Consider any particular hyperedge $f\in\cals_j$. We will expose $C_f$ using a branching process as in \cite{mr1}:

\begin{tabbing} 
Initially, every blank of $f$ is labelled {\em open}.\\
Whi\=le there are open blanks:\\
\> Tak\=e any open blank, and choose a uniform vertex-copy for that blank; \\
\>\>the blank is labelled {\em closed}.\\
\> Assign every other copy of the same vertex to a  blank chosen uniformly\\
\>\> from those not yet assigned vertex-copies.\\
\> For each blank {$g$} that is chosen:\\
\>\>if it is open (i.e. already known to be in $C_f$), we label it {\em closed};\\
\>\>else, every other blank in the same hyperedge as {$g$} is labelled {\em open}.
\end{tabbing}
	
{\em Intuition:} By Lemma~\ref{lsi}(e,f), the vast majority of blanks lie in 1-edges.  So when we assign copies to blanks, we usually do not create any new open blanks.  Thus, the process tends to die out very quickly and $C_f$ typically has few (in fact, no) vertices outside of $f$.

To analyze this branching process, we
consider the following experiment.  Select a sequence of vertices $u_1,...,u_{\log^2 n}$, each chosen from $S_j$ without replacement, where the probability that $w\in S_j\bk \{u_1,...,u_{\ell-1}\}$ is chosen to be $u_{\ell}$ is proportional to $d_j(w)$. 

To run the branching process: each time that we need to select a new vertex-copy {$x$} for an open blank, we simply take the next vertex on the list $u_1,...$ and choose one of its vertex-copies {and let $x$ be that vertex-copy}.  Note that these vertices are chosen with the correct distribution. 

We will prove {that} there is a constant $W_1>0$ such that for any $y\geq 0$ and every $1\leq t<\log^2 n$:

\begin{equation}\label{edCu1}
\ex_3\left(\sum_{\ell=1}^{t}d^-(u_\ell)\right)<t\left(1+2Z_1\frac{|S_{j}|}{n} +2\frac{\log^2n}{|S_i|}\right);
\end{equation}


\begin{equation}\label{edCu2} 
\mbox{for each $a$-edge $f$:}\qquad\pr_3\left(|C_f|\geq t ~\left|~  \sum_{i=1}^{t}d^-(u_i)=y\right.\right)<\left(\frac{W_1|S_j|}{n}\right)^{\rup{\frac{t-a}{(r-1)(k-1)}}}.
\end{equation}

{\em Proof of (\ref{edCu1}):}   From Lemma~\ref{lsi}(g) and the fact that $\sum_{u\in S_j} d_j(u)\geq|S_j|$, initially the expected value of $d^-(u_{\ell})$ is:
\[\frac{\sum_{u\in S_j}d_j(u)d^-(u)}{\sum_{u\in S_j}d_j(u)}\leq 1+Z_1\frac{|S_j|}{n}+\frac{\log^2 n}{|S_j|}.\]
  After removing at most $\log^2 n$ of the vertices of $S_j$,  the expected value of $d^-(u_{\ell})$ rises to at most
\[\left(1+Z_1\frac{|S_{j}|}{n}+\frac{\log^2 n}{|S_j|}\right)\frac{|S_j|}{|S_j|-\log^2n}< 1+2Z_1\frac{|S_{j}|}{n}+2\frac{\log^2n}{|S_j|},\]
since $|S_j|\geq n^{\d}\log^2 n$.

{\em Proof of  (\ref{edCu2}):}   Conditioning on the event $\sum_{i=1}^{t}d^-(u_i)=y$,
exposes information about $d^-(u_1),...,d^-(u_{t})$.  This affects our conditional probability  only through its effect on the conditional distribution of $d_j(u_1),...,d_j(u_t)$, as EXPOSURE determined  $(d_j(u),d^-(u))$ for each $u\in S_j$.  So we can deal with this conditioning by proving that for {\em any } choice of $\d_1,...,\d_{t}$: 
\begin{equation}\label{edCu11}
\pr_3(|C_u|\geq t\mid d_j(u_1)=\d_1,...,d_j(u_{t})=\d_{t-1})<\left(\frac{W_1|S_j|}{n}\right)^{\rup{\frac{t-a}{(r-1)(k-1)}}}.
\end{equation}
Suppose that we selected vertex $u_{\ell}$.   We place one of its vertex-copies  into the open blank that we are filling.   In order to create at least one new open blank, we must assign at least one of the other $d_j(u_{\ell})-1$ copies  to a blank in an $a$-edge for some $a\geq2$; we call such an edge a {\em plural edge}. 

The total number of blanks in $S_j$ is at least $\inv{r-1}$ times the number of vertex-copies in $S_{j+1}$ (the extreme case is when every hyperedge is a $(1,r-1)$-edge) and thus is at least $\frac{|S_{j+1}|}{r-1}\geq \frac{|S_j|}{{2}(r-1)}$ by Lemma~\ref{lsi}(a,b).
If we have assigned the copies of fewer than $\log^2n$ vertices, then the number of  unfilled blanks is at least $\frac{|S_j|}{{2}(r-1)}-(k-1)\log^2 n \geq \frac{|S_j|}{{2}r}$.  Lemma~\ref{lsi}(f) says that the total number of blanks in all plural edges is at most $Z_1|S_j|^2/n+\log^2 n$.   So the probability that at least one of those $d_j(u_{\ell})-1\leq k-2$ copies of $u_{\ell}$ is assigned to  a plural edge at most 
\[(k-2)\left(\frac{Z_1|S_j|^2}{n}+\log^2 n\right)\left/\frac{|S_j|}{{2}r}\right.<{2k}rZ_1\frac{|S_j|}{n}.\]
If at least one copy of $u_{\ell}$ is assigned to a plural edge, then we create at most $(k-2)(r-1)$ new open blanks.   For an $a$-edge $f$, if $|C_f|\geq t$, then we must {have} created at least $t-a$ open blanks (in addition to the $a$ initial open blanks in $f$) during the exposure of the first $t$ vertices $u_1,\ldots,u_{t}$.  So at least $\rup{\frac{t-a}{(r-1)(k-1)}}$ of the vertices $u_1,...,u_{t}$ must have a copy assigned to a plural edge. The probability of this occurring is at most:
\[{t\choose  \rup{\frac{t-a}{(k-1)(r-1)}}}\left( {2k}rZ_1\frac{|S_j|}{n}\right)^{\rup{ \frac{t-a}{(k-1)(r-1)}}}
<\left(\frac{\fix{et\cdot}2krZ_1|S_j|}{\rup{\frac{t-a}{(k-1)(r-1)}}n}\right)^{\rup{ \frac{t-a}{(k-1)(r-1)}}}
<\left(W_1\frac{|S_j|}{n}\right)^{\rup{ \frac{t-a}{(k-1)(r-1)}}},\]
for an appropriate constant $W_1$, thus establishing~(\ref{edCu2}).

Having proven (\ref{edCu1})  and (\ref{edCu2}), these imply that there is a constant $W$ such that

\begin{equation}\label{edCu}
\mbox{for each $a$-edge $f$:}\qquad
\ex_3(D^-(C_f))\leq a\left(1+W\frac{|S_j|}{n}+5\frac{\log^2 n}{|S_{{j}}|}\right).
\end{equation}

\noindent{\em Proof}   First note that  (\ref{edCu2})  implies:  
\be
\pr(|C_f|>\log^2 n)\leq \left(\frac{W_1|S_j|}{n}\right)^{\log^2n/(r-1)(k-1)}=o\left({n^{-10}}\right).\lab{bigCf}
\ee

Next, note that $D^-(C_f)\leq rcn$, the number of vertex-copies in the entire configuration.  So

\begin{eqnarray}
\ex_3(D^-(C_f))
&=&\sum_{y=1}^{rcn}y\pr_3(D^-(C_f)=y)\non\\
&\leq&\sum_{y=1}^{rcn}y\left(\pr_3(|C_f|\geq \log^2n)+\sum_{t= a}^{\log^2n}\pr_3\left((|C_f|\geq t) \wedge \Big(\sum_{i=1}^{t}d^-(u_i)=y\Big)\right)\right)\non\\
&\leq&O(n^2)\pr_3(|C_f|\geq \log^2n)\non\\
&&+\sum_{y=1}^{rcn}y\sum_{t= a}^{\log^2n}\pr_3\left(\sum_{\ell=1}^{t}d^-(u_{\ell})=y\right)\left(\frac{W_1|S_j|}{n}\right)^{\rup{ \frac{t-a}{(k-1)(r-1)}}}\mbox{by (\ref{edCu2})}\non\\
&\leq&o(n^{-1})+\sum_{t= a}^{\log^2n}\left(\frac{W_1|S_j|}{n}\right)^{\rup{ \frac{t-a}{(k-1)(r-1)}}}\ex_3\left(\sum_{\ell=1}^{t}d^-(u_\ell)\right)\ \mbox{{by~\eqn{bigCf}}}\non\\
&<&o(n^{-1})+\sum_{t= a}^{\log^2n}\left(\frac{W_1|S_j|}{n}\right)^{\rup{ \frac{t-a}{(k-1)(r-1)}}}t\left(1+2Z_1\frac{|S_{j}|}{n} +2\frac{\log^2n}{|S_i|}\right)
\qquad\mbox{by (\ref{edCu1})}\non\\
&<&o(n^{-1})+\left(a+2k^2r^2\frac{W_1|S_j|}{n}\right)\left(1+2Z_1\frac{|S_{j}|}{n} +2\frac{\log^2n}{|S_i|}\right)\non\\
&<&a+Wa\frac{|S_j|}{n}+5{a}\frac{\log^2 n}{|S_j|},
\end{eqnarray}
for an appropriate constant $W$; this is (\ref{edCu}).

Our goal is to bound $D^-(R_j)\leq\sum_{f\in\Phi} D^-(C_{f})$ in order to obtain (\ref{recMain}).  As in Step 2,  we will instead focus on a related quantity: {We use $|f|^*$ to denote the number of {\em vertices} in $f$; i.e.\ if $f$ contains multiple copies of a vertex $u$ then $u$ is counted only once in $|f|^*$.}   We define
\[X=\sum_{f\in\Phi} \min\{D^-(C_{f}),\log^3 n\}-|f|^*.\]

Note that  $D^-(C_f)\geq |f|^*$ as every vertex $u\in f$ satisfies $d^-(u)\geq 1$. Also, $\log^3n>r\geq|f|^*$ and so $X$ is non-negative.

We will show that, with very high probability, $D^-(C_{f})< \log^3 n$ for every $f$, which allows us to work with $X$.
The advantages of doing so are twofold: (i) subtracting $|f|^*$ from each term in the summand has the same advantage as switching to $X$ in the analysis of Step 2; (ii) by introducing an upper bound of $\log^3 n$ on the contribution of each term, we bound the effect of each random choice on $X$.  Since $\sum_{f\in\Phi} |f|\leq \blank$,
(\ref{edCu}) implies
\begin{equation}\label{ee3x}
\ex_3(X)\leq \sum_{f\in\Phi} |f|\left(W\frac{|S_j|}{n}+\frac{5\log^2 n}{|S_j|}\right)\leq \blank\left(W\frac{|S_j|}{n}+\frac{\log^2 n}{|S_j|}\right).
\end{equation}

To prove concentration of $X$, we again use McDiarmid's Inequality.

Our random choice in Step 3 is an assignment of the vertex-copies of $S_{j}$ to the blanks in the $a$-edges.   As in the analysis of Step 2, we do this by taking a uniform permutation of those vertex-copies.

We first bound the amount by which exchanging two vertex-copies can affect $X$.
Suppose that we replace a copy of $u$ in an edge $f$ with a copy of $w$ in an edge $f'$.  We will show that this cannot increase the contribution of the edges in $C_f$ to $X$ by more than $2(k-1)\log^6 n$.  

Let $C$ be the component containing $f$ after the copy of $u$ is replaced by a copy $\omega$ of $w$.   Consider removing $\omega$ from $f$; so the size of $f$ has been reduced by one. {Let $C'$ by the hypergraph obtained from $C$ by the removal of $\omega$.}  Let $C_1$ be the component {of $C'$} containing $f$ (if $f$ had size 1 and hence is now empty, then $C_1=\emptyset$) {and} let $C_2$ be the union of the other components {of $C'$}.  Note that  {\em before the copy of $u$ was replaced with $\omega$}: $C_1$ was a subgraph of $C_f$ and  $C_2$ was a subgraph of the component containing $w$.  

{\em Case 1: $|C_1|,|C_2|\geq \log^3n$.}  Then for every hyperedge $f'\in C_1$, $\min(|C_{f'}|,\log^3 n)=\log^3 n$ both before and after the switch.  The same is true of every hyperedge in $C_2$.  So the contribution of those edges to $X$ was unchanged by the switch.

{\em Case 2: $|C_1|\geq \log^3n,|C_2|<\log^3 n$.}  The contribution of the edges of $C_1$ was unchanged by the switch, but the contribution of each edge in $C_2$ may have increased by up to $\log^3n$.  There are at most $(k-1)|C_2|$ vertex-copies amongst those edges and so there are at most $(k-1)|C_2|<(k-1)\log^3n$ such edges.  So the total contribution from those edges increases by less than $(k-1)\log^6n$.

In each of the remaining two cases, similar reasoning shows that the total contribution of the edges in $C_{{f}}$ increases by at most $2(k-1)\log^6n$. {The same argument shows that the total contribution of the edges in $C_{f'}$, the component containing $f'$ (the other hyperedge involved in the switch) is at most  $2(k-1)\log^6n$.} Thus no switch can increase $X$ by more than $4(k-1)\log^6n$.  By considering the inversion of a switch, no switch can decrease $X$ by more than $4(k-1)\log^6n$.  So we can take $\varrho=4k\log^6n$.

}

 We next show that the event $X\geq s$ can be certified by revealing the outcomes of at most 
$2rs$ choices. If $X\ge s$ then there must be a subset of the edges of $\Phi$, say $f_1,...,f_{\ell^*}$, with $D^-(C_{f_1})\geq |f_1|^*+ s_1,...,D^-(C_{f_{\ell^*}})\geq |f_{\ell^*}|^* + s_{\ell^*}$, and $s_1+...+s_{\ell^*}=s$ {with each $s_{\ell}\ge 1$}. For each $1\leq\ell\leq\ell^*$, we will certify that $D^-(C_{f_{\ell}})\geq|f_{\ell}|^* + s_{\ell}$ by revealing at most $2rs_{\ell}$ choices as follows. We discuss two cases.

{\em Case 1: $|C_{f_{\ell}}|\leq |f_{\ell}|^* + s_{\ell}$}. In this case, we reveal a spanning tree of $C_{f_{\ell}}$.  That is, we reveal the choice of vertex-copy assigned to each hyperedge of a spanning tree of $C_{f_{\ell}}$.  This spanning tree only needs to consist of $f_{\ell}$, and at most $s_{\ell}$ additional edges, and so the total number of vertex-copies whose assignments are revealed is at most $r(1+s_{\ell})\leq 2rs_{\ell}$, as $s_{\ell}\geq 1$. 

{\em Case 2: $|C_{f_{\ell}}|> |f_{\ell}|^*+s_{\ell}$}. In this case, we will reveal a connected subgraph of $C_{f_{\ell}}$ containing $f_{\ell}$ and at least $s_{\ell}$ additional vertices. As described in Case 1, we can do so using at most $s_{\ell}+1$ edges and hence by revealing the assignments of at most $r(1+s_{\ell})\leq 2rs_{\ell}$ vertex-copies.  Since each of these vertices has $d^-\geq 1$, this will certify that $D^-(C_{f_{\ell}})\geq  |f_{\ell}|^*+s_{\ell}$.

This shows that we may certify $X\ge s$ by exposing at most $2rs$ choices so we can take $q=2r$.
\fix{
Setting $t=\max\{\ex(X),6\times 2048k^2r\log^{13}n\}$, it is easy to see that $25\varrho\sqrt{q\ex(X)}+128\times \varrho^2q< t$ {and $\ex(X)+t\leq 2t$}.
So {applying McDiarmid's Inequality with $\varrho=4k\log^6n$ and $q=2r$} yields:
\bea
\pr(X>\ex( X ) + 2t)&\leq& 4\exp\left(-\frac{t^2}{{1024k^2r\log^{12}} n(\ex(X)+t)}\right)<4\exp\left(-\frac{t}{{2048}k^2r\log^{12} n}\right)\nonumber\\
&<&4\exp(-6\log n)=o(n^{-5}).\label{emd}
\eea

}

To use our bound on $X$ to obtain a bound on $D^-(R_j)$, we first note that  by~\eqn{bigCf}, with probability at least $1-n^{-9}$: every $C_f$ has size at most $\log^2 n$ and hence by Lemma~\ref{lsi}(i),
 $D^-(C_f)<\log^3 n$ for all $f$.
This, (\ref{ee3x}) and  (\ref{emd}) yield that with probability at least $1-n^{-5}$:

\fix{
\begin{eqnarray*}
D^-(R_j)&\leq& X+\sum_{f\in\Phi} |f^*|\leq X+\sum_{f\in\Phi} |f|\\
& \leq &X+\blank\leq\ex_3(X)+2t+\blank\\
&\leq&3\ex_3(X)+12\times 2048k^2r\log^{{13}}n+\blank\\
&\leq&\blank\left( 1+3W\frac{|S_j|}{n}+3\frac{\log^2 n}{|S_j|}\right)+12\times 2048k^2r\log^{{13}}n\\
&\leq&\blank\left( 1+3W\frac{|S_j|}{n}\right)+13\times 2048k^2r\log^{{13}}n,
\end{eqnarray*}
}since $\blank\leq (k-1)|S_j|$, fix{and so $\blank \log^2 n/|S_j|=o(\log^{13}n)$}.  This, with  (\ref{ey1}) and  (\ref{ey2}),  yields that with probability at
least $1-3n^{-5}$, we have:

\begin{eqnarray*}
D^-(R_j)&\leq&
 \left(D^-(R_{j+1}) + Z'_1\frac{|S_{j}|}{n}D^-(R_{j+1}) 
+\frac{r^2k}{\a(c)}\frac{|S_j|}{ n}\sum_{i=I_v}^{j+2} |R_i|+\fix{Z_1'}\log^2 n\right)\\
&&\qquad
\times\left( 1+\fix{3}W\frac{|S_j|}{n}\right) +\fix{13\times 2048k^2r}\log^{13}n\\
&\leq&
D^-(R_{j+1}(v))+Z\frac{|S_j|}{n}\sum_{\ell=I_v}^{j+1}{D}^-(R_{\ell}(v))
+\log^{14} n,
\end{eqnarray*}
 for suitable constants $Z'_1$ and $Z$.   This is~(\ref{recMain}).  

So the probability that~(\ref{recMain}) holds for a given $j$ is at least $1-\frac{3}{\e}n^{-5}>1-n^{-4}$. Taking the union bound for $B\le j\le I_v$ yields Lemma~\ref{l.rj}.
\proofend

\subsection{Solving the recurrence}
\lab{sec:rec2}

Because Lemma~\ref{l.rj} holds only for $|S_j|>n^{\d}\log^2n$ and for $j\ge B$, we define:
\[
i_0 = \min\{i:\ |S_{i+1}|< n^{\d}\log^2n\};
\]
i.e.\ $i_0$ is the largest index so that Lemma~\ref{l.rj} applies to $S_B,...,S_{i_0}$.  Note that Lemma~\ref{lsi}(a,b) implies
\be\label{e.i0}
|S_{i_0}|<2n^{\d}\log^2n.
\ee
Recall that
\[
R^{(B)}(v)= R(v)\setminus \Big(\cup_{1\le i<B}S_i\Big).
\]

We will prove:
\begin{lemma}\lab{lrec3}
There is a constant $X=X(r,k)>0$ such that for any $1\leq i\le \imax$ and any $v\in S_i$:
\[\pr(|R^{(B)}(v)|>n^{X\d})\fix{\le}\inv{n^2}.\]
\end{lemma}

\proof Taking $B$ and $Z$ as specified in Lemma~\ref{l.rj}, and taking the union bound over all vertices $v$, we have that with probability at least $1-n^{-\fix{2}}$: \eqn{recMain} holds for all  $v$ and all $B\leq j\leq \min\{I_v,i_0\}$.

We fix a particular $i\geq B$ and  $v\in S_i$.  As in the previous section, we say $R_{\ell}$ instead of $R_{\ell}(v)$. Restating~\eqn{recMain} and replacing $j$ with $j'$ (because of an index change below) yields:

\be\label{e.rj}
D^-(R_{j'})\leq D^-(R_{j'+1})+Z\frac{|S_{j'}|}{n}\sum_{\ell=i}^{j'+1}D^-(R_{\ell})
+\log^{14} n,\qquad\mbox{for all } B\le j'\le \min\{i,i_0\}.
\ee

This recursive equation bounds $D^-(R_{j'})$ in terms of $D^-(R_{\ell})$ for values of $\ell$ that are {\em larger} than $j'$.   In order to analyze this using a recursive equation where values are bounded by values with {\em smaller} indices, we will perform the change of indices:
\[j:= i -j'.\]

We will recursively define $r_j$ such that:

\begin{eqnarray}
r_j&\geq&D^-(R_{i-j})\label{elj}
\end{eqnarray}

First, we set up our base cases. Because~(\ref{e.rj}) only holds for  $j'\le \min\{i,i_0\}$, i.e.\ for $j\geq\max\{0,i-i_0\}$, we need a base case for all smaller values of $j$.  We set 
\[j_0= \max\{0,i-i_0\}\] 
and define:

\[\mbox{For }  0\le j\leq j_0:\qquad r_j=n^{2\d}.\]
We verify that something even stronger than~(\ref{elj})  holds for this base case. 
Applying Lemma~\ref{lsi}(\fix{i}), we have
$D^-(R_{i-j})\leq D^-(S_{i-j})<\fix{\log n\cdot}|S_{i-j}|.$ 
  Lemma~\ref{lsi}(c) and~(\ref{e.i0}) yield
\be\lab{total}
\sum_{j=0}^{i-i_0}D^-(R_{i-j})\le \fix{\log n}\sum_{j=0}^{i-i_0}|S_{i-j}|\le \fix{\log n}\sum_{\ell\ge i_0}|S_{\ell}| \le \fix{\log n\cdot}Z_1|S_{i_0}|n^{\d/2}=O(n^{3\d/2}\log^{\fix{3}} n).
\ee

 This immediately yields~(\ref{elj}) for all $j\le \fix{j_0}$.

We define the following recursion.
\be
\forall  j_0+1\le j\le i-B:
r_{j}=r_{j-1}+Z\frac{|S_{i-j}|}{n}\sum_{\ell=j_0}^{j-1}r_{\ell}+n^{2\d}.\lab{rec}
\ee

We argue inductively that ~(\ref{elj}) holds for all $j_0\leq j\leq i-B$.  We have already
seen that it holds for the base case $j=j_0$.  For higher values of $j$,~(\ref{e.rj}) yields:
\begin{eqnarray*}
D^-(R_{i-j})&\leq &D^-(R_{i-(j-1)})+Z\frac{|S_{i-j}|}{n}\sum_{\ell=i}^{{i-j+1}}D^-(R_{\ell})
+\log^{14} n\\
&\leq& r_{j-1}+Z\frac{|S_{i-j}|}{n}\sum_{\ell=j_0}^{j-1}r_{\ell}
+\sum_{\ell=0}^{j_0-1}D^-(R_{i-\ell})+\log^{14} n\quad\mbox{by Lemma~\ref{lsi}(d) and induction on~(\ref{elj})}\\
&\le & r_{j-1}+Z\frac{|S_{i-j}|}{n}\sum_{\ell=j_0}^{j-1}r_{\ell}+n^{2\d}\quad \mbox{by~\eqn{total} {and noting that $\log^{14}n$ is absorbed by $n^{2\d}$}} \\
&=&r_{j}.
\end{eqnarray*}

Thus $r_j$ bounds $|R_{i-j}|$:

\[|R_{i-j}|\le D^-(R_{i-j})\leq r_j ,\quad \forall \ 0\leq j\leq i-B,\]
It will be convenient to define:
\[t_j=\sum_{\ell=j_0}^j r_{\fix{\ell}}.\]
Therefore, by~\eqn{total},
\be
|R^{(B)}(v)|= \sum_{j=0}^{i-B}|R_{i-j}| \leq \sum_{j=0}^{i-B}D^-(R_{i-j})= \sum_{j=j_0}^{i-B}D^-(R_{i-j})+\sum_{j=0}^{j_0-1}D^-(R_{i-j})\le t_{i-B}+n^{2\d}.\lab{Rbound}
\ee
To prove Theorem~\ref{mt3}, it is sufficient to show that $t_{i-B}=n^{O(\d)}$.

Since $r_j=t_j-t_{j-1}$, rearranging~\eqn{rec}, we have
$$
t_{j}-t_{j-1}= t_{j-1}-t_{j-2}  + Z\frac{|S_{i-j}|}{n}t_{j-1}
+ n^{2\d}, \quad \forall j\ge j_0+1,
$$
where $t_{j_0}=n^{2\d}$ and $t_{j_0-1}=0$.
We solve this recurrence. It will be helpful to find sequences $(a_j)_{j\ge j_0}$ and $(b_j)_{j\ge j_0+1}$ that satisfy
\begin{equation}
t_{j}-a_jt_{j-1}=b_j(t_{j-1}-a_{j-1}t_{j-2})+n^{2\d},\quad \forall j\ge j_0+1.\lab{ezzzz}
\end{equation}
Rearranging gives
\[t_{j}-t_{j-1}=(a_j-1+b_j)t_{j-1}+(-a_{j-1}b_j)t_{j-2}+n^{2\d},\]
so we require that for all $j\ge j_0+1$,
\bean
a_j-1+b_j&=&1+ Z\frac{|S_{i-j}|}{n}\\
 -a_{j-1}b_j&=&-1,
\eean
and we may set initial condition $a_{j_0}=b_{j_0+1}=1$.

So we define $a_j,b_j$ recursively as:
\[
b_j=1/a_{j-1},\quad a_j=2-\frac{1}{a_{j-1}}+ Z\frac{|S_{i-j}|}{n},
\]
and~(\ref{ezzzz}) holds. Since $a_j+\inv{a_{j-1}}\ge 2$ for each $j$ and $a_{j_0}=1$, it follows that $a_j\ge 1$ for every $j\ge j_0$.   Next, we show that
there is a constant $D=D(r,k)>0$  such that for every $j\ge j_0+1$,
\begin{equation}
1+Z\frac{|S_{i-j}|}{n}\le a_j\le 1+D\sqrt{\frac{|S_{i-j}|}{n}}.\lab{induction}
\end{equation}
The lower bound follows immediately from $a_{j-1}\geq 1$ and the recursion
$a_j=2-1/a_{j-1}+ Z|S_{i-j}|/n$.  We prove the upper bound
by induction. By taking  $D\geq Z$ we ensure that~\eqn{induction} holds for $j= j_0+1$. Now
 assume that $j\ge j_0+2$ and that~\eqn{induction} holds for $j-1$, and so:
\[
\frac{1}{a_{j-1}}\geq 1-D\sqrt{\frac{|S_{i-j+1}|}{n}}.\]
Since $a_j=2+ Z|S_{i-j}|/n-1/a_{j-1}$, we have
\[ a_j\le 1+\frac{Z|S_{i-j}|}{n}+D\sqrt{\frac{|S_{i-j+1}|
}{n}}.
\]
By the definition of $j_0$ and $i_0$, we have $|S_{i-j}|\ge n^{\d}\log^2 n$ for all $j\ge j_0$. So
  by Lemma~\ref{lsi}(b), we have that for all $j_0\le j\le i-B$, $|S_{i-j+1}|\le (1-Y_2\sqrt{|S_{i-j}|/n})|S_{i-j}|$ for constant $Y_2=Y_2(r,k)>0$.
Thus,
\bean
1+Z\frac{|S_{i-j}|}{n}+D\sqrt{\frac{|S_{i-j+1}|}{n}}&\le & 1+Z\frac{|S_{i-j}|}{n}+\left(1-\frac{Y_2}{2}\sqrt{|S_{i-j}|/n}\right)D\sqrt{\frac{|S_{i-j}|}{n}}\\
&=&1+D\sqrt{\frac{|S_{i-j}|}{n}}-\left(\frac{DY_2}{2}-Z\right)\frac{|S_{i-j}|}{n}\le 1+D\sqrt{\frac{|S_{i-j}|}{n}}
\eean
where the last inequality holds
by choosing $D>2Z/Y_2$. Thus,~\eqn{induction} holds also for $j$ and thus it holds for every $j_0\le j\le i-B$.

Now we continue to solve the recurrence~\eqn{rec}. Let $c_j=t_{j}-a_jt_{j-1}$. Then, since $b_j=1/a_{j-1}\le 1$ for every $j\ge j_0+1$, we have
\[
c_{j}=b_j c_{j-1}+n^{2\d} \le c_{j-1}+n^{2\d}\le c_{j_0} +(j-j_0) n^{2\d}\le n^{2\d}+(j-j_0) n^{2\d}.
\]
 Since $j<i=O(n^{\d/2}\log n)$ by Theorem~\ref{mt}(b), this yields
\be
t_{j}-a_jt_{j-1} \le n^{2\d}+ O(n^{5\d/2}\log n)\le  U:=n^{3\d},\quad \mbox{for all $j\ge j_0+1$.}\lab{rec2}
\ee

Let $\ell_0$ be the maximum integer for which $|S_{\ell_0}|\ge n^{1-\d}$.  Again applying Theorem~\ref{mt}(b), we have $\ell_0 =O(n^{\d/2}\log n)$. Note that, \fix{due to the monotonicity of $|S_{\ell}|$ by Lemma~\ref{lsi}(a,b),} $\ell_0<i_0$ trivially by the definition of $\ell_0$ and $i_0$ and by the fact that $0<\d<1/2$.

Now is a good time to recall that our goal is to show $t_{i-B}=n^{O(\d)}$.

\eqn{rec2} says $t_{j}\leq  a_jt_{j-1} + U$. Applying this
recursively yields that for every $1\le \ell\le i-B$,
\be
t_{i-B}\le t_{i-B-\ell}\prod_{h=0}^{\ell-1}a_{i-B-h}+U\left(1+\sum_{h_2=0}^{\ell-2}\prod_{h=0}^{h_2} a_{i-B-h}\right).
\lab{rec3}
\ee

Since $a_h\ge 1$ for each $h$, we have $1+\sum_{h_2=0}^{\ell-2}\prod_{h=0}^{h_2} a_{i-B-h}\le \ell\prod_{h=0}^{\ell-2} a_{i-B-h}$. First assume that $i\ge \ell_0+2$. Taking $\ell=\ell_0-B+2$ in~(\ref{rec3}) and noting that
$a_{i-\ell_0-1}<U\ell_0$ by~\eqn{induction}  yields
\be
t_{i-B}\le   t_{i-\ell_0-2}\left(\prod_{j=0}^{\ell_0-B+1} a_{i-B-j}\right)+U \ell_0 \prod_{j=0}^{\ell_0-B} a_{i-B-j}\le (1+t_{i-\ell_0-2})U \ell_0 \prod_{j=0}^{\ell_0-B} a_{i-B-j}.\lab{ezz}
\ee
Again applying~\eqn{induction}, we have:
\be
\prod_{j=0}^{\ell_0-B} a_{i-B-j}\le\exp\left(D\sum_{j=0}^{\ell_0-B} \sqrt{\frac{|S_{j+B}|}{n}}\right)
=\exp\left(D\sum_{j=B}^{\ell_0} \sqrt{\frac{|S_j|}{n}}\right).
\lab{ezz2}
\ee

Next, we bound $\exp\left(D\sum_{j=B}^{\ell_0} \sqrt{\frac{|S_j|}{n}}\right)$. By Lemma~\ref{lsi}(b), we have for all $j\le \ell_0$,
\[|S_{j+1}|\leq \exp\left(-Y_2\sqrt{\frac{|S_{j}|}{n}}\right)|S_j|,\]
and so
\[|S_{\ell_0+1}|\le\exp\left(-Y_2\sum_{j=B}^{\ell_0}\sqrt{\frac{|S_j|}{n}}\right)|S_B|.
\]
By the definition of $\ell_0$ and Lemma~\ref{lsi}(a), we have $|S_{\ell_0+1}|>\hf n^{1-\delta}$, so
\[
\exp\left(Y_2\sum_{j=B}^{\ell_0}\sqrt{\frac{|S_j|}{n}}\right)\le \frac{|S_B|}{|S_{\ell_0+1}|}< \frac{n}{|S_{\ell_0+1}|}=O(n^{\delta}),
\]
and so
\[
\exp\left(D\sum_{j=B}^{\ell_0}\sqrt{\frac{|S_j|}{n}}\right)=n^{O(\d)},
\]
since both $D$ and $Y_2$ are positive constants.
This, (\ref{ezz}),~(\ref{ezz2}), $\ell_0=O(n^{\d/2}\log n)$ and $U=n^{3\d}$ yield
\be
t_{i-B}=n^{O(\d)}t_{i-\ell_0-2}.
\lab{rec4}
\ee
It is easy to see that if $i<\ell_0+2$, then the above argument already proves that $t_{i-B}=n^{O(\d)}$, by taking $\ell=i-B$ in~\eqn{rec}. So we assume that $i\ge \ell_0+2$.
It only remains to show that $t_{i-\ell_0-2}=n^{O(\d)}$.
The same recursion that produced~(\ref{rec3})  yields that for every $1\le \ell\le i-\ell_0-2-j_0$,
\[
t_{i-\ell_0-2}\le t_{i-\ell_0-2-\ell}\prod_{h=0}^{\ell-1}a_{i-\ell_0-2-h}+U\left(1+\sum_{h_2=0}^{\ell-2}\prod_{h=0}^{h_2} a_{i-\ell_0-2-h}\right).
\]
Arguing as for~\eqn{ezz}, this time taking $\ell=i-\ell_0-2-j_0$ yields (noting that $t_{j_0}\le n^{2\d}$ always)
\be
t_{i-\ell_0-2}\le (t_{j_0}+Ui)\prod_{j=0}^{i-\ell_0-3-j_0} a_{i-\ell_0-2-j}
=n^{O(\d)}\prod_{j=0}^{i-\ell_0-3-j_0} a_{i-\ell_0-2-j},
\lab{eq:t}
\ee
and arguing as for~\eqn{ezz2} yields
\be
\prod_{j=0}^{i-\ell_0-3-j_0} a_{i-\ell_0-2-j}\leq\exp\left(D\sum_{j=\ell_0+2}^{i-1-j_0} \sqrt{\frac{|S_j|}{n}}\right).
\lab{eq:tt}
\ee

By  the definition of $\ell_0$ we have $|S_{\ell_0+2}|< n^{1-\d}$.
By our definition of $j_0$, for every $\ell_0+2\le j\le i-1-j_0$, we have $|S_{j}|\ge n^{\d}\log^2 n$, and so we can apply Lemma~\ref{lsi}(a) to show
\[
|S_{j}|\le (1-Y_2n^{-\d/2})^{j-(\ell_0+2)}|S_{\ell_0+2}|
\le (1-Y_2n^{-\d/2})^{j-(\ell_0+2)}n^{1-\d}, \quad \forall\ \ell_0+2\le j\le i-1-j_0,
\]
which implies that
\[
\sum_{j=\ell_0+2}^{i-1-j_0}\sqrt{\frac{|S_j|}{n}}\le n^{-\d/2}\sum_{j\ge 0} (1-Y_2n^{-\d/2})^j =O(1).
\]
This proves that $t_{i-\ell_0-2}=n^{O(\d)}$ by~\eqn{eq:t} and~\eqn{eq:tt}. So by~\eqn{rec4} it completes our proof that $t_{i-B}=n^{O(\d)}$. Since $|R^{(B)}(v)|\le t_{i-B}+n^{2\d}$ by~\eqn{Rbound}, this proves the lemma.
\proofend

\subsection{Bounding the maximum depth}
\lab{sec:final}

Lemma~\ref{lrec3} bounds $|R^{(B)}(v)|$ for any non-$k$-core vertex $v$ in $AP_r(n,m)$.  Corollary~\ref{ccon0} implies that the same bound holds for $\H_r(n,m)$. 

The following lemma allows us to bound $|R(v)|$ using the bound on $|R^{(B)}(v)|$. To present the lemma,
 we define $N^s(A)$ to be the set of vertices with distance at most $s$ from $A$. The following easily proved lemma is from~\cite[Lemma 34]{amxor}.
\begin{lemma}\lab{lem:neighbours}
Assume $s,c>0$ are $O(1)$. A.a.s.\ for every subset of vertices $A$ in $\H_r(n,cn)$ such that $A$ induces a connected subgraph, $|N^s(A)|=O(|A|+\log n)$.
\end{lemma}

Applying Lemma~\ref{lem:neighbours} with $A=R^{(B)}(v)$ yields our bound on the maximum depth as follows:

{\em Proof of Theorem~\ref{mt3}:} 
By Lemma~\ref{lrec3}, there is a constant $X>0$ such that in $AP_r(n,m)$ \aas there are no vertices $v$ with $R^{(B)}(v)>n^{X\d}$. Any vertices in $R(v)$ that are not contained in $R^{(B)}(v)$, must have been removed during the first $B$ rounds of the parallel stripping process, and must be within distance $B$ of some member of $R(v)$. This bounds each $R_B(v)$ in $AP_r(n,m)$; Corollary~\ref{ccon0} shows that the same bound holds for $\H_r(n,m)$. Now we can apply Lemma~\ref{lem:neighbours}, to show that \aas for each $v\in \H_r(n,m)$, $|R(v)|\leq Y(|R^{(B)}(v)| +\log n)<n^{2X\d}$, for some constant $Y=Y(B,r,k)$. 
\proofend

\section{Proofs of Lemma~\ref{lsi} {and Theorem~\ref{mt}(a): the tight lower bound}}\lab{slsi}

\fix{
In this section, we will prove the key Lemma~\ref{lsi}.
Throughout this section, we have $c=c_{r,k}+n^{-\d}$ for some $0<\d<\hf$.  We have already seen that Lemma~\ref{lsi} is needed to complete the proof of Theorem~\ref{mt2}.  Lemma~\ref{lsi}(a,b) will also immediately imply the lower bound in Theorem~\ref{mt}(a):
}

{\em Proof of Theorem~\ref{mt}(a):} The upper bound was proved in Section~\ref{s.mtub}, so we only need to establish the lower bound: that a.a.s.\ the stripping number is  $\Omega(n^{\d/2}\log n)$. 

By Lemma~\ref{lsi}(a,b), there must be an iteration $i_0$ in the parallel process such that $\hf n^{1-\d}<|S_{i_0}|<n^{1-\d}$. Taking the constant $Y_1$ from Lemma~\ref{lsi} and noting that $2\d <1$, we let $a>0$ be a constant satisfying $1-2Y_1a>2\delta$. 
Applying Lemma~\ref{lsi}(a) recursively for all $i_0\le i\le \ell=an^{\d/2}\log n$,  we have
\[
|S_{i_0+\ell}|\ge (1-Y_1n^{-\d/2})^{\ell}|S_{i_0}|\ge \exp(-2Y_1\ell n^{-\d/2})\cdot \hf n^{1-\d} = \hf n^{-2aY_1+1-\d}.
\]
This is valid since $|S_{i_0+j}|\ge |S_{i_0+\ell}|\ge n^{\d}\log^2 n$ for all $0\le j\le \ell$ (which is desired in order to apply Lemma~\ref{lsi}) by our choice of $a$. This shows that the number of iterations in the parallel stripping process applied to $AP_r(n,m)$ is at least $an^{\d/2}\log n$;  Corollary~\ref{ccon0} implies that the same bound holds for $\H_r(n,m)$.  This confirms the lower bound in Theorem~\ref{mt}(a).
\proofend

\fix{Most of the work in our proof of Lemma~\ref{lsi} goes towards proving parts (a,b).  This requires a very tight analysis of the evolution of $L_t$,  the total degree of the light vertices (i.e. vertices of degree less than $k$) in $G_t$, the hypergraph remaining after $t$ iterations of SLOW-STRIP.  Much of this work can be viewed as a strengthening of the analysis from Section~\ref{smt1}.  A key result from that section is Lemma~\ref{llt1} which bounds $\ex(L_{t+1}-L_t\mid \calf_t)<-K_1n^{-\d/2}$.  This sufficed to prove the upper bound for Theorem~\ref{mt}(b), but to prove the lower bound, we need to strenghten Lemma~\ref{llt1} in two ways:  (i) when $L_t>n^{1-\d}$, the bound on the expected change is decreased to $-\Theta(\sqrt{L_i/n})$; and (ii) we obtain matching lower bounds on the expected change. This is Corollary~\ref{cor:l-br}.
 
 Of course, the expected change in $L_i$ depends on the rate at which new light vertices are created.  In Section~\ref{s.vdk} we studied this rate by examining two key closely related parameters:  roughly speaking, $\zeta_t$ is the average degree of the heavy vertices in $G_t$ (i.e. vertices of degree at least $k$) and $\bar p_t$ is the probability that a uniform vertex-copy from the heavy vertices is a copy of a degree $k$ vertex.  We obtained a coarse bound on $\zeta_t$ (Lemma~\ref{l:zetaT}, Corollary~\ref{cor:zetaT}) which was sufficient to obtain Lemma~\ref{llt1}.
For our tighter analysis in this section, we will require a much tighter bound on $\zeta_t$ (Lemma~\ref{l:zetaseq}). In addition, we will introduce  two more parameters.}

Recall that $G_t$ is the hypergraph (i.e.\ configuration) remaining after $t$ iterations of SLOW-STRIP on $AP_r(n,m)$, and $\hG_i$ is the hypergraph (i.e.\ configuration) remaining after $i-1$ iterations of the parallel stripping process.  Recall also that $t(i)$ is the iteration of SLOW-STRIP corresponding to the beginning of iteration $i$ of the parallel stripping process.  
Recall that $\tau$ denotes the step when SLOW-STRIP terminates.


Recall the definitions of $L_t$, $D_t$, $N_t$, ${\bar p_t}$ from Sections~\ref{smt1} and~\ref{s.vdk}, and that $\zeta_t=D_t/N_t$. {Recall that $\calf_t=\fix{\{(L_s,N_s,D_s)\}_{s\le t}}$.} \fix{Our first new parameter is a very close approximation of $L_{t+1}-L_t$:}
\be
\br_t=-1+(\dd-1)(k-1)\bar p_t.\lab{def:br}
\ee

This is a key parameter in analyzing the evolution of $(L_t)_{t\ge 0}$. If we view $(L_t)$ as a branching process then $1+\br_t$ approximates the branching parameter of $(L_t)$; \fix{i.e.\ the expected number of new light vertices formed during step $t$ of SLOW-STRIP. Note that for $r=2$, this is exactly the expected number of new light vertices formed when the other endpoint of a deleted edge is heavy. But we need to account for (i) the possibility of both endpoints being light, and (ii)  the fact that when $r>2$ we typically remove multiple heavy vertex-copies.  }
By~\eqn{barpt}  and~\eqn{barpt2}, and noting that $h_{t,i}\leq L_t/(L_t+D_t-(r-2))=O(L_{\fix{t}}/n)$ (by Corollary~\ref{ccoresize}), we have that a.a.s.\ for every $t(B)\le t\le \tau$,
\begin{equation}
\ex(L_{t+1}-L_t\mid \calf_t)=\br_t+O(L_t/n),\quad \ex(L_{t+1}-L_t\mid \calf_t)\le\br_t+O(n^{-1}).\lab{br}
\end{equation}
The second part of (\ref{br}) above is applied when we only require an upper bound on $\ex(L_{t+1}-L_t\mid \calf_t)$. However, in some cases we need a lower bound as well, and we will use the first part.

By~(\ref{taustar}), (\ref{et*}) and~(\ref{br1}), there is a constant $K>0$  such that 
\begin{equation}
\mbox{a.a.s.\ for every}\ t(B)\le t\le \tau,\ \ \br_t\le -Kn^{-\d/2}. \lab{brupper}
\end{equation}

By~\eqn{ztau}, a.a.s.\ $\z_{\tau}=\z+\Theta(n^{-\d/2})$.
 By Lemmas~\ref{l:monotone} and~\ref{l2:monotone} and considering the Taylor expansion of $\psi$ around $\z$, a.a.s.\ 
 \[
 \bar p_{\tau}=\psi(\z_{\tau})+O(n^{-1/2}\log n)=\psi(\z)+\psi'(\z)(\z_{\tau}-\z)+O((\z_{\tau}-\z)^2+n^{-1/2}\log n)= \psi(\z)-\Theta(n^{-\d/2}).
 \]
 By~\eqn{relation2}, $\psi(\z)=1/(r-1)(k-1)$, which implies that a.a.s.\ \be
 \br_{\tau}=-\Theta(n^{-\d/2})\lab{brcore}
 \ee by~\eqn{def:br}.

\fix{
Our second new parameter, $\pi(G_t)$, roughly counts the number of vertices in $G_t$ which are not in the $k$-core and so will be removed before the process ends. Recall from Corollary~\ref{ccoresize} that \aas\ the $k$-core  has size $\a n +(K_1+o(1))n^{1-\d/2}$, for a particular constant $K_1>0$. So the number of non-$k$-core vertices is simply $|V(G_t)|$ minus that number.  

We will prove (Lemma~\ref{lem:pi-br}) that $br_t$ is the same multiplicative order as the negation of the proportion of  vertices in $G_t$ that are not in the $k$-core, until that proportion drops below $n^{-\d/2}$; from that point on, $br_t$ remains at $-\Theta(n^{-\d/2})$ (see~(\ref{brcore})). To reflect this,  we define $\pi(G_t)$ to stay at  $\Theta(n^{1-\delta/2})$ once the remaining number of vertices drop below that. So we define:
\[\pi(G_t):=|V(G_t)|-\alpha n-\frac{K_1}{2} n^{1-\delta/2}.\]
}

\fix{Over the next few subsections, we obtain tight bounds on $\br_t$ in terms of $L_t$ (Lemma~\ref{lem:l-br}); coupled with~(\ref{br}), these will yield our bounds on the expected change in $L_t$.  To do so, we will link $\br_t$ and $\pi_t$, using $\zeta_t$.  Our first step will be to obtain a tighter bound on $\zeta_t$.}

\subsection{Controlling $\zeta_t$}

We list below a few facts that we will use.
Since $c=c_{r,k}+n^{-\d}, \d<\hf$, we can assume by Lemma~\ref{lcoresize2} that there is a $k$-core on a linear number of vertices. At each step of SLOW-STRIP, we remove at most one hyperedge. Thus there is a constant $Q=Q(r,k)$ such that
in every step, the average degree of the heavy vertices is changed by at most $\pm Q/n$; i.e.
\begin{equation}\lab{ezt}
\zeta_{t+1}-\zeta_{t}=O(1/n) \mbox{ uniformly for all } 0\leq t<\tau.
\end{equation}
Therefore, for any $\e>0$, Lemma~\ref{l:B} allows us to choose sufficiently large $B$ such that
$|\zeta_t-\zeta_{\tau}|\le \eps/2$ for all $t\ge t(B)$. Recalling the definition of $\z$ from~(\ref{zeta}), we also know that a.a.s.\ $|\zeta_{\tau}-\zeta|=o(1)$ by Corollary~\ref{ccoresize}. Hence, for all $t\ge t(B)$, $|\zeta_t-\zeta|\le |\zeta_t-\zeta_{\tau}|+|\zeta_{\tau}-\zeta|\le \eps$. This immediately gives (Fa) below.
\begin{description}
\item[(Fa)] For every $\eps>0$, we can choose $B$ sufficiently large such that \aas for all $t\ge t(B)$, $|\zeta_t-\zeta|<\eps$.
\item[(Fb)] {A.a.s.\ }for every $t\ge t(B)$, $\bar p_t=(1+O(n^{-1/2}\log n)) \psi(\zeta_t)$, by Lemma~\ref{l:monotone} {and Theorem~\ref{tkim}}.
\item[(Fc)] We can choose $\eps>0$ sufficiently small so that there are $c_1,c_2>0$ such that $-c_1<\psi'(x)<-c_2$ for all $x$ such that $|x-\zeta|<\eps$ by Lemma~\ref{l2:monotone} and since $\z=\z_{r,k}>k$.
\end{description}
 By (Fa) and (Fc), we may assume $B$ is chosen so that
\begin{description}
\item[(Fc')] $-c_1<\psi'(\zeta_t)<-c_2$ uniformly for all $t\ge t(B)$.
\end{description}
By~(\ref{ezt}) we have
\begin{description}
\item[(Fd)] $\zeta_i=\zeta_t+O((i-t)/n)$ uniformly for every $0\le t\le i\le \tau$.
\end{description}

{We will apply Lemma~\ref{l:azuma} several times. To formalise its application, we need to define normal configurations as in Section~\ref{s.vdk}, and define several stopping times when various a.a.s.\ events such as (Fa) and (Fb) fail. We skip such tedious settings of stopping times, and note that by the same treatment as in Section~\ref{s.vdk}, we will only create an $o(1)$ error in all probability bounds if we apply Lemma~\ref{l:azuma} with events in (Fa)--(Fd) assumed.}

In the next lemma, we prove a more precise form of (Fd).  \fix{This is our strengthening of Lemma~\ref{l:zetaT}.} 

\begin{lemma}\lab{l:zetaseq}
A.a.s.\ for every $t(B)\le t<i\le\tau$,
\begin{enumerate}
\item[(a)] $\zeta_i\le \zeta_t+O(\log n/n)$ uniformly for every $0\le t\le i\le \tau$;
\item[(b)] if $i-t\ge \log n$, then $\zeta_i=\zeta_t-\Theta((i-t)/n)$ uniformly.
\end{enumerate}
\end{lemma}
\begin{proof}
 By Lemma~\ref{l:zetaMartingale}, there exist two constants $\rho_1>\rho_2>0$ such that a.a.s.\ for every $t(B)\le t<\tau$ 
$$
-\frac{\rho_1}{n}\le \ex(\zeta_{t+1}- \zeta_t\mid \calf_t)\le -\frac{\rho_2}{n}.
$$
Moreover, $|\zeta_{t+1}-\zeta_t|=O(1/n)$ by~(\ref{ezt}). We apply Lemma~\ref{l:azuma} with $c_n=O(1/n)$ and $a_n=\rho_2/n$ {and with a similar argument as in the proof of Lemma~\ref{llt2}} to show, for every $i>t$ and $j\ge 0$,
\bea
\pr(\zeta_{i}\ge \zeta_{t}-(i-t)\rho_2/n+j)&\le& \exp(-\Omega(j^2 n^2/(i-t))), \lab{1} \\
\pr(\zeta_{i}\le \zeta_{t}-(i-t)\rho_1/n-j)&\le& \exp(-\Omega(j^2 n^2/(i-t))). \lab{2}
\eea

Then by the union bound (by taking $j=(i-t)\rho_2/2n$ in~\eqn{1} and taking $j=(i-t)\rho_1/2n$ in~\eqn{2} for each $t$ and each $i\ge t+\log n$), we obtain (b). Part (a) follows by (b) and the fact that for each $i\le t+\log n$, we always have $\zeta_i=\zeta_t+O(\log n/n)$ by (Fd).
\end{proof}



\subsection{Relations between $L_t$, $\pi(G_t)$ and $\br_t$}

The next Lemma essentially says that if we can bound the expected change in $L_i$ then we can show $L_i$ is concentrated.  Recall that $\t$ is the stopping time of SLOW-STRIP; i.e.\ the first iteration $t$ for which $L_t=0$.

{Lemma~\ref{llt1} says that there are constants $K,B>0$ such that 
a.a.s.\ for every $t(B)\le t<\tau$,
\be
\ex(L_{t+1}\mid \calf_t )\leq L_t-Kn^{-\d/2}.\lab{cond}
\ee
It is convenient to consider a random process $(\LL_t)_{t\ge t(B)}$ such that $\ex(\LL_{t+1}\mid \LL_t )\leq \LL_t-Kn^{-\d/2}$ for all $t\ge t(B)$. The process $(\LL_t)$ can be defined in various ways depending on each application and in many applications we will let $\LL_t=L_t$ (or sometimes with a shift of the subscript) for all steps in which~\eqn{cond} holds (c.f.\ the proof of Lemma~\ref{llt2} and the arguments in Section~\ref{s.mtub}).
We first prove some a.a.s.\ properties of such processes.
}

\begin{lemma}\lab{lem:Lconcentration} {Let $(\LL_i)_{i\ge 0}$ be a random process and define $T$ to be the minimum integer such that $\LL_T\le 0$.} Suppose {there is a constant $C>0$ and} functions $C>a=a(n)>b=b(n)\geq \Theta(n^{-\d/2})$ such that the following bounds always hold for every $i\geq 0$:
\begin{enumerate}
\item[(i)]  $n^{\d}\log^2 n\le \LL_{0}{\le n}$;
\item[(ii)] {$|\LL_{i+1}-\LL_i|\le C$;}
\item[(ii)] $-a\leq \ex(\LL_{i+1}-\LL_i\mid \fix{\{\LL_s\}_{s\le i}})\leq -b$.
\end{enumerate}
 Then, with probability at least $1-o(n^{-1})$,
 \begin{itemize}
 \item[(a)] $\LL_0-2a i<\LL_i<\LL_0-\hf  b i$ for all $n^{\d}\log^{1.5} n\le i\le {n^2}$;
 \item[(b)] $\LL_i<2\LL_0$ for all $0\le i\le {n^2}$;
 \item[(c)] $\inv{2a}\LL_0<T<\frac{2}{b}\LL_0$.
  \end{itemize}
\end{lemma}

\begin{proof}   We start with the upper bound in part (a).
We will apply Lemma~\ref{l:azuma} with $X_{n,\ell}=\LL_{\ell}-\LL_0$. So we can take $a_n=-b$. {By (ii) we can choose $c_n=C$.}  

Consider $i\geq n^{\d}\log^{1.5} n$. If $\LL_i\geq \LL_0-\hf  bi$ then $X_{n,i}-X_{n,0}=\LL_i-\LL_0\geq i a_n +\hf i b$.  By Lemma~\ref{l:azuma} with $j=\hf i b$, the probability of this is at most
\[\exp\left(-\frac{(\hf i b)^2}{2 i ({C} + |b|)^2}\right)
\leq\exp(-\Omega(b^2 i))\leq\exp(-\Omega(\log^{1.5}n))=o(n^{-3}).\]
The lower bound in part (a) is nearly identical, but this time we apply Lemma~\ref{l:azuma} with $X_{n,i}=\LL_0-\LL_{i}$,  $a_n=a$ and $j=i a$.  Applying the union bound for the at most ${n^2}$ choices for $i$ shows that (a) holds with probability at least $1-o(n^{-1})$.

 For part (b): If $i<\LL_0/{C}$ then the fact that $\LL_{j+1}{\le \LL_j+C}$ implies that $\LL_i<2\LL_0$.  If $i\geq \LL_0/{C}$, {then} $i>  n^{\d}\log^{1.5} n$, then part (a) implies $\LL_i<\LL_0$.

For part (c):  Note that $\frac{2}{b}\LL_0 \ge n^{\d}\log^{1.5} n$, and so part (a) implies that $\LL_i$ reaches 0 for some $i \le \frac{2}{b}\LL_{{0}}$ {(noting here that $\frac{2}{b}\LL_0=o(n^2)$ by assumption (i))}; this yields the upper bound on $T$.  For the lower bound on $T$, we apply the same argument used for the lower bound in part (a) with $i=\inv{2a}\LL_0$  but with $j=\hf \LL_0$.  This time we get
\[\pr(\LL_{i}<\hf \LL_0)<\exp(-\Omega(\LL_0^2/i))\leq\exp(-\Omega(a\LL_0))=o(n^{-2}),\]
thus providing the {desired lower} bound on $T$.

\end{proof}

\fix{Our goal is to bound $\br_t$ in terms of $L_t$.  First we bound it in terms of $\pi(G_t)$.}

\begin{lemma} \lab{lem:pi-br} There are two constants $C_1,C_2>0$ such that
a.a.s.\ $-C_1\pi(G_t)/n\le \br_t\le -C_2\pi(G_t)/n$ for every $t\ge t(B)$.
\end{lemma}
\begin{proof} Recall that $\zeta_t$ denotes the average degree of heavy vertices in $G_t$. Let $t'$ be the maximum integer such that $\pi(G_{t'})= K_1n^{1-\d/2}$, where $K_1$ comes from Corollary~\ref{ccoresize}. Note that $\pi(G_t)$ is a non-increasing function of $t$, which decreases by at most 1 in each step. Hence, for all $t\le t'$, $\pi(G_t)\ge K_1n^{1-\d/2}$ and for all $t> t'$, $\pi(G_t)< K_1n^{1-\d/2}$.

By Corollary~\ref{ccoresize}, a.a.s.\ $\pi(G_{\tau})\sim (K_1/2)n^{1-\d/2}$. Since for every $t$, $|\pi(G_t)-\pi(G_{t+1})|\le 1$ as at most one vertex is removed in each step, we have that a.a.s.\ for all $t\le t'$, $\tau-t\ge \pi(G_{t})-\pi(G_{\tau}) { =\Omega(\pi(G_{t}))}$. {This is because $\pi(G_{\tau})\sim (K_1/2)n^{1-\d/2}$ and so $\pi(G_{\tau})\sim \hf \pi(G_{t'})<\frac23 \pi(G_t)$ for every $t\le t'$. The constant bounds involved in the asymptotic notation above and below will be uniform for all $t$.} In particular, $\tau-t'\ge \pi(G_{t'})-\pi(G_{\tau})\ge (K_1/3) n^{1-\d/2}$. 
On the other hand, for every $t$, $\tau-t\le k(\pi(G_t)-\pi(G_{\tau}))\le k\pi(G_t)$, since every light vertex in the queue $\msq_t$ takes less than $k$ steps to be removed. So a.a.s.\ uniformly for every $t\le t'$, we have
\bea
(K_1/3)n^{1-\d/2}&\le& \tau-t'\le \tau-t \leq k \pi(G_t);\lab{t}\\
\tau-t &=& \Theta(\pi(G_t)). \lab{t2}
 \eea

If $t\leq t'$ then $\t-t { =\Omega}(n^{1-\d/2})>\log n$ by~\eqn{t}.
Thus by Lemma~\ref{l:zetaseq}(b)
we have $\zeta_t=\zeta_{\tau}+\Theta((\tau-t)/n)$. By (Fb) and (Fc') and by~\eqn{def:br},
{a.a.s.\ for all $t(B)\le t\le\tau$}, 
\be
\br_t-\br_{\tau}=\Theta(1)(\bar p_t-\bar p_{\tau})=\Theta(1)(\psi(\zeta_t)-\psi(\zeta_{\tau}))=-\Theta(1)(\zeta_t-\zeta_{\tau}),\lab{br-zeta}
\ee
 and {then by the fact that $\zeta_t=\zeta_{\tau}+\Theta((\tau-t)/n)$ for all $t\le t'$ and}~\eqn{t2},
a.a.s.\ for each $t\leq t'$,
\be
\br_t-\br_{\tau}=-\Theta((\tau-t)/n)=-\Theta(\pi(G_t)/n).\lab{br-zeta2}
\ee
The constants in the asymptotic notations above are uniform for all $t\le \tau$. {By~\eqn{brcore}, a.a.s.\ $\br_{\tau}=-\Theta(n^{-\d/2})$. Since $\pi(G_{\tau})\sim (K_1/2)n^{1-\d/2}$,
it follows then that a.a.s.\ for all $t(B)\le t\le t'$, $\br_t=-\Theta(\pi(G_t)/n)$; note that $-\Theta(n^{-\d/2})$ is absorbed as $\pi(G_t)/n\ge \pi(G_{\tau})/n=\Theta(n^{-\d/2})$ for every $t$. This proves that our lemma holds for all $t\le t'$.
}

Now we consider $t>t'$.  \fix{ Applying~(\ref{t}) with $t:=t'$, we obtain that a.a.s.\ $\tau-t'\leq k\pi(G_{t'})=kK_1n^{1-\d/2}$. This implies that} for all $t'\le t\le\tau$, $\tau-t\le \tau-t'=O(n^{1-\d/2})$. Then by~(\ref{ezt}) we have
\be
\zeta_{\tau}-\zeta_t=O(n^{-\d/2}) \lab{11}
\ee
So by~\eqn{br-zeta}, $\br_t-\br_{\tau}=-\Theta(\z_t-\z_{\tau})=O(n^{-\d/2})$ for all ${t'}\le t\le \tau$.
Then, we must have $\br_t=-\Theta(n^{-\d/2})$ for all ${t'}<t\le \tau$ by~\eqn{brupper} and~\eqn{brcore}.
The definition of $t'$ implies that  $\pi(G_t)/n=\Theta(n^{-\d/2})$ for all $t>t'$.  This {implies that our lemma holds also for $t>t'$}.  
\end{proof}


\fix{And now we are ready  to bound $\br_t$ in terms of $L_t$:}

\begin{lemma}\lab{lem:l-br} There are constants $D_1,D_2,B>0$ such that
a.a.s.\ for all $t\geq t(B)$,
\begin{enumerate}
\item[(a)] for all $t$ such that $L_t\ge n^{1-\d}$, $-D_1\sqrt{L_t/n}\le \br_t\le -D_2\sqrt{L_t/n}$.
\item[(b)] for all $t$ such that $L_t<n^{1-\d}$, $-D_1n^{-\d/2}\le\br_t\le -D_2n^{-\d/2}$.
\end{enumerate}
\end{lemma}

\fix{
Combining Lemma~\ref{lem:l-br} with~(\ref{br}), and applying a slight adjustment to $D_1,D_2$  yields our strengthening of Lemma~\ref{llt1}:

\begin{corollary}\lab{cor:l-br} There are constants $D_1,D_2,B>0$ such that
a.a.s.\ for all $t\geq t(B)$,
\begin{enumerate}
\item[(a)] for all $t$ such that $L_t\ge n^{1-\d}$, $-D_1\sqrt{L_t/n}\le \ex(L_{t+1}-L_t\mid \calf_t)\le -D_2\sqrt{L_t/n}$.
\item[(b)] for all $t$ such that $L_t<n^{1-\d}$, $-D_1n^{-\d/2}\le\ex(L_{t+1}-L_t\mid \calf_t)\le -D_2n^{-\d/2}$.
\end{enumerate}
\end{corollary}

This strenghtening is what we need to prove Lemma~\ref{lsi}(a,b).}

{\em Proof of Lemma~\ref{lem:l-br} } We take $B$ large enough so that the relevant preceding results hold.

 We have $\br_t\leq -Kn^{-\d/2}$ for some constant $K>0$, for all $t(B)\le t\le \t$ by~\eqn{brupper}. {For any $t(B)\le t\le \t$, we will consider $i$ such that $t\le i\le \t$.}  By Lemma~\ref{l:zetaseq}(a), a.a.s.\ {for any such pair of $t,i$}, we have $\zeta_i\le \zeta_t+O(\log n/n)$. Thus by (Fb) and (Fc') and the definition of $\br_i$,  a.a.s.\ for all {pairs of $t,i$ such that} $t(B)\leq t<i\leq\t$,
\bean
\br_i&=&\br_t -\Theta(1)(\bar p_t-\bar p_i)=\br_t-\Theta(1)(\psi(\z_t)-\psi(\z_i))=\br_t-\Theta(1)(\z_i-\z_t)\\
&\ge& \br_t-\Theta\left(\frac{\log n}{n}\right)\ge 2\br_t,
\eean
since $\log n/n=o(|\br_t|)$. By~\eqn{br}, a.a.s.\ for every $t(B)\leq t<i\leq\t$,
\[
 2\br_t-O(L_i/n)\leq \ex(L_{i+1}-L_i\mid \calf_i)\leq \br_i+O(n^{-1}).
\]
We know that a.a.s.\ for every $t(B)\leq t<i\leq\t$, we have $\br_t,\br_i\leq -K n^{-\d/2}$  by~\eqn{brupper}. So for some constant $A>0$, we have a.a.s. for all  $t(B)\leq t<i\leq \t$:
\be
2\br_t-AL_i/n\le \ex(L_{i+1}-L_i\mid \calf_i)\le -\frac{K}{2}n^{-\d/2}.\lab{ebr*}
\ee

Let $t_1$ be the smallest $t$ such that  $L_{t+1} <n^{\d}\log^2 n$. We first prove the lemma for all $t\le t_1$. 

Note that since $\br_t\geq -1$ (by~\ref{def:br}), we can take $B$ large enough so that for $i\geq B$, $L_i/n$ is small enough that the LHS of~\eqn{ebr*} is at least -3. So we let $\tau^*$ denote the first step that the condition $-3<\ex(L_{i+1}-L_i\mid \calf_i)\le -\frac{K}{2}n^{-\d/2}$ fails; define $\tau^*=\tau$ if it never fails. Assume $t\le t_1$.  
 Define $(\LL_j)_{j\ge 0}$ by letting $\LL_0=L_t$ and $\LL_{i-t}=L_i$ for each $t<i\le \tau^*$ and $\LL_{i+1-t}=\LL_{i-t}-Kn^{-\d/2}$  for all $i\ge \tau^*$. Moreover,~\eqn{ebr*} implies that a.a.s.\ $\LL_{i-t}=L_i$ for all $t\le i\le \tau$. 

Now the process $(\LL_i)$ satisfies the hypotheses in Lemma~\ref{lem:Lconcentration}(i) (by the definition of $t_1$), (ii) with $C=kr$ and (iii) with $a=3$ and $b=Kn^{-\d/2}$. Now by Lemma~\ref{lem:Lconcentration}(b) (and by taking the union bound over $t$), \aas  $L_i<2L_{t}$ (corresponding to $\LL_{i-t}<2\LL_0$ for each given $t$) for all pairs of $i,t$ such that $t(B)\le t\le i<\tau$. Then~\eqn{ebr*} becomes
 \be
2\br_t-2AL_t/n\le \ex(L_{i+1}-L_i\mid \calf_i)\le -\frac{K}{2}n^{-\d/2}.\lab{ebr**}
\ee
So we can apply Lemma~\ref{lem:Lconcentration} again to $(\LL_i)_{i\ge 0}$ defined as above (except that we modify $\tau^*$ to reflect~\eqn{ebr**} rather than~\eqn{ebr*}) with $a=-2\br_t+2AL_t/n$.  Letting $T$ denote the first step that $\LL_i$ becomes at most zero,  Lemma~\ref{lem:Lconcentration}(c) shows that with probability $1-o(n^{-1})$, $T\geq  4A L_t/(|\br_t|+L_t/n)$; i.e.\ $\t\geq t+4A L_t/(|\br_t|+L_t/n)$.  Taking the union bound over $t(B)\le t\le t_1$ and using that by~\eqn{ebr**} a.a.s.\ $\LL_{i-t}=L_t$ for all $t(B)\le t<i\le\tau$ shows that a.a.s.\ for all $t(B)\leq t\leq t_1$:
\[\tau\geq  t+4A L_t/(|\br_t|+L_t/n).\]

We delete at least one vertex for every $k-1$ steps of SLOW-STRIP, and so  $|V(G_t)\setminus V(G_{\tau})|\geq (\t-t)/(k-1)$. Therefore,  applying Corollary~\ref{ccoresize}, we have that \aas for all $t(B)\leq t\leq t_1$:
\bea
\pi(G_t)&=&|V(G_t)|-\alpha n-\frac{K_1}{2} n^{1-\delta/2}
\geq|V(G_t)| -|V(G_{\tau})|+\frac{K_1}{3}n^{1-\d/2}\nonumber\\
&\geq&\frac{4A L_t}{(k-1)(|\br_t|+L_t/n)}+\frac{K_1}{3}n^{1-\d/2}. \lab{epgt}
\eea
By Lemma~\ref{lem:pi-br}, there exists a constant $C_2>0$ such that a.a.s.\ for all $t(B)\leq t\leq t_1$,
\[
\br_t\le -C_2\frac{\pi(G_t)}{n}\le-C_2 \left(\frac{4A L_t/n}{(k-1)(|\br_t|+L_t/n)}+\frac{K_1}{3}n^{-\d/2}\right),
\]
and so
\be
|\br_t|\ge C_2 \max\left(\frac{4A L_t/n}{(k-1)(|\br_t|+L_t/n)},\frac{K_1}{3}n^{-\d/2}\right).\lab{sigma1}
\ee
Taking $B$ large enough that $L_t/n$ is sufficiently small for $t\geq t(B)$ (by Lemma~\ref{l:B}), and rearranging~\eqn{sigma1}, there is a constant $D_2>0$, such that a.a.s. for all $t(B)\leq t\leq t_1$:
\be\lab{egoal1}
|\br_t|\geq D_2\max\left( \sqrt{\frac{L_t}{n}},n^{-\d/2}\right).
\ee
This yields the upper bounds in our lemma for $t\leq t_1$.
Next we prove the lower bounds; i.e., we wish to prove that for some constant $D_1>0$ for all $t(B)\leq t<i\leq t_1$:
\be\lab{egoal}
|\br_t|\leq D_1\max\left( \sqrt{\frac{L_t}{n}},n^{-\d/2}\right).
\ee
The proof is similar to that for the upper bound so we briefly describe the arguments.
Let $A_2,A_3{>0}$ be the implicit constants in (Fd),~\eqn{br-zeta}, and set $A_1=1/(2A_2A_3)$.
Applying~\eqn{br-zeta},~(Fd)  we get that a.a.s.\ for any $t(B)\leq t\leq i\le \min\{t+A_1|\br_t|n,\t\}$:
\[\br_i-\br_t\leq A_3(\z_{t}-\z_i)\leq A_2A_3((i-t)/n)\leq A_1A_2A_3|\br_t|=-\hf\br_t,\]
{where the first inequality above comes from {(Fb), (Fc') and ~\eqn{def:br} (in the same way that~\eqn{br-zeta} is deduced)}, the second inequality follows by (Fd) and the third inequality follows by  the upper bound of $i$; the equality holds by the definition of $A_1$.}
Thus, $\br_i\le \hf\br_t$. So, by~\eqn{br}:
\be\lab{eggg}
\ex(L_{i+1}\mid \calf_i)- L_i\le \frac{1}{2}\br_t+O(1/n),\quad \forall t(B)\leq t< i\le \min\{t+A_1|\br_t|n,\t\}.
\ee
We can assume $|\br_t|> \sqrt{(8/A_1)L_t/n}$, as otherwise~(\ref{egoal}) holds with $D_1=\sqrt{8/A_1}$. 
Therefore we have (i) $8L_t/|\br_t|<A_1|\br_t|n$ and (ii) the RHS of~(\ref{eggg}) is at most $\frac{1}{4}\br_t$ (noting that $|br_t|>>1/n$ by~\eqn{brupper}).

Define $(\LL_i)_{i\ge 0}$ analogously to how it was defined above,
and apply Lemma~\ref{lem:Lconcentration}(c) with $b=\frac{1}{4}\br_t$ (and with $a=-2\br_t+2AL_t/n$, as argued above). By the a.a.s.\ correspondence between $\LL_{i-t}$ and $L_t$, and noting that (i) above allows us to apply~\eqn{eggg} for all $i\leq \min\{t+8L_t/|\br_t|,\t\}$, we can show that with probability $1-o(n^{-1})$,  the stopping time
\be\lab{esay*}
\tau<t+\frac{2}{b}L_t=t+8L_t/|\br_t|.
\ee
Taking the union bound over $t(B)\le t\le t_1$ shows that \aas~\eqn{esay*} holds for all $t(B)\leq t\leq t_1$ for which $|\br_t|> \sqrt{(8/A_1)L_t/n}$.

Since we remove at most one vertex during each iteration of SLOW-STRIP, it follows that $G_t$ contains at most $8L_t/|\br_t|$ non-$k$-core vertices by~\eqn{esay*}.
Recalling that $\pi(G_t)$ is approximately the number of non-$k$-core vertices in $G_t$ plus $\hf K_1 n^{-\d/2}$, this implies that
$\pi(G_t)\le A_4(L_t/|\br_t|+ n^{1-\d/2})$ for some constant $A_4>\max\{8,\hf K_1\}$.
Now, by Lemma~\ref{lem:pi-br}, we have that there is a constant $C_1>0$ such that a.a.s.\ for all $t(B)\leq t\leq t_1$ with $|\br_t|> \sqrt{(8/A_1)L_t/n}$ we have:
\[\br_t\ge -C_1\pi(G_t)/n \ge -C_1A_4\left(\frac{L_t/n}{|\br_t|}+ n^{-\d/2}\right).\]
This implies~(\ref{egoal}) for $t(B)\leq t\leq t_1$ with $D_1=\max\{\sqrt{8/A_1},C_1A_4\}$.

Now we consider $t>t_1$. Since $L_{t_1}\ge n^{\d}\log^2 n$, by~\eqn{esay*}, a.a.s.\ $\tau-t_1<8L_{t_1}/|\br_{t_1}|$. By the definition of $t_1$ we also have that $L_{t_1}\leq L_{t_1+1}+r<2n^{\d}\log^2 n$. (\ref{egoal1}) says that a.a.s.\ $|\br_{t_1}|\ge D_2 n^{-\d/2}$. So we obtain: a.a.s.\ for all $t>t_1$,
\[
\tau-t<\tau-t_1<(16/D_2)n^{3\d/2}\log^2n.
\]
This implies that a.a.s.\ for all $t_1<t\le \tau$, $t-t_1<(16/D_{\fix{2}})n^{3\d/2}\log^2n$. Then, recalling that $A_2,A_3$ are the implict constants from (Fd) and~\eqn{br-zeta}, and since $|\br_{t_1}|\geq D_2 n^{-\d/2}$  (from~\eqn{egoal1}), a.a.s.\ for all $t_1<t\le \tau$,
\[|\br_t-\br_{t_1}|\leq A_2A_3(t-t_1)/n\le A_2A_3(16/D_1)n^{3\d/2-1}\log^2n=o(n^{-\d/2})=o(\br_{t_1}),
 \]
 as $\d<1/2$. Since~\eqn{egoal1},~\eqn{egoal} hold for $t=t_1$ then by increasing $D_1$ and decreasing $D_2$, these bounds hold for $t>t_1$, and hence for all $t(B)\leq t\leq\t$.
\proofend

\subsection{Proof of the key lemma}

And now we can prove Lemma~\ref{lsi}, which we restate.  

Recall that we are carrying out SLOW-STRIP on $AP_r(n,m)$, and so implicitly, this yields the steps of the parallel stripping process applied to $AP_r(n,m)$.  $\imax$ is the number of iterations carried out by the parallel stripping process, and $S_i$ is the set of vertices removed in iteration $i$. The number of hyperedges removed in iteration $i$ that contain $u$ is $\fix{d^+}(u)$ for each $u\in S_i$, and is $d^-(u)$ for each $u\in S_{i+1}$ (if $i<\imax$).

\newtheorem*{lsi}{Lemma~\ref{lsi}}
\begin{lsi} There exist constants $B,Y_1,Y_2,Z_1$ dependent only on $r,k$, such that \aas for every $ B\leq i< \imax$ with $|S_i|\ge n^{\d}\log^2 n$:
\begin{enumerate}
\item[(a)] if $|S_i|< n^{1-\d}$ then $(1-Y_1n^{-\d/2})|S_i|\leq |S_{i+1}|\leq  (1-Y_2n^{-\d/2})|S_i|$;
\item[(b)] if $|S_i|\geq n^{1-\d}$ then $(1-Y_1\sqrt{\frac{|S_i|}{n}})|S_i|\leq |S_{i+1}|\leq  (1-Y_2\sqrt{\frac{|S_i|}{n}})|S_i|$;
\item[(c)] $\sum_{j\ge i}|S_j|\le {Z_1}|S_i|n^{\d/2}$.
\item[(d)] $|S_i|\leq\sum_{u\in S_{i}} d^-(u) < |S_{i}|+Z_1\frac{|S_{i}|^2}{n}+\log^2n$;
\item[(e)] $|S_{i+1}|\leq \sum_{a,b}abM_i^{a,b}\leq |S_{i+1}|+Z_1\frac{|S_{i}|^2}{n}+\log^2n$;
\item[(f)]$\sum_{a\geq 2,b\leq r-a}abM_i^{a,b}\leq Z_1\frac{|S_{i}|^2}{n}+\log^2n$;   
\item[(g)] $\sum_{u\in S_i}\fix{d^+}(u)d^-(u)\leq\sum_{u\in S_i}\fix{d^+}(u)+Z_1\frac{|S_{i}|^2}{n}+\log^2n$;
\item[(h)] $\sum_{u\in S_{i}} (d^-(u))^2\le Z_1|S_i|$;
\item[(i)] $d^-(u)<\log n$ for all $u\in\cup_{i=2}^{\imax}S_i$.
\end{enumerate}
\end{lsi}

\proof We take $B$ large enough so that the relevant preceding results hold.

{\bf Part (i):} We allocate $rcn$ vertex-copies to \fix{$n$} bins. So the probability that at least one bin has size at least $\log n$ is at most $n\pr[\Bin(rcn, \inv{n})\geq\log n]$ which is easily computed to be $o(1)$.  Part (i) follows since $d^-(u)$ is less than the size of bin $u$.

{\bf Part (d):}  We first prove (d) assuming (e).
The first inequality follows since $d^-(u)\geq 1$ for all $u\in S_i$ (with $i\geq 2$).
For the second inequality, note that $\sum_{u\in S_{i}} d^-(u)= \sum_{a\geq 1,b\leq r-a}bM_{i-1}^{a,b} \leq\sum_{a,b}abM_{i-1}^{a,b}$, and so the bound follows from  Lemma~\ref{lsi}(e) \fix{and the fact that a.a.s.\ $|S_{i+1}|=\Theta(|S_i|)$ by Lemma~\ref{lsi}(a,b)}. \smallskip

{\bf Part (f):} $\sum_{a\ge 2,b\le r-a}abM_i^{a,b}\le r^2\sum_{a\ge 2,1\le b\le r-a}M_i^{a,b}=r^2 Y$, where $Y$ is the number of hyperedges {(i.e.\ $r$-tuples)} that contains at least two {vertex-copies} in $S_i$ and at least one vertex-copy in $S_{i+1}$.

Consider SLOW-STRIP from step $t(i)$ to $t(i+1)-1$, i.e.\  the steps when the vertices in $S_{i}$ are removed. When  SLOW-STRIP removes a hyperedge incident with $u\in S_i$, it removes a vertex-copy from $u$ and another $\dd-1$ vertex-copies chosen uniformly at random.  This edge counts towards $Y$ \fix{only if} at least one of these $\dd-1$ vertex-copies are from $S_i$. Regardless of what happened during the removal of previous vertices from $S_i$, there are at most $(k-1)|S_i|$ remaining vertex-copies from $S_i$, and there are a total of $\Theta(n)$ available vertex-copies (by Corollary~\ref{ccoresize}). So the probability that at least one vertex-copy is selected from $S_i$ is at most $O(|S_i|/n)$. We remove at most $(k-1)|S_i|$ hyperedges during this phase.  So $Y$ is dominated in distribution by a binomial variable $\Bin((k-1)|S_i|, O(|S_i|/n))$ and so the Chernoff bound yields that with probability at least $1-n^{-2}$, $Y=O(|S_i|\cdot |S_i|/n+\log n)$. Multiplying by the at most $n$ choices for $i$ completes the proof for part (f).\smallskip

{\bf Part (e):} 
The lower bound is trivial as every vertex in $S_{i+1}$ is incident with at least one hyperedge counted by $\sum_{a,b}M_i^{a,b}$, and each $(a,b)$-hyperedge contains $b$ vertices of $S_{i+1}$.
We now prove the upper bound.

Let {$\Lambda$} denote the sum of $d^-(u)-1$ over all $u\in S_{i+1}$ with $d^-(u)\geq 2$.  We note that
\bea
\sum_{a,b}abM_i^{a,b}&\leq&\sum_{b}bM_i^{1,b}+\sum_{a\ge 2,b\le r-a}abM_i^{a,b}\\
&=&\sum_{u\in S_{i+1}}d^-(u)+O(|S_i|^2/n)\qquad\mbox{ by part (f) }\\
&=& |S_{i+1}|+  {\Lambda}  + O(|S_i|^2/n)\qquad\mbox{since $d^-(u)\geq 1$ for all $u\in S_{i+1}$.}
\eea

{To bound $\Lambda$,  we bound $X_j$ which is defined to be the number of vertices $u\in S_{i+1}$ with $d^-(u)\geq j$.  
To do so, we expose the degree sequence of the configuration remaining after $i$ iterations of the parallel stripping process, and then consider choosing the configuration using Bollob\'as' configuration model.  We emphasize that this exposure of the degree sequence, and choice of the configuration are only carried out for the purposes of the proof of~(\ref{exj}) below; the exposure does not carry on outside of this lemma - in particular to the analysis in Section~\ref{sec:SS1}.

If a vertex $u$ is counted by $X_j$, then we must have $j\leq \deg(u)\leq k-1+j$, where $\deg(u)$ is its degree in the remaining configuration, and $j$ copies of $u$ must be selected for removal.  At most $(r-1)(k-1)|S_i|$ vertex-copies are randomly selected for removal, and there are at least $\a n$ vertex-copies to choose from (by Corollary~\ref{ccoresize}). So the probability that $u$ is counted by $X_j$ is at most
\[{\deg(u)\choose j}\left(\frac{(r-1)(k-1)|S_i|}{\a n}\right)^j\leq {k-1+j\choose j}\left(\frac{(r-1)(k-1)|S_i|}{\a n}\right)^j
<\left(Z\frac{|S_i|}{4n}\right)^j,\]
for a suitable constant $Z=Z(r,k)$.

The probability that $u,u'$ both count towards $X_j$ are negatively correlated - if $u$ counts towards $X_j$ then at least $j$ copies are chosen from $u$ and hence are not chosen from $u'$. (Note: the reason that $X_j$ is defined to be the number of $u\in S_{i+1}$ with $d^-(u)\geq j$ rather than $d^-(u)= j$ is that it makes this negative correlation easier to see.) Noting that there are at most $n$ vertices with $j\leq \deg(u)\leq k-1+j$, this implies that $X_j$ is dominated by the binomial $BIN(n,2^{-j}\left(Z\frac{|S_i|}{n}\right)^j)$.  Applying the Chernoff Bound and summing over all choices of $j$ implies that with probability at least $1-n^{-2}$
\be\label{exj}
\mbox{for all $j\geq 2$:} \qquad X_j\leq n\left(Z\frac{|S_i|}{2n}\right)^j+\log n<\frac{Z}{2^j}\frac{|S_i|^2}{n}.
\ee

(\ref{exj}) implies that $\Lambda\leq \sum_{j\geq 2} (j-1)X_j=O(|S_i|^2/n +\log n)$. Multiplying by the at most $n$ choices for $i$ completes the proof for part (e)  and hence of part (d). \smallskip
}

{\bf Part (g):}  Here, we focus again on ${\Lambda}$ from part (e), this time using the value from iteration $i-1$.  So ${\Lambda}$ is the sum of $d^-(u)-1$ over all $u\in S_{i}$ with $d^-(u)\geq 2$.
As we showed in part (e), a.a.s.\ ${\Lambda}<Z  |S_i|^2/n +\log^2 n$ (for every $i$).  Part (g) now follows since:
\[\sum_{u\in S_i} \fix{d^+}(u)d^-(u)=\sum_{u\in S_i} \fix{d^+}(u) + \sum_{u\in S_i} \fix{d^+}(u)(d^-(u)-1)\leq \sum_{u\in S_i} \fix{d^+}(u) + (k-1){\Lambda}.\]
\smallskip

For the remaining parts, we focus on:
\[\ld_i=\sum_{v\in S_i} \fix{d^+}(v); \qquad \mbox{ i.e. } \ld_i=L_{t(i)}.\]
We prove the following relation between $\ld_i$ and $|S_i|$: {a.a.s.\ for all $i$ under the assumptions of this lemma,}
\be\lab{ld-S}
\ld_i= (k-1)|S_i|(1+O(|S_i|/n+\log^2 n/|S_i|)).
\ee
Since $\fix{d^+}(v)\le k-1$ for every $v\in S_i$, immediately we have $\ld_i\le (k-1)|S_i|$. If a vertex $v$ in $S_i$ has $\fix{d^+}(v)<k-1$ then $d^-(v)$ must be at least two since otherwise $v$ should have been removed before the $i$-th iteration of the process. Let $I_2$ denote the 
set of vertices $v\in S_i$ with $d^-(v)\ge 2$ then
\[
\ld_i\ge (k-1)|S_i\setminus I_2|.
\]
By (d), 
\[
|S_i\setminus I_2|+2|I_2|\le \sum_{u\in S_i}d^-(u) \le |S_i|+O(|S_i|^2/n+\log^2 n).
\]
It yields 
$
|I_2|=O(|S_i|^2/n+\log^2 n)$ and so $|S_i\setminus I_2|=|S_i|+O(|S_i|^2/n+\log^2 n)=|S_i|(1+O(|S_i|/n+\log^2 n/|S_i|))$.
It follows then that $\ld_i\ge (k-1)|S_i|(1+O(|S_i|/n+\log^2 n/|S_i|))$. Combining with the upper bound that $\ld_i\le (k-1)|S_i|$,~\eqn{ld-S} follows.  \smallskip

{\bf Parts (a,b):} For (a), we have  $\ld_i=\Omega(n^{\d}\log^2 n)$ and $\ld_i=O(n^{1-\d})$ by the assumptions of the lemma and by~\eqn{ld-S}. We will bound $\ex(L_{j+1}-L_j\mid \calf_j)$ in order to apply Lemma~\ref{lem:Lconcentration} as in the proof for Lemma~\ref{lem:l-br}.

Since SLOW-STRIP  removes at least one and at most $r$ vertex-copies from $S_i$ in every iteration $t(i)\le j<t(i+1)$, we have $t(i+1)-t(i)=\Theta(\ld_i)$, uniformly over $i$ satisfying the conditions of the lemma. So
$t(i+1)>t(i)+n^{\d}\log^{1.5}n$.  Furthermore, by (Fd),
for all $t(i)\le j\le t(i+1)-1$, $\zeta_j-\zeta_{t(i)}=O((j-t(i))/n)=O(\ld_i/n)$, and so by~\eqn{br-zeta}, we have $\br_j=\br_{t(i)}+O(\ld_i/n)=\br_{t(i)}+O(n^{-\d})$.  \fix{Corollary~\ref{lem:l-br}(b) says {that a.a.s.}:}
\be\lab{eabab}
-D_1n^{-\d/2}\le\ex(L_{j+1}-L_j\mid \calf_j)\le - D_2n^{-\d/2},
\ee
for all $t(i)\le j\le t(i+1)-1$. This allows us to apply  Lemma~\ref{lem:Lconcentration} in a similar manner as in the proof of Lemma~\ref{lem:l-br}. I.e.\ we define a process $\LL_j$ that is equal to $L_{j+t(i)}$ so long as~\eqn{eabab} holds, and such that we can apply Lemma~\ref{lem:Lconcentration} to $\LL_j$; then we translate what this says about $\LL_j$ to $L_j$ since~\eqn{eabab} \aas\ holds for all $j$.  This yields that for each $i$ satisfying the conditions of the lemma, with probability $1-o(n^{-1})$,
\[L_{t(i+1)}=L_{t(i)}-\Theta(n^{-\d/2})(t(i+1)-t(i))=L_{t(i)}(1-\Theta(n^{-\d/2})).\]
Taking the union bound over all $i$ yields part (a) with $|S_i|$ replaced by $\ld_i$. Then part (a) follows by~\eqn{ld-S} by noting that {$n^{-\d/2}$}  dominates the other errors $|S_i|/n$ and $\log^2 n/|S_i|$ for $|S_i|$ in this range. \smallskip

The proof of part (b) is nearly identical, applying Lemma~\ref{lem:l-br}(a)  rather than Lemma~\ref{lem:l-br}(b).\smallskip

{\bf Part (c):} Recall that $\t$ is the  stopping time of  SLOW-STRIP. As  in the proofs of parts (a,b), we will apply Lemma~\ref{lem:Lconcentration}.  There are two possible ranges for $L_j$ and in both, we can take $b=\Theta(n^{-\d/2})$: if $L_j<n^{1-\d}$ we use $b=\Theta(n^{-\d/2})$;  if $L_j\geq n^{1-\d}$ we use $b=\Theta(\sqrt{L_j/n})\geq \Theta(n^{-\d/2})$.  Lemma~\ref{lem:Lconcentration}(c)  yields that with probability $1-o(n^{-1})$, $\LL_{t}\leq \LL_0+t \Theta(n^{-\d/2})$ for each $t\geq n^{\d}\log^{1.5}n$, which implies (since \whp\  $L_{t(i)+t}=\LL_t$ until step $\t$ when $\LL_t$ drops to zero)
$\t-t(i)\leq \Theta(n^{\d/2})\ld_i$.

 Since the total degree in $\msq_t$ decreases by at most $\dd$ in every iteration of SLOW-STRIP, we have
\[\sum_{j\ge i}\ld_j\le \dd (\tau-t(i))  \leq \Theta(n^{\d/2})\ld_i.\]
Taking the union bound over all $i$ and then applying~\eqn{ld-S} yields part (c).\smallskip

{\bf Part (h):}  {By~(\ref{exj}) applied to iteration $i-1$ and using parts (a,b), we have that with probability at least $1-n^{-2}$ there is a constant $Z_1$ such that:
\be\label{e.lsih}
\sum_{u\in S_i}(d^-(u))^2 \leq |S_i|+\sum_{j\geq 2}(j^2-1)X_j<|S_i| + O(|S_{i-1}|)<Z_1|S_i|.
\ee
Part (h) then follows} by multiplying {the failing probability of at most $n^{-2}$} by the at most $n$ choices for $i$. \smallskip

\proofend

\fix{
\no {\bf Acknowledgement}   We thank two anonymous referees for their careful reading and several corrections.
}

\newpage
\no{\Large \bf Appendix}

\begin{proof}[Proof of Lemma~\ref{l:diff}.] First we bound $y:=\mu(c)-\mu_{r,k}$. Recall that $\mu(c)$ is the larger root of $h(\mu)=c$ and $\mu_{\dd,k}$ is
the unique root of $h(\mu)=c_{\dd,k}$. As $h(x)$ is convex over $x\in (0,+\infty)$ and has derivative $0$ at $x=\mu_{\dd,k}$.
The Taylor expansion at $\mu=\mu_{\dd,k}$ gives
$$
c=h(\mu_{\dd,k})+\frac{h''(\mu_{\dd,k})}{2}y^2+O(y^3)=c_{\dd,k}+\frac{h''(\mu_{\dd,k})}{2}y^2(1+O(y)).
$$
Thus, $y=\sqrt{2/h''(\mu_{\dd,k})}n^{-\d/2}+O(n^{-\d})=K_1n^{-\d/2}+O(n^{-\d})$, by letting $K_1=\sqrt{2/h''(\mu_{\dd,k})}$.
Next, we bound
$$
f_k(\mu(c))-\alpha_{\dd,k}=f_k(\mu(c))-f_k(\mu_{\dd,k})=f_k'(\mu_{\dd,k})y+O(y^2).
$$
Recall that
$$
f_k(x)=e^{-x}\sum_{i\ge k} x^i/i!.
$$
Thus, $f'_k(x)=e^{-x}x^{k-1}/(k-1)!$. Hence, $f'_k(\mu_{\dd,k})>0$. It follows then that there is a constant $K_2>0$, such that
$$
f_k(\mu(c))-\alpha=K_2n^{-\d/2}+O(n^{-\d}).
$$
Similarly, there is a $K_3>0$ such that
\[
\frac{1}{\dd}\mu(c)f_{k-1}(\mu(c))-\beta=K_3n^{-\d/2}+O(n^{-\d}).\qedhere
\]
\end{proof}

\begin{proof}[Proof of Lemma~\ref{l:degreeK}.]
 By the definitions of $\mu_{\dd,k}$ and $h(\mu)$ from~\eqn{murk}, $h'(\mu_{\dd,k})=0$. Since
$$
h'(x)=\frac{f_{k-1}(x)^{\dd-1}-x(\dd-1)f_{k-1}(x)^{\dd-2}f'_{k-1}(x)}{f_{k-1}(x)^{2(\dd-1)}},
$$
and $f'_k(x)=f_{k-1}(x)-f_k(x)$ for all $k\ge 1$, we have
$$
f_{k-1}(\mu_{\dd,k})=\mu_{\dd,k}(\dd-1)(f_{k-2}(\mu_{\dd,k})-f_{k-1}(\mu_{\dd,k})),
$$
i.e.,
$$
\frac{f_{k-2}(\mu_{\dd,k})}{f_{k-1}(\mu_{\dd,k})}=\frac{1+\mu_{\dd,k}(\dd-1)}{\mu_{\dd,k}(\dd-1)}=1+\frac{1}{\mu_{\dd,k}(\dd-1)}.
$$
On the other hand,
$$
f_{k-2}(\mu_{\dd,k})=f_{k-1}(\mu_{\dd,k})+e^{-\mu_{\dd,k}}\frac{\mu_{\dd,k}^{k-2}}{(k-2)!},
$$
it follows immediately that
\be
e^{-\mu_{\dd,k}}\frac{\mu_{\dd,k}^{k-2}}{(k-2)!f_{k-1}(\mu_{\dd,k})}=\frac{1}{\mu_{\dd,k}(\dd-1)}.\lab{critical}
\ee
By the definition of $\bar\rho_{\dd,k}(k)$, $\alpha$ and $\beta$, the assertion follows thereby.
\end{proof}

\begin{proof}[Proof of Lemma~\ref{lem:AP}]

The Poisson Cloning Model for a random $r$-uniform configuration on $n$ vertices, $H_{PC}(n,p;r)$  is defined in~\cite{jhk} as follows: 
For each vertex $v$, select $d(v)$ to be a Poisson variable with mean $p{n-1\choose r-1}$ where these $n$ Poisson variables are independent.  Create $d(v)$ copies (i.e. clones) of each vertex $v$. If $D=\sum d(v)$ is a multiple of $r$ then take a uniformly random partition of the vertex-copies into parts of size $r$, as in the AP-model.  If $D$ is not a multiple of $r$ then choose one of the parts to have size $D \mod r$, and so the resulting hypergraph will have one edge of size less than $r$.  Note that the probability that $r$ divides $D$ is bounded from below by a positive constant.  This implies that if a property $Q$ holds \whp\ for $H_{PC}(n,p;r)$ then it holds \whp\ when conditioning on the event that $r$ divides $D$.      

\cite{jhk} notes that this model is equivalent to: Choose $D$ to be a Poisson variable with mean $np{n-1\choose r-1}$, create $D$ vertex-copies, and assign each vertex-copy to a uniformly chosen vertex.  Then take a uniformly random partition as described above.   Thus, the AP-model $AP_r(n,m)$ is exactly $H_{PC}(n,p;r)$ conditioned on $D=rm$.  

We say that a property $Q$ of configurations is monotone decreasing (resp. increasing) if:  Consider any configuration 
for which $Q$ holds.  Create a new configuration by removing any part and its vertex-copies (resp. adding a part containing $r$ new vertex-copies and assigning those vertex-copies to any bin).  Then $Q$ holds for the new configuration.   We will argue that if a monotone property $Q$ holds \whp\ in  $H_{PC}(n,p;r)$ then $Q$ holds \whp\ in 
$AP_r(n,m)$ if $m=\frac{np}{r}{n-1\choose r-1}$.

Note that $AP_r(n,m-1)$ is equivalent to: choose a configuration from $AP_r(n,m)$, select a uniformly random part, and remove it along with its vertex-copies.  It follows that if $Q$ is monotone increasing (resp.\ decreasing) then the probability that $Q$ holds in $AP_r(n,m)$ is at least (resp.\ at most) the probability that $Q$ holds in $AP_r(n,m-1)$.
Suppose $Q$ is monotone increasing and that $Q$ holds \whp\ in  $H_{PC}(n,p;r)$.  Note that the probability that $D=Po(p{n-1\choose r-1})$ is a multiple of $r$ that is at least $p{n-1\choose r-1}$ is bounded  from below by a positive constant.  This implies that $Q$ holds \whp\ for $H_{PC}(n,p;r)$ when conditioning that $D$ is a multiple of $r$ that is at least $p{n-1\choose r-1}$.  
Since $AP_r(n,m)$ is $H_{PC}(n,p;r)$ conditioned on $D=rm=p{n-1\choose r-1}$ and since the probability that $Q$ holds in $AP_r(n,m)$ is monotone increasing with $m$ it follows that $Q$ holds \whp\ for $H_{PC}(n,p;r)$ when conditioning that $D=p{n-1\choose r-1}$, i.e.\ for  $AP_r(n,m)$.  The argument for $Q$ monotone decreasing is almost identical.

\cite{jhk} proves that Lemma~\ref{lcoresize2} holds for $H_{PC}(n,p;r)$ where $\frac{np}{r}{n-1\choose r-1}=m$.  Note that each part of Lemma~\ref{lcoresize2} is the intersection of two monotone properties, one increasing and the other decreasing.  For example, the first part says that there is a positive function $g(n)=O(n^{3/4})$ such that (i) the $k$-core has at least $\a(c)n - g(n)$ vertices and (ii) the $k$-core has at most $\a(c)n + g(n)$ vertices.  Since all of these properties hold \whp\ for  $H_{PC}(n,p;r)$ they hold \whp\ for $AP_r(n,m)$.   Corollary~\ref{ccoresize} follows from Lemma~\ref{lcoresize2} and Lemma~\ref{l:diff}, so it also holds \whp\ for   $AP_r(n,m)$.

\end{proof}

\begin{proof}[Proof of Lemma~\ref{l:gk}.]
Set $f(x)=e^xf_k(x), h(x)=xe^xf_{k-1}(x)$.  So we wish to show that $f'(x)h(x)<f(x)h'(x)$.

\[f(x)=\sum_{i\geq k}\frac{x^i}{i!},\qquad f'(x)=\sum_{i\geq k}\frac{x^{i-1}}{(i-1)!}.\]

\[h(x)=\sum_{i\geq k}\frac{x^i}{(i-1)!},\qquad h'(x)=\sum_{i\geq k}\frac{ix^{i-1}}{(i-1)!}.\]

\[f(x)h'(x)=\sum_{i,j\geq k}\frac{ix^{i+j-1}}{(i-1)!j!}, \qquad
f'(x)h(x)=\sum_{i,j\geq k}\frac{x^{i+j-1}}{(i-1)!(j-1)!}.\]

When $i=j$, the contribution to each sum is the same:  $\frac{x^{2i-1}}{(i-1)!(i-1)!}$. When $i\neq j$, consider the contribution of $(i,j)$ plus the contribution of $(j,i)$.   The sum of these contributions to $f(x)h'(x)$ and $f'(x)h(x)$ is
\[x^{i+j-1}\left(\frac{i}{(i-1)!j!}+\frac{j}{i!(j-1)!}\right), \qquad x^{i+j-1}\frac{2}{(i-1)!(j-1)!}.\]
The contribution to $f(x)h'(x)$ is larger since $\frac{i}{j}+\frac{j}{i}>2$.
\end{proof}

\begin{proof}[Proof of Lemma~\ref{l2:monotone}] Let $h(\la)=e^{-\la}\la^{k-1}/f_{k-1}(\la)(k-1)!$.
Recall $g_k(x)$ above Lemma~\ref{l:gk} and recall the definition of $\psi(x)$ in~\eqn{h}. Then,
\[
\psi(x)=h(g_k^{-1}(x)),
\]
where $g_k^{-1}$ is the inverse of $g_k$. It is easy to see and we have mentioned before that $\lim_{x\to 0} g_k(x)=k$. So for every $x>k$, $g_k^{-1}(x)$ exists, and by Lemma~\ref{l:gk} and the chain rule, the derivative of $g_k^{-1}(x)$ is positive. Thus, in order to show that $\psi'(x)<0$, it is sufficient to show that for every $\la>0$, $h'(\la)<0$.

Showing that $h'(\la)<0$ is equivalent to showing that for every $\la>0$,
$$
(k-1)\sum_{j\ge k-1}\frac{\la^j}{j!}-\la\sum_{j\ge k-2}\frac{\la^j}{j!}=\frac{\la^{k-2}}{(k-2)!}\left(-\la+(k-1-\la)\sum_{j\ge k-1}\frac{\la^{j-k+2}}{[j]_{j-k+2}}\right)
$$
is negative. It is trivially true if $\la\ge k-1$. Now assume that $\la<k-1$. Since for every $j\ge k-1$, we have

$$
\frac{\la^{j-k+2}}{[j]_{j-k+2}}\le\left(\frac{\la}{k-1}\right)^{j-k+2},
$$
where the inequality is strict except for $j=k-1$. Thus,
\bean
-\la+(k-1-\la)\sum_{j\ge k-1}\frac{\la^{j-k+2}}{[j]_{j-k+2}}&<&-\la+(k-1-\la)\sum_{j\ge 0}\left(\frac{\la}{k-1}\right)^{j+1}\\
&=&-\la+(k-1-\la)\cdot\frac{\la}{k-1}\cdot\frac{1}{1-\frac{\la}{k-1}}=0.
\eean
This completes the proof that $h'(\la)<0$.
\end{proof}

We need the following technical lemma before proving Lemma~\ref{l:rho}.
\begin{lemma}\lab{l40}  For all $k,\dd\ge 2$ with $(r,k)\neq (2,2)$, $h(x)<1/(r-1)$ for all $x\geq x^*$ where
\[
h(x)=\frac{e^{-x} x^{k-1}}{f_{k-1}(x)(k-2)!}, \ \ \mbox{and}\ \ x^*=r(k-1)-\frac{r}{r-1}.
\]
\end{lemma}
\begin{proof}
It is easy to see that $h(x)$ is a decreasing function on $x>0$ (dividing $e^{-x} x^{k-1}$ from both the numerator and the denominator and then the numerator is constant whereas the denominator is an increasing function of $x$). Hence if we can prove $h(x)<1/(r-1)$ for some $x\le x^*$ then we are done.

We first prove a well-known inequality with respect to $f_k(\mu)$. 
\begin{claim}\lab{fkinq}
For any integer $k<\lfloor \mu \rfloor$, $f_k(\mu)>1/2$.
\end{claim}

\begin{proof}
Let $M$ be the maximum integer such that $M\le \mu-1$. Then we must have $k-1\le M-1$. By the definition of $f_k(\mu)$, we only need to prove that
\[
\sum_{i=0}^{k-1} e^{-\mu} \frac{\mu^i}{i!}\le \sum_{i=0}^{M-1} e^{-\mu} \frac{\mu^i}{i!}<1/2. 
\]
Let $p(i)=e^{-\mu} \mu^i/i!$; it suffices then to show that for every $1\le i\le M$, $p(M-i)\le p(M+i)$. This follows easily from
\[
\frac{p(M-i)}{p(M+i)}=\prod_{j=-i+1}^{i}\frac{M+j}{\mu}\le \prod_{j=-i+1}^{i}\left(1+\frac{-1+j}{\mu}\right)\le \exp\left(\sum_{j=-i}^{i-1}\frac{j}{\mu}\right)=\exp\left(-\frac{i}{\mu}\right)<1.\qedhere
\]
\end{proof}

 Now we continue to prove Lemma~\ref{l40}.
   We split our discussion into two cases. In both, we will reduce the lemma to checking a finite number of pairs $(k,r)$, which is straightforward.

{\em Case 1: $k=2$, $r\ge 3$.} Now $h(x)=e^{-x}x/(1-e^{-x})$ and $x^*=r-r/(r-1)$. By computing the derivative of $(r-1)\cdot h(x^*(r))$ with respect to $r$, it is easy to see that $h(x)$ is a decreasing function on $r\ge 3$. Hence it suffices to verify that $h(x^*(3))<1/2$ which can be easily done. 

{\em Case 2: $k\ge 3$.} We may easily verify the lemma for $(k,r)=(3,2)$. So we assume that $(k,r)\neq(3,2)$. \fix{Let $x_1=r(k-2)$. Clearly $x_1\le x^*$ and $\lfloor x_1\rfloor>k-1$. By Claim~\ref{fkinq}, we have $f_{k-1}(x_1)>1/2$. Since $h(x)$ is decreasing on $x>0$, it suffices to prove that $h(x_1)<1/(r-1)$ and thus it suffices to prove that}
$e^{-x_1} x_1^{k-1}/(k-2)!<\fix{\frac{1}{2(r-1)}}$.
Define
\[
\phi(x,r)=\frac{e^{-rx} (rx)^{x+1}}{x!}.
\]
Then, $e^{-x_1} x_1^{k-1}/(k-2)!=\phi(k-2,r)$ and so it suffices to prove 
\be
(r-1)\phi(k-2,r)<\hf.\lab{phi}
\ee
By computing the derivative of $\fix{(r-1)\cdot}\phi(x,r)$ with respect to $r$, we see that 
\begin{equation}\label{ephixr}
\mbox{for } \fix{(r-1)^2\ge r^2/(x+1)}: \qquad (r-1)\cdot \phi(x,r) \mbox{ is \fix{non-increasing} on } r.
\end{equation}

It is easy to see that
\[
\frac{\phi(x+1,r)}{\phi(x,r)}=\frac{e^{-r}(rx+r)^{x+2}}{(x+1)(rx)^{x+1}}= r e^{-r}\left(1+\frac{1}{x}\right)^{x+1}\le re^{-r+1+1/x}\le re^{-r+2},
\]
which is strictly less than one for any $x\ge 4$ and $r\ge 2$. So
\begin{equation}\label{ephixr2}
\mbox{for }x\ge 4, r\ge 2: \quad (r-1)\cdot \phi(x,r) \mbox{ is decreasing on } x.
\end{equation}

{\bf $k\geq 6$:}
We can easily verify that $\phi(4,2)<1/2$, and so by~(\ref{ephixr}),  this establishes (\ref{phi}) for $k=6,r\geq 2$. Then~(\ref{ephixr2}) establishes (\ref{phi}) for all $k\geq {6},r\geq 2$.

{\bf $k= 5$:}
We can easily verify that $2\phi(3,3)<1/2$, and so by~(\ref{ephixr}),  this establishes (\ref{phi}) for $k=5,\fix{r\geq 2}$. 

{\bf $k= 4$:}
We can easily verify that $3\phi({2},4)<1/2$, and so by~(\ref{ephixr}),  this establishes (\ref{phi}) for $k=4,\fix{r\geq 3}$. For $r=2$, we verify the lemma directly from the definitions of $h(x)$ and $x^*$.

{\bf $k= 3$:}
We can easily verify that $5\phi({1},6)<1/2$, and so by~(\ref{ephixr}),  this establishes (\ref{phi}) for $k=3,\fix{r\geq 4}$. For $r\in\{2,3\}$, we verify the lemma directly from the definitions of $h(x)$ and $x^*$.

\remove{
{\bf $k\geq 6$:}
We can easily verify that $\phi(4,2)<1/2$, and so by~(\ref{ephixr}),  this establishes (\ref{phi}) for $k=6,r\geq 2$. Then~(\ref{ephixr2}) establishes (\ref{phi}) for all $k\geq {6},r\geq 2$.

{\bf $k= 5$:}
We can easily verify that $2\phi(3,3)<1/2$, and so by~(\ref{ephixr}),  this establishes (\ref{phi}) for $k=5,r\geq 3$. For $(k,r)=(5,2)$, we verify the lemma directly from the definitions of $h(x)$ and $x^*$.

{\bf $k= 4$:}
We can easily verify that $3\phi({2},4)<1/2$, and so by~(\ref{ephixr}),  this establishes (\ref{phi}) for $k=4,r\geq 4$. For $r\in\{2,3\}$, we verify the lemma directly from the definitions of $h(x)$ and $x^*$.

{\bf $k= 3$:}
We can easily verify that $5\phi({1},6)<1/2$, and so by~(\ref{ephixr}),  this establishes (\ref{phi}) for $k=3,r\geq 6$. For $r\in\{2,3,4,5\}$, we verify the lemma directly from the definitions of $h(x)$ and $x^*$.
}
\end{proof}

\begin{proof}[Proof of Lemma~\ref{l:rho}.]
We first verify that $\zeta>k$. Recall that $g_k(x)=xf_{k-1}(x)/f_k(x)$. By definition of $\zeta$, $\alpha$ and $\beta$ in~(\ref{zeta}) and~\eqn{alpha-beta}, 
\[\zeta=g_k(\mu_{r,k}).\]  	It is easy to see that $\lim_{x\to 0} g_k(x)=k$. Lemma~\ref{l:gk} says that $g_k(x)$ is strictly increasing on $x>0$. {It is easy to see that for any $(r,k)\neq (2,2)$, $h_{r,k}(\mu)$ in~\eqn{e.hrk} tends to infinity both when $\mu\to 0$ and when $\mu\to\infty$. This implies that $\mu_{r,k}$ is a positive real number.} Thus, $g_k(\mu_{r,k})>k$.

 Next we prove that $\z<r(k-1)$.      
By~\eqn{critical},
\be
e^{-\mu_{\dd,k}}\frac{\mu_{\dd,k}^{k-1}}{(k-1)!f_{k-1}(\mu_{\dd,k})}=\frac{1}{(k-1)(\dd-1)}.\lab{critical2}
\ee
As we said above, $\z=g_k(\mu_{r,k})$ and so by~(\ref{egmrk}), $\z=\mu_{r,k}f_{k-1}(\mu_{r,k})/f_k(\mu_{r,k})$. Hence,
$\z<r(k-1)$ is equivalent to 
\[
\frac{f_k(\mu_{r,k})}{\mu_{r,k}f_{k-1}(\mu_{r,k})}>\frac{1}{r(k-1)}.
\]
 Note that $f_k(\mu_{r,k})=f_{k-1}(\mu_{r,k})-e^{-\mu_{r,k}}\mu_{r,k}^{k-1}/(k-1)!$; combined with~\eqn{critical2} the left hand side above is equal to
 \[
 \frac{1}{\mu_{r,k}}\left(1-e^{-\mu_{r,k}}\frac{\mu_{r,k}^{k-1}}{(k-1)!f_{k-1}(\mu_{r,k})}\right)=\frac{1}{\mu_{r,k}}\left(1-\frac{1}{(k-1)(r-1)}\right).
 \]
  Hence $\zeta<r(k-1)$ is equivalent to 
  \be\label{eqqq}
  \frac{1}{\mu_{r,k}}\left(1-\frac{1}{(k-1)(r-1)}\right)>\frac{1}{r(k-1)} \quad i.e.\ \mu_{r,k}<\mu^*:=r(k-1)-\frac{r}{r-1}.
  \ee

  We consider the function $h(x)$ from Lemma~\ref{l40}; i.e.\ 
  \[
  h(x)=\frac{e^{-x}x^{k-1}/(k-2)!}{f_{k-1}(x)}.
  \]
  By~\eqn{critical2} (and multiplying $k-1$ on both sides), $h(\mu_{r,k})=1/(r-1)$. So by  Lemma~\ref{l40},  $\mu_{r,k}<\mu^*$, thus establishing~(\ref{eqqq}) and hence the lemma.
\end{proof}


\begin{thebibliography}{99}



\bibitem{amxor} D. Achlioptas and M. Molloy. {\em The solution space geometry of random linear equations.}  Random Structures and Algorithms, 46(2): 197--231, (2015).

\bibitem{AFP} J. Aronson, A. Frieze and B. G. Pittel. {\em Maximum matchings in sparse random graphs: Karp-Sipser re-visited.} Random Structures and Algorithms 12(2): 111--177, (1998).


\bibitem{azuma} K. Azuma. {\em Weighted sums of certain dependent random variables.} Tokuku Math.\ J.\ {\bf 19} (1967), 357~-~367.

\bibitem{bb} B. Bollob\'{a}s. {\em A probabilistic proof of an asymptotic formula for the number of labelled regular graphs.} Europ. J. Combinatorics {\bf 1} {311--316} (1980).

\bibitem{bbcore} B. Bollob\'{a}s. {\em The evolution of sparse graphs.} Graph Theory and Combinatorics. Proc. Cambridge Combinatorial Conf. in honour of Paul Erd{\H ̈o}s (B. Bollob ́as, ed.), Academic Press (1984), {35--57}.


\bibitem{CW} J. Cain and N. Wormald, Encores on cores,
{\em Electron. J. Combin.}, 13 (2006), no. 1, Research Paper 81, 13 pp.

\bibitem{csw} J. Cain, P. Sanders and N. Wormald. {\em The random graph threshold for $k$-orientiability and a fast algorithm for optimal multiple-choice allocation.} Proceedings of 18th  SODA (2007), {469--476}.




\bibitem{cc} C. Cooper. {\em The cores of random hypergraphs with a given degree sequence.}
 Random Structures Algorithms {\bf 25} (2004), {353--375}.

\bibitem{dmcore} A. Dembo and A. Montanari. {\em Finite size scaling for the core of large random hypergraphs.} Annals of Applied Probability, {\bf 18} (2008), {1993--2040}.



\bibitem{DKLP} J. Ding,  J.H. Kim,  E. Lubetzky and Y. Peres.
{\em Anatomy of a young giant component in the random graph.}
Random Structures Algorithms {\bf 39} (2011), {139--178}.


\bibitem{dklp2} J. Ding,  J.H. Kim,  E. Lubetzky and Y. Peres.
{\em Diameters in supercritical random graphs via first passage percolation.} Combinatorics, Probability and Computing {\bf 19} (2010), {729--751}.

\bibitem{fern} D. Fernholz and V. Ramachandran. {\em The $k$-orientability thresholds for $G_{n,p}$.} Proceedings of SODA 2007, {459--468}.

\bibitem{fkp} N. Fountoulakis,  M. Koshla and K. Panagiotou. {\em The multiple-orientability thresholds of random hypergraphs.}  Proceedings of  SODA 2011, {1222--1236}.

\bibitem{frmix} N. Fountoulakis and B. Reed. {\em The evolution of the mixing rate of a simple random walk on the giant component of a random graph.} Random Structures and Algorithms {\bf 33} (2008), {68--86}.




\bibitem{g2} P. Gao, Analysis of the parallel peeling algorithm: a short proof, arXiv:1402.7326.

\bibitem{gmxor} P. Gao and M. Molloy, in preparation.

\bibitem{gmarxiv} P. Gao and M. Molloy, Inside the clustering threshold for random linear equations, arXiv:1309:6651.


\bibitem{pw} P. Gao and N. Wormald. {\em Load balancing and orientability thresholds for random hypergraphs.} Proceedings of STOC (ACM) (2010), {97--104}.


\bibitem{hmwc} G. Havas, B.S. Majewski, N.C. Wormald, and Z.J. Czech. {\em Graphs, hypergraphs and hashing.} In 19th International Workshop on Graph-Theoretic Concepts in Computer Science (WG’93), Lecture Notes in Computer Science  {\bf 790} (1993) {153--165}.

\bibitem{hoeffding} W. Hoeffding, Probability inequalities for sums of bounded random variables, {\em Journal of the American Statistical Association} 58(301), 1963.

\bibitem{ikkm} M. Ibrahimi, Y. Kanoria, M. Kraning, and A. Montanari. {\em The set of solutions of random xorsat formulae.}
Proceedings of the Twenty-Third Annual ACM-SIAM Symposium
               on Discrete Algorithms, SODA 2012, Kyoto, Japan, January
               17-19, 2012, pages 760--779. SIAM, 2012.




\bibitem{JL}
S. Janson and M. J. Luczak. {\em Asymptotic normality of the k-core in random graphs}. The annals of applied probability, 18(3): 1085-1137,  (2008).

\bibitem{jmt} J. Jiang, M. Mitzenmacher and J. Thaler, Parallel Peeling Algorithms, arXiv:1302.7014.



\bibitem{jhk} J.H.Kim.  {\em Poisson cloning model for random graphs.}  	arXiv:0805.4133v1



\bibitem{lmss} M. Luby, M. Mitzenmacher, A. Shokrollahi, and D. Spielman.
{\em Efficient Erasure Correcting Codes.}
IEEE Transactions on Information Theory, {\bf 47} (2001), {569--584}.

\bibitem{lmss2} M. Luby, M. Mitzenmacher, A. Shokrollahi, and D. Spielman
{\em Improved Low-Density Parity-Check Codes Using Irregular Graphs.}
IEEE Transactions on Information Theory, {\bf 47} (2001), {585--598}.

\bibitem{tlcomp}  T. {\L}uczak.  {\em Component behaviour near the critical point of the random graph process.}
Rand. Struc. \& Alg. {\bf 1} (1990), {287--310}.

\bibitem{tlcomp2}  T. {\L}uczak.  {\em Random trees and random graphs.} Rand. Struc. \& Alg. {\bf 13} (1998), {485--500}.





\bibitem{cm}  C. McDiarmid. {\em Concentration for independent permutations.}
Combinatorics, Probability and Computing {\bf 11} (2002), {163--178}.

\bibitem{mmcore} M. Molloy {\em Cores in random hypergraphs and boolean formulas.}
Random Structures and Algorithms {\bf 27}, {124--135} (2005).

\bibitem{mr2} M. Molloy and B. Reed. {\em Critical Subgraphs of a Random Graph.} Electronic J. Comb. {\bf 6}  (1999),  R35.

\bibitem{mr1} M. Molloy and B. Reed. {\em A critical point for random graphs with a given degree sequence.}
Random Structures and Algorithms {\bf 6} {161--180} (1995).

\bibitem{mrbook}  M. Molloy and B. Reed.  {\em Graph Colouring and the
Probabilistic Method}. Springer (2002).


\bibitem{psw} B. Pittel, J. Spencer and N. Wormald.
{\em Sudden emergence of a giant $k$-core in a random graph.}
J. Comb. Th. B {\bf 67}, {111--151} (1996).

\bibitem{rw} O. Riordan and N. C. Wormald. {\em The diameter of sparse random graphs.}
Combinatorics, Probability and Computing {\bf 19} (2010), {835--926}.


\bibitem{s} C. M. Sato, On the robustness of random k-cores, arXiv:1203.2209.



\bibitem{mt} M. Talagrand.  {\em Concentration of measure and isoperimetric
inequalities in product spaces.} Institut Des Hautes \'{E}tudes Scientifiques,
Publications Math\'{e}matiques {\bf 81} (1995), 73~-~205.


\end{thebibliography}
\end{document}